\documentclass{article}

\usepackage[utf8]{inputenc}
\usepackage{hyperref}
\usepackage{amsfonts}
\usepackage{enumerate}
\usepackage{amsmath,amsthm,amssymb}
\usepackage{geometry}
\usepackage{graphicx}
\usepackage[english]{babel}
\usepackage{fancyhdr}
\usepackage{nameref} 
\usepackage{mathtools}
\usepackage[mathscr]{eucal}
\usepackage{latexsym}
\usepackage{tikz-cd}
\usetikzlibrary{decorations.pathmorphing}
\usepackage{dsfont}
\usepackage{bbm}
\usepackage{etoolbox}
\usepackage{csquotes}
\usepackage{xurl}
\usepackage[noabbrev]{cleveref}
\usepackage{anyfontsize}

\usepackage{amscd}
\usepackage[only,mapsfrom]{stmaryrd}

\usepackage{pgf,tikz,pgfplots}
\pgfplotsset{compat=1.15}
\usepackage{mathrsfs}
\usetikzlibrary{arrows}
\usepackage[bbgreekl]{mathbbol}

\DeclareMathOperator{\Hom}{Hom}
\DeclareMathOperator{\Cont}{Cont}

\DeclareMathOperator{\Ext}{Ext}
\DeclareMathOperator{\Gal}{Gal}

\newcommand{\N}{\ensuremath{\mathbbm{N}}}
\newcommand{\Q}{\ensuremath{\mathbbm{Q}}}
\newcommand{\R}{\ensuremath{\mathbbm{R}}}
\newcommand{\Z}{\ensuremath{\mathbbm{Z}}}
\newcommand{\F}{\ensuremath{\mathbbm{F}}}

\newcommand{\TT}{\ensuremath{\mathscr{T}}}
\newcommand{\CC}{\ensuremath{\mathscr{C}}}

\newcommand{\MM}{\ensuremath{\mathcal{M}}}

\newcommand{\DD}{\ensuremath{\mathcal{D}}}
\newcommand{\DDD}{\ensuremath{\mathbf{D}}}

\renewcommand{\H}{\mathrm{H}}
\renewcommand{\O}{\ensuremath{\mathcal{O}}}

\newcommand{\HH}{\ensuremath{\mathbbm{H}}}
\newcommand{\GGamma}{\ensuremath{\mathbb{\Gamma}}}
\newcommand{\LH}{\ensuremath{\mathcal{LH}}}
\newcommand{\RH}{\ensuremath{\mathcal{RH}}}
\newcommand{\NH}{\ensuremath{\mathcal{NH}}}
\newcommand{\Sch}{\ensuremath{\mathbf{Sch}}}
\newcommand{\p}{\ensuremath{\mathfrak{p}}}

\newcommand{\Set}{\mathbf{Set}}

\newcommand{\Grp}{\mathbf{Grp}}
\newcommand{\Ab}{\mathbf{Ab}}
\newcommand{\Top}{\mathbf{Top}}

\newcommand{\LCA}{\mathbf{LCA}}
\newcommand{\FLCA}{\mathbf{FLCA}}

\makeatletter
\newtheorem*{rep@theorem}{\rep@title}
\newcommand{\newreptheorem}[2]{%
\newenvironment{rep#1}[1]{%
 \def\rep@title{#2 \ref{##1}}%
 \begin{rep@theorem}}%
 {\end{rep@theorem}}}
\makeatother

\theoremstyle{plain}
\newtheorem{thm}{Theorem}[section]
\newreptheorem{thm}{Theorem}
\newtheorem{cor}[thm]{Corollary}
\newreptheorem{cor}{Corollary}
\newtheorem{lem}[thm]{Lemma}
\newreptheorem{lem}{Lemma}
\newtheorem{prop}[thm]{Proposition}
\newreptheorem{prop}{Proposition}
\newtheorem*{thm*}{Theorem}
\newtheorem*{lem*}{Lemma}
\newtheorem*{prop*}{Proposition}
\newtheorem{thmx}{Theorem}

\theoremstyle{definition}
\newtheorem{defn}[thm]{Definition}
\newtheorem{ex}[thm]{Example}
\newtheorem{rmk}[thm]{Remark}

\newcommand{\leftin}[2] {\prescript{}{#1}{#2}}

 \crefname{cor}{Corollary}{Corollaries}
\crefname{thm}{Theorem}{Theorems}
\crefname{prop}{Proposition}{Propositions}
\crefname{lem}{Lemma}{Lemmas}
\crefname{rmk}{Remark}{Remarks}
\crefname{cns}{Construction}{Constructions}
\crefname{ntt}{Notation}{Notation}
\crefname{ex}{Example}{Examples}
\crefname{defn}{Definition}{Definitions}
\crefname{section}{Section}{Sections}

\numberwithin{equation}{section}

\title{Duality for the condensed Weil-étale realisation of $1$-motives over $p$-adic fields}

\author{Marco Artusa\text{, with an appendix written by Takashi Suzuki}}
\date{}

\begin{document}

\maketitle

\begin{abstract}
    We extend Tate duality for Galois cohomology of abelian varieties to $1$-motives over a $p$-adic field, improving a result of Harari and Szamuely. To do this, we replace Galois cohomology with the condensed cohomology of the Weil group. This is a topological cohomology theory defined in a previous work, which keeps track of the topology of the $p$-adic field. To see $1$-motives as coefficients of this cohomology theory, we introduce their condensed Weil-étale realisation. Our duality takes the form of a Pontryagin duality between locally compact abelian groups.
\end{abstract}

\setcounter{tocdepth}{2}
\tableofcontents
\section{Introduction}
Duality theorems are among the central statements in arithmetic geometry. For a $p$-adic field $F$, the earliest example is due to Tate \cite{WC}, who proves a duality theorem for Galois cohomology of abelian varieties. His result can be stated as follows \begin{thmx}\label{intro:thmtate}
Let $A$ be an abelian variety over a $p$-adic field $F$ and let $A^*$ be its dual abelian variety. For $q=0,1$, there is a perfect pairing \[
\H^q(G_F,A(\overline{F}))\times \H^{1-q}(G_F,A^*(\overline{F}))\rightarrow \H^2(G_F,\overline{F}^{\times})=\Q/\Z
\] between a profinite abelian group and a torsion abelian group.
\end{thmx}
The goal of this paper is to extend this result to $1$-motives replacing Galois cohomology with the condensed cohomology of the Weil group, developed in \cite{Artusa}. Indeed, such a generalisation in terms of Galois cohomology is not satisfying enough. The reason is that completion procedures are needed on cohomology groups to make the paring perfect (see for example \cite[Corollary 2.3]{ADT} and \cite[Theorem 0.1]{HarariSzamuely}), which lose information on the original groups. In our new framework, the duality for $1$-motives is expressed as a Pontryagin duality between locally compact abelian groups and no completion procedure is needed.

Let $F$ be a $p$-adic field. Certain shortcomings of Galois cohomology for duality theorems have been observed in \cite{Artusa}, where a new cohomology theory is provided by replacing the Galois group $G_F$ with the Weil group $W_F$ and the category of sets with the category $\CC$ of condensed sets. The resulting cohomology groups $\HH^q(B_{\hat{W}_F},-)$ are condensed abelian groups, and the category of coefficients, denoted by $\Ab(B_{\hat{W}_F})$, is the category of condensed abelian groups with an action of the pro-condensed group $\hat{W}_F$. Since $\CC$ contains compactly generated topological spaces as a full subcategory stable by all limits, both the coefficients and the cohomology groups are naturally topologised. 

In this paper, we associate to a $F$-commutative group scheme locally of finite type,  say $G$, a condensed abelian group with an action of $\hat{W}_F$. We informally present this construction. Let $L$ be the completion of the maximal unramified extension $F^{\mathrm{un}}$. Let us fix separable closures $\overline{F}\subset \overline{L}$ of $F$ and $L$ respectively. The Weil group $W_F$ acts continuously on the abelian group $G(\overline{L})$. We observe that $\overline{L}$ has a natural topology induced by the one of $L$, which comes from the non-archimedean topology on $F$. Then $G(\overline{L})$ can be seen as a topological abelian group with the topology induced by the one of $\overline{L}$. This operation gives $G(\overline{L})$ the structure of a condensed abelian group with an action of $\hat{W}_F$. We denote this object $G(\overline{L})_{\mathrm{cond}}$ and we call it \emph{condensed Weil-étale realisation} of $G$. We consider the cohomology groups $\HH^q(B_{\hat{W}_F},G(\overline{L})_{\mathrm{cond}})$, condensed abelian groups which keep track of the topology of $F$. The underlying abelian groups coincide with the cohomology groups $\H^q(W_F,G(\overline{L}))$ considered by Karpuk in \cite{Karpuk2}. 

By functoriality, we extend this construction to $1$-motives. A $1$-motive over $F$ is a complex of $F$-commutative group schemes concentrated in degrees $-1$ and $0$, say $\MM=[Y\to E]$, where $Y$ is étale-locally isomorphic to the constant group $\Z^r$ and $E$ is a semiabelian variety, i.e.\ an extension of an abelian variety by a torus. Following \cite{Deligne}, we associate to $\MM$ a Cartier dual $\MM^*=[Y^*\to E^*]$. We consider the condensed Weil-étale realisations $\MM(\overline{L})_{\mathrm{cond}},\MM^*(\overline{L})_{\mathrm{cond}}$ and we study the total derived cohomologies $R\GGamma(B_{\hat{W}_F},\MM(\overline{L})_{\mathrm{cond}}),R\GGamma(B_{\hat{W}_F},\MM^*(\overline{L})_{\mathrm{cond}})$, showing that they are represented by bounded complexes of locally compact abelian groups. We define the $\R/\Z$-twist of the condensed cohomology of $\MM^*$, and we denote it $R\GGamma(B_{\hat{W}_F},\MM^*(\overline{L})_{\mathrm{cond}};\R/\Z)$. Using a condensed version of biextensions, we build a cup-product pairing in $\DDD^{\mathrm{b}}(\CC)$ \begin{equation}\label{intro:cppairingmotives}
    R\GGamma(B_{\hat{W}_F},\MM(\overline{L})_{\mathrm{cond}})\otimes^L R\GGamma(B_{\hat{W}_F},\MM^*(\overline{L})_{\mathrm{cond}};\R/\Z)\rightarrow \HH^1(B_{\hat{W}_F},\mathbbm{G}_m(\overline{L})_{\mathrm{cond}};\R/\Z)=\R/\Z.
\end{equation} Our main result can be stated as follows \begin{repthm}{duality1motives2}
The cup-product pairing \eqref{intro:cppairingmotives} is perfect. Moreover, it induces perfect cup-product pairings of locally compact abelian groups \[\begin{split}
\HH^{-1}(B_{\hat{W}_F},\MM(\overline{L})_{\mathrm{cond}})\otimes \HH^{1}(B_{\hat{W}_F},\MM^*(\overline{L})_{\mathrm{cond}};\R/\Z)^{\mathrm{lca}}&\rightarrow \R/\Z, \\
\HH^{0}(B_{\hat{W}_F},\MM(\overline{L})_{\mathrm{cond}})^{\mathrm{lca}}\otimes \HH^{0}(B_{\hat{W}_F},\MM^*(\overline{L})_{\mathrm{cond}};\R/\Z)&\rightarrow \R/\Z,\\
\HH^{1}(B_{\hat{W}_F},\MM(\overline{L})_{\mathrm{cond}})\otimes \HH^{-1}(B_{\hat{W}_F},\MM^*(\overline{L})_{\mathrm{cond}};\R/\Z)&\rightarrow \R/\Z\\
\end{split}
\] and a perfect pairing \[
\HH^0(B_{\hat{W}_F},\MM(\overline{L})_{\mathrm{cond}})^{\mathrm{nh}}\otimes^L \HH^1(B_{\hat{W}_F},\MM^*(\overline{L})_{\mathrm{cond}};\R/\Z)^{\mathrm{nh}}[-1]\rightarrow \R/\Z.
\]
\end{repthm}
We observe the presence of 4 pieces contributing to the cohomology of $R\GGamma(B_{\hat{W}_F},\MM(\overline{L})_{\mathrm{cond}})$. Indeed, the condensed abelian group $\HH^0(B_{\hat{W}_F},\MM(\overline{L})_{\mathrm{cond}})$ is not locally compact and can be decomposed in a maximal locally compact quotient $\HH^0(B_{\hat{W}_F},\MM(\overline{L})_{\mathrm{cond}})^{\mathrm{lca}}$ and a non-Hausdorff subgroup $\HH^0(B_{\hat{W}_F},\MM(\overline{L})_{\mathrm{cond}})^{\mathrm{nh}}$. The same holds for $\HH^1(B_{\hat{W}_F},\MM(\overline{L})_{\mathrm{cond}};\R/\Z)$ and, more generally, for every object of the left heart of $\DDD^{\mathrm{b}}(\LCA)$, the bounded derived category of the quasi-abelian category $\LCA$ of locally compact abelian groups. 

If $\MM=[0\rightarrow A]$ is an abelian variety, the underlying abelian groups of $\HH^q(B_{\hat{W}_F},A(\overline{L})_{\mathrm{cond}})$ and $\HH^q(B_{\hat{W}_F},A^*(\overline{L})_{\mathrm{cond}};\R/\Z)$ coincide with $\H^q(G_F,A(\overline{F}))$ and $\H^{q-1}(G_F,A^*(\overline{F}))$ respectively. Thus our result is a topological improvement of Theorem \ref{intro:thmtate}. Indeed, Tate's cohomology groups are abelian groups, hence they are not naturally topologised, contrarily to our condensed cohomology groups. Nonetheless, Tate's duality for abelian varieties is a key ingredient to prove our result. If $\MM=[0\rightarrow T]$ is a torus, our result is a condensed-Weil version of Tate-Nakayama duality which does not need profinite completion, and is a consequence of \cite[Theorem 4.27]{Artusa}. In general, our result is an improved version of \cite[Theorem 0.1]{HarariSzamuely} in terms of the Weil group and of Condensed Mathematics.
\subsection{Relation with previous work}
Let us recall existing generalisations of Theorem \ref{intro:thmtate} and their shortcomings, starting from the setting of Galois cohomology. Let $T$ be a torus over $F$ and let $Y^*$ be its Cartier dual. Cartier duality induces a cup-product pairing \begin{equation}\label{intro:eqn:noncondensedpairingtori}
    \H^q(G_F,Y^*(\overline{F}))\times \H^{2-q}(G_F,T(\overline{F}))\rightarrow \H^2(G_F,\overline{F}^{\times})=\Q/\Z.
\end{equation}
To generalise Theorem \ref{intro:thmtate} to tori, one would expect this pairing to be perfect. However, this does not happen. Indeed, if $T=\mathbbm{G}_m$ and $q=0$, then \eqref{intro:eqn:noncondensedpairingtori} becomes the pairing \[
\Z\times \Q/\Z\rightarrow \Q/\Z,
\] which cannot be perfect since the induced map $\Z\rightarrow \Hom(\Q/\Z,\Q/\Z)=\hat{\Z}$ is not an isomorphism. However, we recover perfectness of the pairing if we replace the $0$th cohomology group with its profinite completion. This result is called Tate-Nakayama duality and can be stated as follows (see for example \cite[Corollary 2.3]{ADT})
\begin{thmx}\label{intro:thmtatenakayama}
The pairing \eqref{intro:eqn:noncondensedpairingtori} becomes perfect if in the case $q\neq 1$ we replace $\H^0$ with its profinite completion.
\end{thmx}
If $T=\mathbbm{G}_m$ and $q=2$ we recover the reciprocity isomorphism of local class field theory \[
(F^{\times})^{\wedge}\overset{\sim}{\rightarrow} G_F^{ab}.
\] This completion procedure loses information on the original groups. Moreover, since it is not an exact operation, it causes further issues when we want to generalise Theorems \ref{intro:thmtate} and \ref{intro:thmtatenakayama} to complexes of group schemes, e.g.\ $1$-motives. 

Let $\MM=[Y\rightarrow E]$ be a $1$-motive over $F$ and let $\MM^*$ be its dual $1$-motive. The Poincaré biextension gives a cup-product pairing on Galois cohomology \begin{equation}\label{intro:eqn:noncondensedpairingmotives}
R\Gamma(G_F,\MM(\overline{F}))\otimes^L R\Gamma(G_F,\MM^*(\overline{F}))\rightarrow \H^2(G_F,\overline{F}^{\times})[-1]=\Q/\Z[-1].
\end{equation}
Harari and Szamuely prove the following duality result (see \cite[Theorem 0.1]{HarariSzamuely})
\begin{thmx}\label{intro:thmhs}
The cup-product pairing \eqref{intro:eqn:noncondensedpairingmotives} induces perfect pairings \[
\H^{-1}(G_F,\MM(\overline{F}))_{\wedge}\otimes \H^2(G_F,\MM^*(\overline{F}))\rightarrow \Q/\Z
\] and \[
\H^0(G_F,\MM(\overline{F}))^{\wedge}\otimes \H^1(G_F,\MM^*(\overline{F}))\rightarrow \Q/\Z.
\] Here $\H^0(G_F,\MM(\overline{F}))^{\wedge}$ is the profinite completion of $\H^0(G_F,\MM(\overline{F}))$, while $\H^{-1}(G_F,\MM(\overline{F}))_{\wedge}$ is the kernel of $Y(F)^{\wedge}\rightarrow E(F)^{\wedge}$.
\end{thmx} The need of two different completion procedures $(-)^{\wedge}$ and $(-)_{\wedge}$ comes from the lack of exactness of profinite completion that we mentioned before. Generalisations of Theorem \ref{intro:thmtate} should not involve such completion procedures.

\vspace{0.5em}
In \cite{Karpuk2}, Karpuk tries to solve this issue by using the Weil group $W_F$ instead of the Galois group, following Lichtenbaum's intuition \cite{Licht}. If $G$ is a commutative group scheme over $F$, and $L$ is the completion of the maximal unramified extension of $F$, we have a natural action of $W_F$ on $G(\overline{L})$ and we can consider the cohomology $R\Gamma(W_F,G(\overline{L}))$. If $T$ is a torus over $F$ with Cartier dual $Y^*$, we obtain a pairing \begin{equation}\label{intro:eqn:noncondensedweilpairingtori}
    R\Gamma(W_F,Y^*(\overline{L}))\otimes^L R\Gamma(W_F,T(\overline{L}))\rightarrow \H^1(W_F,\overline{L}^{\times})[-1]=\Z[-1].
\end{equation}Karpuk's version of Theorem \ref{intro:thmtatenakayama} using the Weil group can be stated as follows (see \cite[Theorem 3.3.1]{Karpuk}) 
\begin{thmx}\label{intro:thmkarpuktori}
The cup-product pairing \eqref{intro:eqn:noncondensedweilpairingtori} induces an equivalence in $\DDD^{\mathrm{b}}(\Ab)$
\[
R\Gamma(W_F,T(\overline{L}))\overset{\sim}{\rightarrow} R\Hom_{\Ab}(R\Gamma(W_F,Y^*(\overline{L})),\Z[-1]).
\]
\end{thmx}
This is a duality theorem for tori which does not need profinite completion to hold. It is a significant improvement, but it still encounter certain limitations. First of all, it is a duality theorem which is expressed only at the level of derived categories. Since cohomology groups of $R\Gamma(W_F,Y^*(\overline{L}))$ may have non-trivial torsion, this result does not directly induce a duality between cohomology groups. Moreover, the other map induced by \eqref{intro:eqn:noncondensedweilpairingtori} \[
R\Gamma(W_F,Y^*(\overline{L}))\rightarrow R\Hom(R\Gamma(W_F,T(\overline{L})),\Z[-1])
\] is not an equivalence in general (see \cite[Proposition 3.3.8]{Karpuk}). Thus the pairing \eqref{intro:eqn:noncondensedweilpairingtori} is not perfect. In the same spirit, Karpuk's version of Theorem \ref{intro:thmtate} does not recover a perfect pairing on the cohomology of the Weil group. Indeed, if $A$ is an abelian variety over $F$ and $A^*$ is its dual abelian variety, we have a cup-product pairing \[
R\Gamma(W_F,A(\overline{L}))\otimes^L R\Gamma(W_F,A^*(\overline{L}))\rightarrow \H^1(W_F,\overline{L}^{\times})=\Z
\] and an induced morphism \[
\tau(A):R\Gamma(W_F,A^*(\overline{L}))\rightarrow R\Hom(R\Gamma(W_F,A(\overline{L})),\Z).
\] Karpuk's result is stated as follows (see \cite[Theorem 5.4.4]{Karpuk2})
\begin{thmx}\label{intro:thmkarpuk2}
The map $\tau(A)$ has the following properties:
\begin{enumerate}[(i)]
    \item $\tau^0(A):A^*(F)\rightarrow \Ext(\H^1(W_F,A(\overline{L})),\Z)=\Ext(\H^1(W_F,A(\overline{L})),\Q/\Z)$ is an isomorphism of profinite groups.
    \item $\tau^1(A):\H^1(W_F,A^*(\overline{L}))\rightarrow \Ext(A(F),\Z)$ induces an isomorphism of $\H^1(W_F,A^*(\overline{L}))$ with the torsion subgroup $\Hom(A(F),\Q/\Z)$ of  $\Ext(A(F),\Z)$.
\end{enumerate}
\end{thmx} Even though we are still able to recover Tate's duality for abelian varieties, the cup-product pairing inducing $\tau(A)$ is not perfect. 

Let us point out a deeper explanation of why Karpuk's improvements are not optimal. The reason why the profinite completion appears in duality theorems for Galois cohomology is the following. Such results are expressed in terms of \emph{discrete} Pontryagin duality \[
\Hom_{\Ab}(-,\Q/\Z):\Ab^{\mathrm{op}}\rightarrow \Ab,
\] which realises an antiequivalence of categories between profinite abelian groups and torsion abelian groups, but it is \emph{not} an antiequivalence of categories on the whole category $\Ab$. To make the cup-product pairings perfect one needs to work with profinite and torsion abelian groups and this is why profinite completions appear in Theorems \ref{intro:thmtatenakayama} and \ref{intro:thmhs}. In Karpuk's setting of the cohomology of $W_F$, the results are expressed in terms of $\Z$-linear duality \[
R\Hom_{\Ab}(-,\Z):\DDD^{\mathrm{b}}(\Ab)^{\mathrm{op}}\rightarrow \DDD^{\mathrm{b}}(\Ab),
\] which is not an antiequivalence of categories either. The pairings involved in Theorems \ref{intro:thmkarpuktori} and \ref{intro:thmkarpuk2} are not perfect since the complexes $R\Gamma(W_F,A(\overline{L}))$ and $R\Gamma(W_F,Y^*(\overline{L}))$ are not dualisable. To fix this, we consider the \emph{topological} Pontryagin duality \[
R\underline{\Hom}_{\LCA}(-,\R/\Z):\DDD^{\mathrm{b}}(\LCA)^{\mathrm{op}}\overset{\sim}{\rightarrow} \DDD^{\mathrm{b}}(\LCA)
\] which is an antiequivalence of categories. This suggests that to have satisfying duality theorems, our cohomologies should be naturally topologised, as remarked by Geisser and Morin in \cite{GeisserMor2}. Condensed Mathematics \cite{LCM,BH} provide a topos $\CC$ of condensed sets which contains compactly generated abelian groups as a full subcategory stable by all limits. Moreover, the bounded $\infty$-derived category of condensed abelian groups $\DDD^b(\CC)$ contains $\DDD^b(\LCA)$ as a full stable $\infty$-subcategory (see \cite[Lemma 4.2]{Artusa}). To express our duality theorems as a topological Pontryagin duality, we want a cohomology theory which takes values in $\DDD^b(\CC)$.

\vspace{0.5em}
After the first appearence of this article on the web, it was drawn to the attention of the author that Suzuki proved a duality theorem for $1$-motives in this spirit in \cite{Suz}. Indeed, using the so-called ind-rational pro-étale site of the residue field $k$, denoted by $\mathrm{Spec}(k^{\mathrm{indrat}}_{\mathrm{pro\acute{e}t}})$, he builds a cohomology theory $R\Gamma(F_W,-)$ with values in $\DDD^b(\overline{k}_{\mathrm{pro\acute{e}t}})=\DDD^b(\CC)$, having complexes of fppf sheaves as coefficients. For a $1$-motive $\MM$, the underlying complex of abelian groups $R\Gamma(F_W,\MM)(*)$ coincides with the cohomology of the Weil group. He then proves the following (see \cite[Section 10]{Suz}) 
\begin{thmx}\label{intro:thm:suz}
Let $\MM$ be a $1$-motive over $F$ with Cartier dual $\MM^*$. We have a perfect cup-product pairing in $\DDD^b(\CC)$ \[R\Gamma(F_W,\MM)\otimes^L R\Gamma(F_W,\MM^*)\longrightarrow \mathbf{\H}^1(F_W,\mathbbm{G}_m)=\Z.
\] 
\end{thmx}
This result is a condensed (and $\Z$-linear) version of Theorem \ref{intro:thmhs} and does not need profinite completion to hold. In this sense, it is certainly an improvement of both Theorem \ref{intro:thmhs} and of Theorems \ref{intro:thmkarpuktori} and \ref{intro:thmkarpuk2}. Moreover, it has the great merit of being stated as a condensed duality before the formalisation of the theory by Clausen-Scholze and Barwick-Haine (see the original version \cite{Suz0}). However, contrarily to Harari and Szamuely's result, it is a duality only on the derived category level. Similarly to Theorem \ref{intro:thmkarpuktori}, a concrete expression of the duality on each cohomology object is missing.

\vspace{0.5em}

In this paper, we use a different and independent approach which builds on the results of \cite{Artusa}, where the need for a topological cohomology theory is addressed by using the topos $B_{\hat{W}_F}$, the classifying topos of $W_F$ seen as a pro-object of $\Grp(\CC)$. Our cohomology theory is denoted by $R\GGamma(B_{\hat{W}_F},-)$ and takes values in $\DDD^{\mathrm{b}}(\CC)$. In this way we solve the issues of the classical Tate approach (i.e.\ the profinite completion) and of Karpuk's approach (i.e.\ the lack of topology). The analogue of Theorem \ref{intro:thm:suz} can still be recovered in our setting (see \cref{thm:duality1motives1}). Moreover, we carry out a deeper study of the structure of each cohomology group (\cref{section:structurecohomologymotives}) and of the Pontryagin duality in $\DDD^b(\LCA)$ (\cref{decompositionsection}). Thanks to this, we prove our main result (\cref{duality1motives2}), which is a condensed duality for the cohomology of $1$-motives expressed as a Pontryagin duality on each cohomology group. Thus, it contains Theorem \ref{intro:thmtate} and improves simultaneously Theorems \ref{intro:thmtatenakayama}, \ref{intro:thmhs}, \ref{intro:thmkarpuktori}, \ref{intro:thmkarpuk2} and \ref{intro:thm:suz}. 

\vspace{0.5em}

We conclude this section by saying a few words on the relation between the results of this work and the ones of \cite{Artusa}. Both \cref{duality1motives2} and \cite[Theorem 4.27]{Artusa} extend Tate's local duality to more general coefficients.
\begin{thmx}[Tate local duality]\label{intro:tatelocalduality}
Let $N$ be a finite abelian group with a continuous action of $G_F$. Then we have a perfect cup-product pairing \[
\H^q(G_F,N)\otimes \H^{2-q}(G_F,\Hom(N,\overline{F}^{\times}))\rightarrow \H^2(G_F,\overline{F}^{\times})=\Q/\Z.
\]
\end{thmx}
The difference is that \cite[Theorem 4.27]{Artusa} extends such coefficients in a \emph{topological} direction: from finite discrete abelian groups to some locally compact abelian groups including $\Z^n$ with the discrete topology, countable abelian groups with the discrete topology, $\R^n$ with the Euclidean topology, $(\R/\Z)^n$ with its compact Hausdorff topology. \cref{duality1motives2} extends the coefficients in a \emph{algebro-geometric} direction: from finite group schemes to abelian varieties, tori, and more generally $1$-motives. To do this, the new tool of condensed Weil-étale realisation is fundamental, allowing to see algebro-geometric objects as coefficients of the condensed cohomology of $\hat{W}_F$.
\subsection{Outline}
In \cref{section:realisation} we introduce the condensed Weil-étale realisation. This is a functor which associates to a scheme $X$ which is locally of finite type over a $p$-adic field $F$ an object of $B_{\hat{W}_F}$, i.e.\ a condensed set with an action of the pro-condensed group $\hat{W}_F$. This functor respects fiber products. In particular, if $G$ is a commutative $F$-group scheme locally of finite type, then its condensed Weil-étale realisation $G(\overline{L})_{\mathrm{cond}}$ is an object of $\Ab(B_{\hat{W}_F})$. We also show that this functor respects the exactness of some exact sequences 
\begin{repcor}{exactweilrealisation}
Let \[
0\rightarrow A'\rightarrow A\rightarrow A''\rightarrow 0
\] be an exact sequence of commutative algebraic groups over $F$. Suppose that for every finite field extension $K$ of $L$ the topological abelian groups $A(K)^{\mathrm{top}}$ and $A''(K)^{\mathrm{top}}$ are prodiscrete. Then the corresponding sequence in $\Ab(B_{\hat{W}_F})$ \[
0\rightarrow A'(\overline{L})_{\mathrm{cond}}\rightarrow A(\overline{L})_{\mathrm{cond}}\rightarrow A''(\overline{L})_{\mathrm{cond}}\rightarrow 0.
\] is exact.
\end{repcor}
By functoriality, we extend this construction to $1$-motives. If $\MM=[Y\rightarrow E]$ is a $1$-motive over $F$, we define its condensed Weil-étale realisation $\MM(\overline{L})_{\mathrm{cond}}\coloneqq[Y(\overline{L})_{\mathrm{cond}}\rightarrow E(\overline{L})_{\mathrm{cond}}]$, which represents an object of the derived category $\DDD^{\mathrm{b}}(B_{\hat{W}_F})$. As a consequence of \cref{exactweilrealisation} we prove the existence of a condensed weight filtration 
\begin{repprop}{condensedfiltration}
Let $\MM=[Y\overset{u}{\rightarrow} E]$ be a 1-motive over $F$. We have a filtration \[
0\hookrightarrow W_{-2}\MM(\overline{L})_{\mathrm{cond}}\hookrightarrow W_{-1}\MM(\overline{L})_{\mathrm{cond}}\hookrightarrow W_{0}\MM(\overline{L})_{\mathrm{cond}}=\MM(\overline{L})_{\mathrm{cond}}
\] with $W_{-2}\MM(\overline{L})_{\mathrm{cond}}\coloneqq [0\rightarrow T(\overline{L})_{\mathrm{cond}}]$, $W_{-1}\MM(\overline{L})_{\mathrm{cond}}\coloneqq [0\rightarrow E(\overline{L})_{\mathrm{cond}}]$. Moreover, we have \[\begin{split}
 gr^0\coloneqq W_0\MM(\overline{L})_{\mathrm{cond}}/W_{-1}\MM(\overline{L})_{\mathrm{cond}}&=[Y(\overline{L})_{\mathrm{cond}}\rightarrow 0], \\ gr^{-1}\coloneqq W_{-1}\MM(\overline{L})_{\mathrm{cond}}/W_{-2}\MM(\overline{L})_{\mathrm{cond}}&=[0\rightarrow A(\overline{L})_{\mathrm{cond}}], \\ 
 \end{split}
\]
\end{repprop}

\vspace{1em}
In \cref{section:structurealgebraicgroups1motives} we study the structure of $\HH^q(B_{\hat{W}_F},\MM(\overline{L})_{\mathrm{cond}})$ when $\MM$ is a $1$-motive over $F$. We start with the case where $\MM$ is either a torus (see \cref{prop:structuretorus}), an abelian variety (see \cref{thm:structurecohomologyabvar}) or a semiabelian variety (see \cref{prop:structurecohomsab}). In all these cases the cohomology groups $\HH^q(B_{\hat{W}_F},\MM(\overline{L})_{\mathrm{cond}})$ are represented by locally compact abelian groups (of finite ranks) and vanish if $q\neq 0,1$. In general we cannot expect $\HH^0(B_{\hat{W}_F},\MM(\overline{L})_{\mathrm{cond}})$ to be represented by a locally compact abelian group. For exemple, we set $\MM=[\Z\overset{u}{\rightarrow}\mathbbm{G}_m]$, where we have \[u(\Q_p):\Z\hookrightarrow \Q_p^{\times}, 1\mapsto 1+p.\] Under the identification $\Q_p^{\times}\cong \Z_p\oplus \mathbbm{F}_p^{\times}\oplus \Z$, the image of $u(\Q_p)$ is a dense subset of $\Z_p$. Then $\HH^0(B_{\hat{W}_F},\MM(\overline{L})_{\mathrm{cond}})$ contains a non-Hausdorff subgroup $\HH^0(B_{\hat{W}_F},\MM(\overline{L})_{\mathrm{cond}})^{\mathrm{nh}}\coloneqq \Z_p/\Z$. More generally, we prove the following 
\begin{repthm}{structuremotives}
Let $\MM=[Y\overset{u}{\rightarrow} E]$ be a 1-motive over $F$. Then $R\GGamma(B_{\hat{W}_F},\MM(\overline{L})_{\mathrm{cond}})$ is represented by a bounded complex of locally compact abelian groups of finite ranks. Moreover, we have \begin{itemize}
\item $\HH^{-1}(B_{\hat{W}_F},\MM(\overline{L})_{\mathrm{cond}})$ is a discrete free finitely generated abelian group.
\item $\HH^0(B_{\hat{W}_F},\MM(\overline{L})_{\mathrm{cond}})$ is an extension of $\HH^0(B_{\hat{W}_F},\MM(\overline{L})_{\mathrm{cond}})^{\mathrm{lca}}$ by $\HH^0(B_{\hat{W}_F},\MM(\overline{L})_{\mathrm{cond}})^{\mathrm{nh}}$.
\item $\HH^0(B_{\hat{W}_F},\MM(\overline{L})_{\mathrm{cond}})^{\mathrm{nh}}$ is the cokernel of a dense injective morphism of locally compact abelian groups of finite ranks (non-Hausdorff subgroup).
\item  $\HH^{0}(B_{\hat{W}_F},\MM(\overline{L})_{\mathrm{cond}})^{\mathrm{lca}}$ is a locally compact abelian group of finite ranks.
\item $\HH^1(B_{\hat{W}_F},\MM(\overline{L})_{\mathrm{cond}})$ is discrete of finite ranks.
\item $\HH^q(B_{\hat{W}_F},\MM(\overline{L})_{\mathrm{cond}})=0$ for all $q\neq 0,\pm 1$.
\end{itemize}
\end{repthm}
Moreover, we define the $\R/\Z$-twisted cohomology $R\GGamma(B_{\hat{W}_F},\MM(\overline{L})_{\mathrm{cond}};\R/\Z)$. In the case where $\MM$ is either an abelian variety, a torus or a $F$-group scheme étale-locally isomorphic to $\Z^r$, we also establish the relation between $\HH^q(B_{\hat{W}_F},\MM(\overline{L})_{\mathrm{cond}};\R/\Z)$ and $\HH^q(B_{\hat{W}_F},\MM(\overline{L})_{\mathrm{cond}})$ for all $q$. This is done in \cref{structurerztwist}. Informally speaking, copies of $\Z$ in $\HH^q(B_{\hat{W}_F},\MM(\overline{L})_{\mathrm{cond}})$ are replaced by copies of $\R/\Z$ in degree $q$, while the rest is moved to degree $q-1$ cohomology. 

\vspace{1em}
In \cref{section:duality1motives} we state and prove duality theorems for the condensed cohomology of the Weil group $\HH^q(B_{\hat{W}_F},-)$ with coefficients in condensed Weil-étale realisations of tori, abelian varieties and, more generally, $1$-motives. Our dualities are expressed in terms of Pontryagin duality between complexes in $\DDD^{\mathrm{b}}(\FLCA)$, the bounded derived $\infty$-category of the quasi-abelian category of locally compact abelian groups of finite ranks (see \cite[2.2]{GeisserMor2}). As explained in \cite[4.1]{Artusa}, this is a full stable $\infty$-subcategory of $\DDD^{\mathrm{b}}(\CC)$.

First, we introduce a canonical decomposition in the left heart $\LH(\LCA)$. We define a full subcategory $\NH\subset \LH(\LCA)$ of \emph{non-Hausdorff} objects. Moreover, for all  $M\in\LH(\LCA)$, we show the existence of a unique (up to isomorphism) exact sequence \[
0\rightarrow M^{\mathrm{nh}}\rightarrow M\rightarrow M^{\mathrm{lca}}\rightarrow 0,
\] with $M^{\mathrm{nh}}\in\NH$ and $M^{\mathrm{lca}}\in\LCA$. This is the analogue in $\LH(\LCA)$ of the decomposition of a finitely generated abelian group $A$ \[
0\rightarrow A_{\mathrm{tors}}\rightarrow A\rightarrow A_{\mathrm{free}}\rightarrow 0.
\] We use this decomposition to better understand Pontryagin duality in $\DDD^{\mathrm{b}}(\LCA)$. Indeed, we prove the following 
\begin{replem}{dualitynhlca}
Let $C,D\in \DDD^{\mathrm{b}}(\CC)$ be two objects in the essential image of $\DDD^{\mathrm{b}}(\LCA)\to \DDD^{\mathrm{b}}(\CC)$. Suppose that we have a perfect pairing \[
C\otimes^L D\rightarrow \R/\Z.
\] Then for all $n\in \N$ we have induced perfect pairings in $\Ab(\CC)$ \[
\H^n(C)^{\mathrm{lca}}\otimes \H^{-n}(D)^{\mathrm{lca}}\rightarrow \R/\Z\] and perfect pairings in $\DDD^{\mathrm{b}}(\CC)$\[ \H^n(C)^{\mathrm{nh}}\otimes^L \H^{-n+1}(D)^{\mathrm{nh}}[-1]\rightarrow \R/\Z.
\]\end{replem}
The first pairings identify $\H^n(C)^{\mathrm{lca}}$ with $\underline{\Hom}(\H^{-n}(D)^{\mathrm{lca}},\R/\Z)$ and viceversa, while the second ones identify $\H^n(C)^{\mathrm{nh}}$ with $\underline{\Ext}(\H^{-n+1}(D)^{\mathrm{nh}},\R/\Z)$ and viceversa. All these results hold when $\LCA$ is replaced with $\FLCA$.

Duality for the condensed Weil-étale realisation of tori is treated in \cref{subsection:dualitytori}. If $T$ is a torus over $F$ and $Y^*$ is its Cartier dual, we show the existence of a cup-product pairing \begin{equation}\label{intro:eqn:cupproductpairingtori}
   R\GGamma(B_{\hat{W}_F},T(\overline{L})_{\mathrm{cond}})\otimes^L R\GGamma(B_{\hat{W}_F},Y^*(\overline{L})_{\mathrm{cond}})\rightarrow \HH^1(B_{\hat{W}_F},\overline{L}^{\times})[-1]=\Z[-1].
\end{equation} We prove the following 
\begin{repprop}{tatenakayama}
The cup-product pairing \eqref{intro:eqn:cupproductpairingtori} is a perfect pairing in $\DDD^{\mathrm{b}}(\FLCA)$.
\end{repprop}

In \cref{subsection:dualitycoefficientsabelianvarieties}, we deal with duality for condensed Weil-étale realisations of abelian varieties. Let $A$ be an abelian variety over $F$ and $A^*$ its dual abelian variety. Using a condensed version of the theory of biextensions, which is developed in the Appendix \ref{appendix}, we define a condensed Poincaré pairing \begin{equation}\label{condensedpairing:av:intro}
    \psi_{\mathcal{P}_0(\overline{L})_{\mathrm{cond}}}:A(\overline{L})_{\mathrm{cond}}\otimes^L A^*(\overline{L})_{\mathrm{cond}}\rightarrow \overline{L}^{\times}[1]
\end{equation} and the induced cup-product pairing \begin{equation}\label{intro:condensedcppairing:av}
    R\GGamma(B_{\hat{W}_F},A(\overline{L})_{\mathrm{cond}})\otimes^L R\GGamma(B_{\hat{W}_F},A^*(\overline{L})_{\mathrm{cond}})\rightarrow \HH^1(B_{\hat{W_F}},\overline{L}^{\times})[0]=\Z.
\end{equation} We prove the following
\begin{repthm}{thm:condenseddualityav}(Tate duality for abelian varieties)
The cup-product pairing \eqref{condensedcppairing:av} is a perfect pairing in $\DDD^{\mathrm{b}}(\FLCA)$ and factors through $\R/\Z[-1]$. In particular, for $q=0,1$, we get a perfect pairing \[
\HH^q(B_{\hat{W}_F},A(\overline{L})_{\mathrm{cond}})\otimes \HH^{1-q}(B_{\hat{W}_F},A^*(\overline{L})_{\mathrm{cond}})\rightarrow \R/\Z.
\] between a profinite abelian group (of finite ranks) and a torsion abelian group (of finite ranks). \end{repthm}

We generalise both results to $1$-motives. Let $\MM$ be a $1$-motive over $F$ and $\MM^*$ its dual $1$-motive. As in the case of abelian varieties, we use the condensed theory of biextensions to define a condensed Poincaré pairing 
\begin{equation}\label{intro:eqn:duality1motives2}
    \psi_{\mathcal{P}(\overline{L})_{\mathrm{cond}}}:\MM(\overline{L})_{\mathrm{cond}}\otimes^L \MM^*(\overline{L})_{\mathrm{cond}}\rightarrow \overline{L}^{\times}[1]
\end{equation} and a cup-product pairing 
\begin{equation}\label{intro:eqn:rztwistedcppairing}
    R\GGamma(B_{\hat{W}_F},\MM(\overline{L})_{\mathrm{cond}})\otimes^L R\GGamma(B_{\hat{W}_F},\MM^*(\overline{L})_{\mathrm{cond}};\R/\Z)\rightarrow \HH^1(B_{\hat{W}_F},\overline{L}^{\times};\R/\Z)=\R/\Z.
\end{equation} We prove our main result, which we already stated (\cref{duality1motives2}). Finally, we deduce from this one of the two dualities of \cite[Theorem 0.1]{HarariSzamuely}
\begin{repcor}{cor:harariszamuely1}
The perfect pairing \eqref{intro:eqn:rztwistedcppairing} induces a perfect cup-product pairing in $\FLCA$ \[
\HH^{-1}(B_{\hat{G}_F},\MM(\overline{F})_{\mathrm{cond,\acute{e}t}})_{\wedge}\otimes \HH^2(B_{\hat{G}_F},\MM^*(\overline{F})_{\mathrm{cond,\acute{e}t}})\rightarrow \R/\Z,
\] where we set \[\HH^{-1}(B_{\hat{G}_F},\MM(\overline{F})_{\mathrm{cond,\acute{e}t}})_{\wedge}\coloneqq \mathrm{Ker}((\underline{Y(F)^{\mathrm{top}}})^{\wedge}\rightarrow (\underline{E(F)^{\mathrm{top}}})^{\wedge}).\]
\end{repcor}

\vspace{1em}

In the Appendix \ref{appendix} we show how to define condensed pairings for abelian varieties and $1$-motives, and we prove their compatibility with the classical Poincaré pairings. These pairings are needed for the statement of the duality theorems of \cref{section:duality1motives}. 

After presenting some generalities on cup-product pairings in a topos and their compatibility with respect to morphisms of topoi, we recall the notion of biextension in a topos, which is introduced in \cite[ VII]{SGA7I} and is related to pairings. This relationship can be summarised as follows. Let $\TT$ be a topos and let $A,B,G$ be abelian objects of $\TT$. The set $\mathrm{Biex}^1(A,B;G)$ of biextensions of $(A,B)$ by $G$ modulo isomorphism is endowed with an abelian group structure. Moreover, \cite[VII, Corollary 3.6.5]{SGA7I} gives a canonical isomorphism of abelian groups \begin{equation}\label{eqn:intro:biextensionpairingsga7}
\mathrm{Biext}^1(A,B;G)\overset{\sim}{\rightarrow}\Hom_{\DDD(\TT)}(A\otimes^L B,G[1]).
\end{equation}

In Section \ref{appendixsection:commutativegroupschemes} we analyse biextensions of commutative group schemes over a $p$-adic field $F$. In particular, we suppose that $A,B$ are commutative $F$-group schemes locally of finite type, and that $G$ is an affine algebraic group over $F$. First, we show that if $W$ is a biextension of $(A,B)$ by $G$, then the discrete $W_F$-set $W(\overline{L})$ is a biextension of $(A(\overline{L}),B(\overline{L}))$ by $\mathbbm{G}_m(\overline{L})$ in the topos of discrete $W_F$-sets. Afterwards, we use the condensed Weil-étale realisation to produce condensed biextensions in $B_{\hat{W}_F}$ under the assumption that $G=\mathbbm{G}_m$. Indeed, we prove the following
\begin{repthm}{condensedbiext}
Let $A,B$ be commutative group schemes locally of finite type over $F$, and let $W$ be a biextension of $(A,B)$ by $\mathbbm{G}_m$. Then its condensed Weil-étale realisation $W(\overline{L})_{\mathrm{cond}}$ is a biextension of $(A(\overline{L})_{\mathrm{cond}},B(\overline{L})_{\mathrm{cond}})$ by $\mathbbm{G}_m(\overline{L})_{\mathrm{cond}}$.
\end{repthm}
We also show the compatibility with the classical cup-product pairing in the following 
\begin{repprop}{compatibility3biextensions}
Let $A,B$ be commutative group schemes locally of finite type over $F$, and let $W$ be a biextension of $(A,B)$ by $\mathbbm{G}_m$. Then $W$, its Weil-étale realisation $W(\overline{L})$ and its condensed Weil-étale realisation $W(\overline{L})_{\mathrm{cond}}$ induce compatible cup-product pairings. In other words, the following diagram has commutative squares \[
\begin{tikzcd}[column sep=4.5 em, row sep=1.5em]
R\Gamma(G_F,A(\overline{F}))\otimes^L R\Gamma(G_F,B(\overline{F}))\ar[d]\ar[r,"CP_{W}"]&R\Gamma(B_{G_F}(\Set),\mathbbm{G}_m(\overline{F}))[1]\ar[d]\\
R\Gamma(W_F,A(\overline{L}))\otimes^L R\Gamma(W_F,B(\overline{L}))\ar[r,"CP_{W(\overline{L})}"] & R\Gamma(W_F,\mathbbm{G}_m(\overline{L}))[1]\\
R\Gamma(B_{\hat{W}_F},A(\overline{L})_{\mathrm{cond}})\otimes^L R\Gamma(B_{\hat{W}_F},B(\overline{L})_{\mathrm{cond}})\ar[r,"CP_{W(\overline{L})_{\mathrm{cond}}}"]\ar[u,"\sim"] & R\Gamma(B_{\hat{W}_F},\mathbbm{G}_m(\overline{L})_{\mathrm{cond}})[1]\ar[u,"\sim"]
\end{tikzcd}
\]
\end{repprop}
We apply this discussion to the Poincaré biextension $\mathcal{P}_0$ of an abelian variety $A$ and its dual $A^*$ by $\mathbbm{G}_m$, obtaining a biextension $\mathcal{P}_0(\overline{L})_{\mathrm{cond}}$ of $(A(\overline{L})_{\mathrm{cond}},A^*(\overline{L})_{\mathrm{cond}})$ by $\mathbbm{G}_m(\overline{L})_{\mathrm{cond}}$. This condensed biextension induces the condensed Poincaré pairing \eqref{condensedpairing:av:intro} \[
\psi_{\mathcal{P}_0(\overline{L})_{\mathrm{cond}}}:A(\overline{L})_{\mathrm{cond}}\otimes^L A^*(\overline{L})_{\mathrm{cond}}\rightarrow \mathbbm{G}_m(\overline{L})_{\mathrm{cond}}[1],
\] which is needed for \cref{thm:condenseddualityav}.

Finally, in Section \ref{appendixsection:biextension1motives}, we recall the generalisation of the concept of biextension to two-term complexes, as explained in \cite{Deligne}. This allows to define biextensions of $1$-motives. As a Corollary to \cref{condensedbiext}, we prove its analogue for $1$-motives 
\begin{repcor}{condensedbiextensions1motives}
Let $\MM_1=[Y_1\to E_1]$ and $\MM_2=[Y_2\to E_2]$ be two $1$-motives over $F$, and let $W$ be a biextension of $(\MM_1,\MM_2)$ by $\mathbbm{G}_m$. Then $W(\overline{L})_{\mathrm{cond}}$ is a biextension of $(\MM_1(\overline{L})_{\mathrm{cond}},\MM_2(\overline{L})_{\mathrm{cond}})$ by $\overline{L}^{\times}$.
\end{repcor}
We conclude by applying this result to the Poincaré biextension $\mathcal{P}$ of a $1$-motive $\MM$ and its Cartier dual $\MM^*$ by $\mathbbm{G}_m$, defined in \cite[10.2.11]{Deligne}. We obtain the condensed Poincaré biextension $\mathcal{P}(\overline{L})_{\mathrm{cond}}$ by $(\MM(\overline{L})_{\mathrm{cond}},\MM^*(\overline{L})_{\mathrm{cond}})$ by $\mathbbm{G}_m(\overline{L})_{\mathrm{cond}}$ and the condensed Poincaré pairing \eqref{intro:eqn:duality1motives2} \[
\psi_{\mathcal{P}(\overline{L})_{\mathrm{cond}}}:\MM(\overline{L})_{\mathrm{cond}}\otimes^L \MM^*(\overline{L})_{\mathrm{cond}}\rightarrow\mathbbm{G}_m(\overline{L})_{\mathrm{cond}}[1],
\] which is needed for \cref{duality1motives2}. Finally, we show the compatibility of this condensed pairing with the condensed weight filtration in \cref{lem:lastlemma}.

\vspace{0.5em}

In the Appendix \ref{appendixB}, written by Takashi Suzuki, the relation between Theorem \ref{intro:thm:suz} and \cref{thm:duality1motives1} is treated. 
\subsection{Preliminaries and notation}
Most of these coincide with the conventions of \cite{Artusa}. 
\subsubsection{Set-theoretical conventions}
A category $C$ is small if both $\mathrm{Ob}(C)$ and $\mathrm{Mor}(C)$ are sets. A category is essentially small if it is equivalent to a small category. We fix an uncountable strong limit cardinal $\kappa> \aleph_1$. We say that a set $S$ is $\kappa$-small if we have $|S|<\kappa$.

\subsubsection{Topos theory}
We make use of topos theory, which main reference is \cite{SGA4}. If $T$ is a topos and $X$ is an object of $T$, we denote by $T/X$ the induced topos \cite[IV.5.1]{SGA4}. We call the canonical morphism of topoi $j_X:T/X\rightarrow T$ localisation morphism. If $G$ is a group object in $T$, we denote by $B_G(T)$ the classifyig topos of $G$ (see \cite[IV.2.4]{SGA4}), which is the category of objects of $T$ with an action of $G$. By functoriality of this construction, we have a canonical morphism of topoi \[
f_G:B_G(T)\rightarrow B_*(T)=T.
\] For example, if $G$ is a group, $B_G(\Set)$ is the category of $G$-sets. Then the morphism $f_G$ is such that the functor $f_{G,*}$ is the fixed points functors \[f_{G,*}:B_G(\Set)\rightarrow \Set, \quad X\mapsto X^G.\] 
Let $\hat{G}=(G_i)_{i\in I}$ be a pro-object of the category $\Grp(T)$. As in \cite[2.3]{Artusa}, we define the classifying topos of $\hat{G}$ as \[
B_{\hat{G}}(T)\coloneqq \underset{\substack{\leftarrow\\ i\in I}}{\lim}\,B_{G_i}(T).
\] We have a canonical morphism of topoi \[
f_{\hat{G}}:B_{\hat{G}}(T)\rightarrow T.
\]
For example if $\hat{G}$ is the Galois group of a field $F$ seen as a pro-object of the topos $T=\Set$, then $B_{\hat{G}}(\Set)$ is the category of discrete Galois modules. In this case, the morphism $f_{\hat{G}}$ is such that $f_{\hat{G},*}$ is the fixed points functor.

Whenever the topos $T$ is the topos of ($\kappa$)-condensed sets, we denote $B_G(T)$ and $B_{\hat{G}}(T)$ simply by $B_G$ and $B_{\hat{G}}$.

\subsubsection{\texorpdfstring{$\infty$-}{infty}derived categories} We use the theory of stable $\infty$-categories, developed in \cite{HA}. If $A$ is an abelian category, we denote by $\DDD^{\mathrm{b}}(A)$ its bounded derived $\infty$-category. It is a stable $\infty$-category whose homotopy category is the classical bounded derived category $D^{\mathrm{b}}(A)$. If $T$ is a topos, we denote by $\DDD^{\mathrm{b}}(T)$ the bounded derived $\infty$-category $\DDD^{\mathrm{b}}(\Ab(T))$. 

If $QA$ is a quasi-abelian category in the sense of \cite{Schn}, we denote by $\DDD^{\mathrm{b}}(QA)$ its bounded derived $\infty$-category in the sense of \cite[2.1]{GeisserMor2}. It is a stable $\infty$-category whose homotopy category is the bounded derived category $D^{\mathrm{b}}(QA)$ in the sense of \cite{Schn}.

\subsubsection{Locally compact abelian groups and Pontryagin duality}
We denote by $\LCA$ the quasi-abelian category of $\kappa$-small locally compact abelian groups. Unless stated otherwise, by locally compact we mean $\kappa$-small locally compact. The category $\FLCA$ of locally compact abelian groups of finite ranks (see \cite{Hoff}) is a full subcategory of $\LCA$. Indeed, all locally compact abelian groups of finite ranks are $\kappa$-small. For a locally compact abelian group $A$, we denote by $A^{\vee}$ its Pontryagin dual, i.e.\ the locally compact abelian group $\Hom_{\LCA}(A,\R/\Z)$ with the compact-open topology. The Pontryagin duality induces equivalences of categories $\LCA^{\mathrm{op}}\cong \LCA$ and $\FLCA^{\mathrm{op}}\cong \FLCA$.

Both $\LCA$ and $\FLCA$ are quasi-abelian categories, and we can consider their bounded derived $\infty$-categories $\DDD^{\mathrm{b}}(\LCA)$ and $\DDD^{\mathrm{b}}(\FLCA)$. As showed in \cite[2.1]{GeisserMor2}, the functor $\DDD^{\mathrm{b}}(\FLCA)\to \DDD^{\mathrm{b}}(\LCA)$ is an exact and fully faithful functor of stable $\infty$-categories. It is also $t$-exact. Moreover, Pontryagin duality gives an equivalence \[
\DDD^{\mathrm{b}}(\LCA)^{\mathrm{op}}\overset{\sim}{\rightarrow} \DDD^{\mathrm{b}}(\LCA), \quad X\mapsto X^{\vee}\coloneqq R\underline{\Hom}_{\LCA}(X,\R/\Z)
\] and similarly for $\DDD^{\mathrm{b}}(\FLCA)$.

\subsubsection{Condensed Mathematics}
We denote by $\CC$ the topos of $\kappa$-condensed sets as defined in \cite{LCM}. This is the category of sheaves of sets on the site of $\kappa$-small profinite sets, with the topology where coverings are given by finitely jointly surjective maps. We can also restrict the defining site to $\kappa$-small extremally disconnected topological spaces. If $\Top^{cg}$ denotes the category of $\kappa$-compactly generated topological spaces (i.e.\ such that the continuity of functions can be checked on $\kappa$-small compact Hausdorff subspaces), we have a fully faithful functor \[
\underline{(-)}:\Top^{cg}\rightarrow \CC, \quad X\mapsto \underline{X}\coloneqq\Cont(-,X).
\] Thus $\CC$ is a topos which is able to see topological phenomena.
Unless stated otherwise, condensed (resp.\ compactly generated, resp.\ profinite, resp.\ extremally disconnected) means $\kappa$-condensed (resp.\ $\kappa$-compactly generated, resp.\ $\kappa$-small profinite, resp.\ $\kappa$-small extremally disconnected).

The functor $\underline{(-)}$ respects all limits. Consequently, the category $\Ab(\CC)$ contains $\LCA$ as full subcategory. Moreover, the internal Hom can be equivalently computed in $\Ab(\CC)$ or in $\LCA$ (see \cite[Lecture IV]{LCM}). As it is showed in \cite[4.1]{Artusa}, we have a functor $\DDD^{\mathrm{b}}(\LCA)\rightarrow \DDD^{\mathrm{b}}(\CC)$ which is an exact and fully faithful functor of stable $\infty$-categories. Moreover, the computation of $R\underline{\Hom}(-,-)$ can equivalently be done in $\DDD^{\mathrm{b}}(\CC)$ or $\DDD^{\mathrm{b}}(\LCA)$. In this sense, the Pontryagin duality can be seen as an antiequivalence of a stable $\infty$-subcategory of $\DDD^{\mathrm{b}}(\CC)$.

If $A$ is a condensed abelian group and $m\in\N$, we denote by $\leftin{m}{A}$ and $A/m$ the kernel and the cokernel of the multiplication by $m$ respectively. Moreover, we denote by $TA$ the Tate module of $A$, i.e.\ the cofiltered limit of $\leftin{m}{A}$ along the morphisms $\cdot d:\leftin{d\cdot m}{A}\rightarrow \leftin{m}{A}$. The non-derived profinite completion of $A$ is denoted by $A^{\wedge}$, while the derived completion of any complex of abelian groups $C$ is denoted by $C\otimes^L \hat{\Z}$.

\subsubsection{The Weil group of a \texorpdfstring{$p$}{p}-adic field} Unless stated otherwise, $F$ denotes a fixed $p$-adic field, with ring of integers $\O_F$ and residue field $k$. The units of the ring of integers are denoted by $\O_F^{\times}$. We fix separable closures of $F$ and $k$ and we call them $\overline{F}$ and $\overline{k}$ respectively. We denote by $F^{\mathrm{un}}$ the maximal unramified extension of $F$, and by $L$ the completion of $F^{\mathrm{un}}$. We fix a separable closure $\overline{L}$ of $L$ containing $\overline{F}$. We denote by $G_k$, $G_F$ and by $I_F$ the absolute Galois groups of $k$, $F$ and $F^{\mathrm{un}}$ respectively. By Krasner's Lemma we also have $I_F=\Gal(\overline{L}/L)$. We have a natural isomorphism of topological groups $G_k\overset{\sim}{\rightarrow}\hat{\Z}$ sending the Frobenius automorphism to $1$.
\begin{defn} The Weil groups of $k$ and $F$ are defined as follows \begin{enumerate}[(i)]
    \item The Weil group of $k$, denoted $W_k$, is defined by the following diagram \[\begin{tikzcd}
    W_k\ar[d,"\sim"]\ar[r,"\subset"]&G_k\ar[d,"\sim"]\\
    \Z\ar[r,"\subset"]& \hat{\Z}.
    \end{tikzcd}\]
    \item The Weil group of $F$ is defined as $W_F\coloneqq G_F\times_{G_k} W_k$.
\end{enumerate}
We have an exact sequence of topological groups \[
1\rightarrow I_F\rightarrow W_F\rightarrow W_k\rightarrow 1.
\] We see that $W_F$ is a prodiscrete topological group, inverse limit of discrete groups $W_F/U$, with $U$ varying among the open normal subgroups of $I_F$. If $K$ is a finite extension of $L$ with absolute Galois group $U$, the discrete groups $W_F/U$ acts on $K$ and the fixed points are $F$. We get a continuous action of $W_F$ on $\overline{L}$ with fixed points $F$.
\end{defn}
 
\subsubsection{The condensed cohomology of the Weil group}
We use the results and the notation of \cite[\S 3]{Artusa}. Thanks to the functor $\underline{(-)}:\Top^{cg}\hookrightarrow \CC$, we see the prodiscrete group $W_F$ as a pro-object of the category $\Grp(\CC)$ of condensed groups, and we denote it $\hat{W}_F$. We consider its classifying topos $B_{\hat{W}_F}$, i.e.\ the category of condensed sets with an action of the pro-condensed group $\hat{W}_F$. We get a functor \[
f_{\hat{W}_F,*}:\Ab(B_{\hat{W}_F})\rightarrow \Ab(\CC),
\] whose derived functors give the cohomology theory $\HH^q(B_{\hat{W}_F},-)$. We do the same construction at the derived level, obtaining a functor \[
R\GGamma(B_{\hat{W}_F},-):\DDD^{\mathrm{b}}(B_{\hat{W}_F})\rightarrow \DDD^{\mathrm{b}}(\CC).
\] This is the cohomology we use in this paper.

\subsubsection{Group schemes and Cartier duality}
The coefficients of $R\GGamma(B_{\hat{W}_F},-)$ come from algebraic geometry. We denote by $\Sch^{\mathrm{lft}}_F$ the category of schemes which are locally of finite type over $F$, with morphisms of schemes as morphisms. We endow $\Sch^{lft}_F$ with the fppf-topology, resp.\ the étale topology, and we call $\TT_{\mathrm{fppf}}$, resp.\ $\TT_{\mathrm{\acute{E}t}}$, the corresponding topos. We denote the internal Hom's $\underline{\Hom}_{\Ab(\TT_{\mathrm{fppf}})}(-,-)$ and $\underline{\Hom}_{\Ab(\TT_{\mathrm{\acute{E}t}})}(-,-)$ simply by $\underline{\Hom}_{\mathrm{fppf}}(-,-)$ and $\underline{\Hom}_{\mathrm{\acute{E}t}}(-,-)$, and similarly for $
\underline{\Ext}(-,-)$.

The category $\Ab(\Sch^{\mathrm{lft}}_F)$ of commutative group schemes which are locally of finite type over $F$ is a full subcategory of both $\Ab(\TT_{\mathrm{fppf}})$ and $\Ab(\TT_{\mathrm{\acute{E}t}})$. If $G$ is either finite or a torus, the object $\underline{\Hom}_{\mathrm{fppf}}(G,\mathbbm{G}_m)$ is representable by a certain $G'\in\Ab(\Sch^{\mathrm{lft}}_F)$. Such $G'$ is called \emph{Cartier dual} of $G$, and we also have $G=\underline{\Hom}_{\mathrm{fppf}}(G',\mathbbm{G}_m)$. If $G$ is finite, then so is $G'$. If $G$ is a torus (étale-locally isomorphic to $\mathbbm{G}_m^r)$, then $G'$ is étale-locally isomorphic to $\Z^r$. We obtain a natural biadditive map \[ G\times G'\rightarrow \mathbbm{G}_m\] called \emph{Cartier pairing}. If $A$ is an abelian variety, the sheaf $\underline{\Ext}_{\mathrm{fppf}}(A,\mathbbm{G}_m)$ is represented by the dual abelian variety $A^*$. Moreover, the Poincaré biextension gives a pairing \[
A\otimes^L A^*\rightarrow \mathbbm{G}_m[1]
\] in $\DDD^{\mathrm{b}}(\TT_{\mathrm{fppf}})$. We discuss this pairing more in detail in Appendix \ref{appendix}.

\subsection{Acknowledgements}
This work was carried out during my PhD at the Université de Bordeaux under the supervision of Baptiste Morin, whose support was crucial and to whom I am deeply grateful. I would like to thank Qing Liu for answering to my questions on the topology of rational points of abelian varieties. I am thankful to Adrien Morin for the exchange we had on biextensions. I am grateful to Clark Barwick and Tamás Szamuely for their valuable comments.
\section{Condensed Weil-\'etale realisation}\label{section:realisation}
 In \cite{Conrad}, Conrad shows how the topology of $F$ functorially defines a topology on the $F$-points of those schemes which are locally of finite type over $F$. We use this technique to associate to every commutative group scheme locally of finite type over $F$, resp.\ every $1$-motive $\MM$ over $F$, a \emph{condensed Weil-\'etale realisation}. This object is a condensed abelian group, resp.\ a complex of abelian groups, with an action of $\hat{W}_F$.
\subsection{Realisation of algebraic groups}
Let $K$ be a topological field. By \cite[Proposition 3.1]{Conrad}, there is a functor \begin{equation}\label{conradfunctor}
\mathbf{Sch}^{\mathrm{lft}}_K\rightarrow \Top, \qquad X\mapsto X(K)^{\mathrm{top}}
\end{equation} such that: \begin{enumerate}[i)]
\item if $X\hookrightarrow Y$ is a closed (resp.\ open) immersion of schemes, then $X(K)^{\mathrm{top}}\hookrightarrow Y(K)^{\mathrm{top}}$ is a topological embedding (resp.\ open embedding).
\item we have $(X\times_Z Y)(K)^{\mathrm{top}}=X(K)^{\mathrm{top}}\times_{Z(K)^{\mathrm{top}}} Y(K)^{\mathrm{top}}$.
\item we have $X(K)^{\mathrm{top}}=K$ with its topology when $X=\mathbbm{A}^1_{K}$.
\item (\cite[Example 2.2]{Conrad}) Let $K'/K$ be a finite extension. Let $X$ be a locally finite type $K$-scheme and let $X'\coloneqq X\times_K K'$. Then $X'$ is locally of finite type over $K'$, and the natural map $X(K)^{\mathrm{top}}\to X(K')^{\mathrm{top}}=X'(K')^{\mathrm{top}}$ is a topological closed embedding. Moreover, if $K_1\rightarrow K_2$ is a continuous map of topological rings, and $X$ and $X'$ are as above, the induced map $X(K_1)^{\mathrm{top}}\to X(K_2)^{\mathrm{top}}=X'(K_2)^{\mathrm{top}}$ is continuous.
\end{enumerate}
We apply this construction to define a condensed \'etale realisation of schemes which are locally of finite type over a topological field $K$ which is complete with respect to a non-trivial norm $\lVert - \rVert_K$. Let $\overline{K}$ be a separable closure of $K$. If $K'/K$ is a finite field extension contained in $\overline{K}$, then $\lVert - \rVert_K$ uniquely extends on $K'$, inducing a topology on $K'$. We have $X_{K'}\in \mathbf{Sch}_{K'}^{\mathrm{lft}}$ and we can define $X(K')^{\mathrm{top}}$ as in \eqref{conradfunctor}. By iv), we get a continuous action of $\Gal(K'/K)$ on $X(K')^{\mathrm{top}}$. This leads to the following
\begin{defn}
Let $K$ be a topological field which is complete with respect to a non-trivial norm. Let $X$ be a scheme which is locally of finite type over $K$. We define the \emph{condensed \'etale realisation} of $X$ as \[
X(\overline{K})_{\mathrm{cond,\acute{e}t}}\coloneqq \underset{\substack{\longrightarrow\\ K'/K\\ \text{ finite ext.}}}{\lim}\, \underline{X(K')^{\mathrm{top}}},
\] which can be naturally seen as an object of $B_{\hat{G}_K}$.
\end{defn}
The condensed \'etale realisation defines a functor \[-(\overline{K})_{\mathrm{cond,\acute{e}t}}:\mathbf{Sch}^{\mathrm{lft}}_K\rightarrow B_{\hat{G}_K}.\] which commutes with fiber products. This follows from the property ii) of Conrad's functor, and from the fact that filtered colimits commute with fiber products. Thus we get a functor \[
-(\overline{K})_{\mathrm{cond,\acute{e}t}}:\Ab(\Sch^{\mathrm{lft}}_K)\rightarrow \Ab(B_{\hat{G}_K}), \qquad A\mapsto A(\overline{K})_{\mathrm{cond,\acute{e}t}}.
\]
\begin{ex}
If $A=\mathbbm{G}_m$, then we have $A(\overline{K})_{\mathrm{cond,\acute{e}t}}=\overline{K}^{\times}\coloneqq\underset{\substack{\rightarrow\\ K'/K \\ \text{ finite ext}}}{\lim}\, \underline{(K')^{\times}}$.
\end{ex}
\begin{rmk}
If the norm is archimedean, then $(K,\lVert -\rVert_K)$ is isomorphic to either $\mathbbm{R}$ or $\mathbbm{C}$ with their usual norms. In both cases, the condensed étale realisation of any $X\in\mathbf{Sch}^{\mathrm{lft}}_K$ is $\underline{X(\mathbbm{C})^{\mathrm{top}}}\in B_{\mathrm{Gal}(\mathbbm{C}/K)}$, which is represented by a locally compact $\mathrm{Gal}(\mathbbm{C}/K)$-module (see \cite[Remark 3.2]{Conrad}).
\end{rmk}
Let $F$ be a $p$-adic field. We follow the same approach to define the condensed Weil-\'etale realisation of $F$-schemes locally of finite type. Let $L$ be the completion of the maximal unramified extension $F^{\mathrm{un}}$. We fix an algebraic closure $\overline{L}$ of $L$. For every field extension $K/L$ contained in $\overline{L}$, the discrete valuation on $L$ extends uniquely to $K$, defining a topology on it. Moreover, if we set $U_K\coloneqq \Gal(\overline{L}/K)$, the discrete group $W_F/U_K$ acts continuously on $K$. We have the following 
 \begin{defn}\label{defn:cwer}
 Let $X$ be a scheme which is locally of finite type over $F$. We define the \emph{condensed Weil-\'etale realisation} of $X$ as \[
X(\overline{L})_{\mathrm{cond}}\coloneqq \underset{\substack{\longrightarrow\\ K/L\\ \text{ finite ext.}}}{\lim}\, \underline{X(K)^{\mathrm{top}}},
\] which can be naturally seen as an object of $B_{\hat{W}_F}$.
\end{defn}
We get the \emph{condensed Weil-\'etale realisation} functor \[
 -(\overline{L})_{\mathrm{cond}}:\Sch^{\mathrm{lft}}_F\rightarrow B_{\hat{W}_F}.
 \] which commutes with fiber products and induces a functor \[
 -(\overline{L})_{\mathrm{cond}}:\Ab(\Sch^{\mathrm{lft}}_F)\rightarrow\Ab(B_{\hat{W}_F}).
 \]
 For every $X\in\Sch^{\mathrm{lft}}_F$, the pullback of $X(\overline{L})_{\mathrm{cond}}$ via the morphism of topoi $B_{\hat{I}}\rightarrow B_{\hat{W}_F}$ is the condensed étale realisation of $X_L$, i.e.\ $X_L(\overline{L})_{\mathrm{cond,\acute{e}t}}$.
\begin{ex}
We set $\mu_n\coloneqq \mathrm{Spec}(F[T]/(T^n-1))$. Then $\mu_n(\overline{L})_{\mathrm{cond}}\in\Ab(B_{\hat{W}_F})$ is represented by the finite abelian group $\mu_n(\overline{L})$ of $n$th roots of unity in $\overline{L}$ with a continuous action of $W_F$, which factors through a finite quotient. 
\end{ex}
\begin{ex}\label{example:constantgroupschemerealisation}
Let $M$ be an abelian group, and let $M_F=\coprod_{m\in M} \mathrm{Spec}(F)$ be the constant group scheme over $F$ associated to $M$. By property i) of Conrad's functor, $M(K)^{\mathrm{top}}$ is discrete for all $K$ finite extension of $L$. Thus we have \[
M_F(\overline{L})_{\mathrm{cond}}= M
\] with the discrete topology and the trivial action of $\hat{W}_F$.
\end{ex}
\begin{ex}\label{example:locallyconstantrealisation}
Let $G$ be a commutative group scheme over $F$ which is locally constant for the étale topology, isomorphic to $M_{F'}$ over a finite extension $F'$ of $F$. We have \[
G(\overline{L})_{\mathrm{cond}}=\underline{G(F')^{\mathrm{top}}}=M
\] with the discrete topology and an action of $\hat{W}_F$ which factors through $\Gal(F'/F)$.
\end{ex}
\begin{ex}\label{example:charactertorus}Let $T$ be a torus over $F$ and let $Y^*$ be its Cartier dual. Let $X^*(T)\coloneqq Y^*(\overline{F})$ denote the group of characters of $T_{\overline{F}}$. It is a free finitely generated abelian group isomorphic to $\Z^r$ with a continuous action of $G_F$ which factors through $\Gal(F'/F)$. By \cref{example:locallyconstantrealisation}, we have \[
Y^*(\overline{L})_{\mathrm{cond}}=X^*(T)\in \Ab(B_{\hat{W}_F}),
\] with the discrete topology and an action of $\hat{W}_F$ which factors through the finite quotient $\Gal(F'/F)$. The Cartier pairing \[
\gamma_T:T\times \mathcal{X}(T)\rightarrow \mathbbm{G}_m.
\] induces a bilinear map in $\Ab(B_{\hat{W}_F})$ \[
T(\overline{L})_{\mathrm{cond}}\times Y^*\rightarrow \overline{L}^{\times},
\] and consequently a morphism \begin{equation}\label{equation:bilinearmaptoruscharacter}
T(\overline{L})_{\mathrm{cond}}\rightarrow \underline{\Hom}_{B_{\hat{W}_F}}(X^*(T),\overline{L}^{\times}).
\end{equation} We have the following
\begin{prop}\label{prop:dualtorus}
Let $T$ be a torus over $F$ and let $X^*(T)$ be the group of characters of $T_{\overline{F}}$. Then \eqref{equation:bilinearmaptoruscharacter} is an isomorphism in $\Ab(B_{\hat{W}_F})$.
\end{prop}
\begin{proof}
We have \[ T_{F'}=(\mathbbm{G}_m^r)_{F'}, \qquad Y^*_{F'}=\Z^r_{F'}\] for some $r\in \N$ and for some finite extension $F'$ of $F$. Let $K'$ be a finite extension of $L$ containing $F'$. By definition of $-(\overline{L})_{\mathrm{cond}}$ we have \[
T(\overline{L})_{\mathrm{cond}}=\underset{\substack{\rightarrow\\ K/K'\\ \text{fin.\, ext.}}}{\lim}\, \underline{T(K)^{\mathrm{top}}}.
\] Moreover, by \cite[Corollary 2.58, (i)]{Artusa} we have \[
\underline{\Hom}_{B_{\hat{W}_F}}(X^*(T),\overline{L}^{\times})=\underset{\substack{\rightarrow\\ K/K'\\ \text{fin.\, ext.}}}{\lim}\, \underline{\Hom}_{B_{W_F/U_K}}(X^*(T),\underline{K}^{\times}).
\] The morphism \eqref{equation:bilinearmaptoruscharacter} is induced by morphisms \[
\varphi_K:\underline{T(K)^{\mathrm{top}}}\rightarrow \underline{\Hom}_{B_{W_F/U_K}}(X^*(T),\underline{K}^{\times}).
\] It is enough to show that $\varphi_K$ is an isomorphism in $\Ab(B_{W_F/U_K})$ for all $K$ finite extension of $K'$. Since the morphism of topoi $j:\CC=B_{W_F/U_K}/EW_F/U_K\rightarrow B_{W_F/U_K}$ is a localisation morphism, its pullback is conservative. Hence it is enough to show that the induced morphism \[
j^*\varphi_K:j^*\underline{T(K)^{\mathrm{top}}}\rightarrow \underline{\Hom}_{\CC}(j^*X^*(T),j^*\underline{K}^{\times})
\] is an isomorphism. In other words, we forget the action of the Weil group. Since we have $T_{F'}=(\mathbbm{G}^r_m)_{F'}$ and $Y^*_{F'}=\Z^r_{F'}$, the morphism $j^*\varphi_K$ is \[
j^*\varphi_K: \underline{\mathbbm{G}_m^r(K)^{\mathrm{top}}}\rightarrow \underline{\Hom}_{\CC}(\Z^r,\underline{\mathbbm{G}_m(K)^{\mathrm{top}}}),
\] which is an isomorphism in $\Ab(\CC)$. This concludes the proof.
\end{proof}
\end{ex}
\subsection{Exactness of realisations}
The goal of this section is to study the exactness of condensed \'etale realisation and condensed Weil-\'etale realisation functors.
\begin{thm}\label{thm:exactrealisation}
Let $K$ be a complete field with respect to a non-trivial norm. Let \[
0\rightarrow A'\rightarrow A\rightarrow A''\rightarrow 0
\] be an exact sequence of smooth commutative algebraic groups over $K$.
\begin{itemize}
    \item If the norm is archimidean, the corresponding sequence in $\Ab(B_{\hat{G}_K})$
   \[
0\rightarrow A'(\overline{K})_{\mathrm{cond,\acute{e}t}}\rightarrow A(\overline{K})_{\mathrm{cond,\acute{e}t}}\rightarrow A''(\overline{K})_{\mathrm{cond,\acute{e}t}}\rightarrow 0
\] is exact.
\item If the norm is non-archimedean, the corresponding sequence in $\Ab(B_{\hat{G}_K})$ \[
0\rightarrow A'(\overline{K})_{\mathrm{cond,\acute{e}t}}\rightarrow A(\overline{K})_{\mathrm{cond,\acute{e}t}}\rightarrow A''(\overline{K})_{\mathrm{cond,\acute{e}t}}\rightarrow 0
\] is exact if the topological abelian groups $A(K')^{\mathrm{top}},A''(K')^{\mathrm{top}}$ are prodiscrete for every finite extension $K'$ of $K$.
\end{itemize}
\end{thm}
\begin{proof}
Since the functor \[
-(\overline{K})_{\mathrm{cond,\acute{e}t}}:\Ab(\Sch^{\mathrm{lft}}_K)\rightarrow \Ab(B_{\hat{G}_K})
\] respects fiber products, the only thing to show is that the morphism $A(\overline{K})_{\mathrm{cond,\acute{e}t}}\rightarrow A''(\overline{K})_{\mathrm{cond,\acute{e}t}}$ is an epimorphism in $\Ab(B_{\hat{G}_K})$. 

Let $K'/K$ be a finite field extension. Since the morphism $A_{K'}\rightarrow A''_{K'}$ is smooth, the induced map $f_{K'}:A(K')^{\mathrm{top}}\to A''(K')^{\mathrm{top}}$ is open (see the discussion before \cite[Theorem 4.5]{Conrad}).

If the norm is archimedean, we apply this observation to $K'=\overline{K}=\mathbbm{C}$. The map $f_{K'}$ is a strict epimorphism between locally compact abelian groups. Consequently, it induces an epimorphism on the associated condensed abelian groups (see \cite[Lecture IV]{LCM}).

We now suppose that the norm is non-archimedean and that $A(K')^{\mathrm{top}},A''(K')^{\mathrm{top}}$ are prodiscrete for all finite extension $K'/K$. Let $I_{K'}\coloneqq \mathrm{Im}(f_{K'})$. Since $f_{K'}$ is open, the subgroup $I_{K'}$ is closed and open in $A''(K')^{\mathrm{top}}$. Consequently, the topological abelian group $A''(K')^{\mathrm{top}}/I_{K'}$ is discrete, and we have an exact sequence of condensed abelian groups \[
0\rightarrow \underline{I_{K'}}\rightarrow \underline{A(K')^{\mathrm{top}}}\rightarrow \underline{A''(K')/I_{K'}}\rightarrow 0.
\]
Since $I_{K'}$ is a closed subgroup of $A''(K')^{\mathrm{top}}$ and since the latter is prodiscrete, then so is $I_{K'}$. By \cite[Lemma 3.1]{Andrey}, the map $A(K')^{\mathrm{top}}\rightarrow I_{K'}$ induces an epimorphism on the associated condensed abelian groups. We obtain an exact sequence in $\Ab(\CC)$ \[
0\rightarrow \underline{A'(K')^{\mathrm{top}}}\rightarrow \underline{A(K')^{\mathrm{top}}}\rightarrow \underline{I_{K'}}\rightarrow 0,
\] which identifies $\underline{I_{K'}}$ with $\mathrm{Coim}(\underline{f_{K'}})=\mathrm{Im}(\underline{f_{K'}})$. Thus we have \[
\mathrm{Coker}(\underline{f_{K'}})=\underline{A''(K')^{\mathrm{top}}}/\underline{I_{K'}},
\] which is discrete by the previous discussion. This holds in $\Ab(\CC)$ but also in $\Ab(B_{\Gal(K'/K)})$, being the left adjoint of the localisation morphism $j_{E\Gal(K'/K)}: \CC\rightarrow B_{\Gal(K'/K)}$ conservative. As a consequence, we have \[\begin{split}
\mathrm{Coker}(A(\overline{K})_{\mathrm{cond,\acute{e}t}}\rightarrow A''(\overline{K})_{\mathrm{cond,\acute{e}t}})&=\underset{\substack{\rightarrow\\ K'/K\\ \text{fin.\ ext.}}}{\lim}\, \mathrm{Coker}(\underline{A(K')^{\mathrm{top}}}\rightarrow \underline{A''(K')^{\mathrm{top}}})=\\&=\underset{\substack{\rightarrow\\ K'/K\\ \text{fin.\ ext.}}}{\lim}\, \underline{A''(K')^{\mathrm{top}}/I_{K'}}, \end{split}
\] which is discrete and coincides with the cokernel of $A(\overline{K})\rightarrow A''(\overline{K})$, which is $0$. This concludes the proof.
\end{proof}
We deduce an analogous result for the condensed Weil-\'etale realisation.
\begin{cor}\label{exactweilrealisation}
Let $F$ be a $p$-adic field and let $L$ be the completion of its maximal unramified extension $F^{\mathrm{un}}$. Let \[
0\rightarrow A'\rightarrow A\rightarrow A''\rightarrow 0
\] be an exact sequence of commutative algebraic groups over $F$. Suppose that for every finite field extension $K$ of $L$ the topological abelian groups $A(K)^{\mathrm{top}}$ and $A''(K)^{\mathrm{top}}$ are prodiscrete. Then the corresponding sequence in $\Ab(B_{\hat{W}_F})$ \[
0\rightarrow A'(\overline{L})_{\mathrm{cond}}\rightarrow A(\overline{L})_{\mathrm{cond}}\rightarrow A''(\overline{L})_{\mathrm{cond}}\rightarrow 0.
\] is exact.
\end{cor}
\begin{rmk}
We don't need the smoothness hypothesis anymore, since every algebraic group over a field of characteristic $0$ is smooth (see \cite[Corollary 8.38]{MilneAG}).
\end{rmk}
\begin{proof}
By \cref{thm:exactrealisation} we have an exact sequence in $\Ab(B_{\hat{I}})$ \[
0\rightarrow A'(\overline{L})_{\mathrm{cond,\acute{e}t}}\rightarrow A(\overline{L})_{\mathrm{cond,\acute{e}t}}\rightarrow A''(\overline{L})_{\mathrm{cond,\acute{e}t}}\rightarrow 0.
\] Let us call $j_{EW_k}$ the localisation morphism \[
B_{\hat{W}_F}/EW_k=B_{\hat{I}}\rightarrow B_{\hat{W}_F}.
\] We have $j_{EW_K}^*\circ -(\overline{L})_{\mathrm{cond}}=-(\overline{L})_{\mathrm{cond,\acute{e}t}}$. Conservativity of $j_{EW_K}^*$ implies that the sequence \[
0\rightarrow A'(\overline{L})_{\mathrm{cond}}\rightarrow A(\overline{L})_{\mathrm{cond}}\rightarrow A''(\overline{L})_{\mathrm{cond}}\rightarrow 0
\] is exact in $\Ab(B_{\hat{W}_F})$.
\end{proof}
The hypothesis of prodiscreteness of $A(K)^{\mathrm{top}}$ and $A''(K)^{\mathrm{top}}$ seems hard to check apriori, but in fact it is often satisfied for algebraic groups over $p$-adic fields. I conjecture that if $F$ is a $p$-adic field, this hypothesis holds for all algebraic groups. We will see that it is satisfied in all the cases we are interested in, i.e.\ when $A$ is an abelian variety, a semiabelian variety or a torus over the $p$-adic field $F$. 
\subsection{Realisation of \texorpdfstring{$1$}{1}-
motives}\label{section:realisation1motives}
By \cite{Ray}, a 1-motive $\mathcal{M}$ over $F$ is the following data:\begin{itemize}
\item A $F$-group scheme $Y$ which is \'etale locally isomorphic to the constant group $\Z^r$ for some $r\in \N$.
\item A semiabelian variety $E$ over $F$, i.e.\ an extension of an $F$-abelian variety $A$ by an $F$-torus $T$.
\item A morphism of $F$-group schemes $u:Y\rightarrow E$.
\end{itemize}
We denote by $\MM$ the complex of $F$-group schemes $[Y\overset{u}{\longrightarrow} E]$, where $Y$ is in degree $-1$ and $E$ is in degree $0$. This can be seen as a complex of fppf-sheaves over $F$ and thus it represents an object of the derived category $\DDD^{\mathrm{b}}(\TT_{\mathrm{fppf}})$.  The $1$-motive $\MM$ is equipped with a 3-term weight filtration \[
0\hookrightarrow W_{-2}\MM\hookrightarrow W_{-1}\MM\hookrightarrow W_0\MM=\MM, \qquad W_{-2}\MM=[0\rightarrow T], \, \, W_{-1}\MM=[0\rightarrow E]
\] where we have \[
\MM/W_{-1}\MM=[Y\rightarrow 0], \qquad W_{-1}\MM/W_{-2}\MM=[0\rightarrow A].
\]

\begin{defn}\label{defn:wereal}
Let $\MM=[Y\overset{u}{\rightarrow} E]$ be a 1-motive over $F$.
We define the \emph{condensed Weil-\'etale realisation} of $\MM$ as the complex \[ \MM(\overline{L})_{\mathrm{cond}}\coloneqq [Y(\overline{L})_{\mathrm{cond}}\overset{u(\overline{L})_{\mathrm{cond}}}{\longrightarrow} E(\overline{L})_{\mathrm{cond}}],\] where $Y(\overline{L})_{\mathrm{cond}}$ is in degree $-1$ and $E(\overline{L})_{\mathrm{cond}}$ is in degree $0$. 
\end{defn}
 The complex $\MM(\overline{L})_{\mathrm{cond}}$ represents an object of the derived category $\DDD^{\mathrm{b}}(B_{\hat{W}_F})$. The existence of a condensed weight filtration for $\MM(\overline{L})_{\mathrm{cond}}$ is a consequence of \cref{exactweilrealisation}. \begin{prop}\label{condensedfiltration}
Let $\MM=[Y\overset{u}{\rightarrow} E]$ be a 1-motive over $F$. We have a filtration \[
0\hookrightarrow W_{-2}\MM(\overline{L})_{\mathrm{cond}}\hookrightarrow W_{-1}\MM(\overline{L})_{\mathrm{cond}}\hookrightarrow W_{0}\MM(\overline{L})_{\mathrm{cond}}=\MM(\overline{L})_{\mathrm{cond}}
\] with $W_{-2}\MM(\overline{L})_{\mathrm{cond}}\coloneqq [0\rightarrow T(\overline{L})_{\mathrm{cond}}]$, $W_{-1}\MM(\overline{L})_{\mathrm{cond}}\coloneqq [0\rightarrow E(\overline{L})_{\mathrm{cond}}]$. Moreover, we have \[\begin{split}
 gr^0\coloneqq W_0\MM(\overline{L})_{\mathrm{cond}}/W_{-1}\MM(\overline{L})_{\mathrm{cond}}&=[Y(\overline{L})_{\mathrm{cond}}\rightarrow 0], \\ gr^{-1}\coloneqq W_{-1}\MM(\overline{L})_{\mathrm{cond}}/W_{-2}\MM(\overline{L})_{\mathrm{cond}}&=[0\rightarrow A(\overline{L})_{\mathrm{cond}}], \\ 
 \end{split}
\]
\end{prop}
\begin{proof}
By definition, we have a fibre sequence in $\DDD^{\mathrm{b}}(B_{\hat{W}_F})$ \[
E(\overline{L})_{\mathrm{cond}}[0]\rightarrow \MM(\overline{L})_{\mathrm{cond}}\rightarrow Y(\overline{L})_{\mathrm{cond}}[1].
\] This gives \[
gr^0=[Y(\overline{L})_{\mathrm{cond}}\to 0].
\] Moreover, by \cref{exactweilrealisation} we have a fibre sequence in $\DDD^{\mathrm{b}}(B_{\hat{W}_F})$ \[
T(\overline{L})_{\mathrm{cond}}[0]\rightarrow E(\overline{L})_{\mathrm{cond}}[0]\rightarrow A(\overline{L})_{\mathrm{cond}}[0].
\] This gives \[
gr^{-1}=[0\to A(\overline{L})_{\mathrm{cond}}].
\]
\end{proof}
\begin{rmk}
To use \cref{exactweilrealisation}, one should prove that $E(K)^{\mathrm{top}}$ and $A(K)^{\mathrm{top}}$ are prodiscrete for every finite field extension $K$ of $L$. This is done in \cref{prodiscreteabelianvarieties} and \cref{prodiscretesemiabelianvarieties}. We can say that the proof of \cref{condensedfiltration} will be complete only after the proof of such results. This does not cause any problem since we will not use this result before \cref{section:structurecohomologymotives}.
\end{rmk}
\section{Condensed structures on cohomology groups}\label{section:structurealgebraicgroups1motives}
If $G$ is an algebraic group over $F$, then $G(\overline{L})$ is a discrete $W_F$-module, and its cohomology groups are given by \[
\H^q(B_{W_F}(\Set),G(\overline{L}))=\underset{\substack{\rightarrow\\ K/L}}{\lim}\,\H^q(B_{W_F/\Gal(\overline{L}/K)}(\Set),G(K)).
\]
The condensed abelian group $\HH^q(B_{\hat{W}_F},G(\overline{L})_{\mathrm{cond}})$ is a topologised version of $\H^q(B_{W_F}(\Set),G(\overline{L}))$. Indeed we have the following 
\begin{prop}\label{prop:globalsectionscohomag}
Let $G$ be an algebraic group over $F$. Then for all $q$ the canonical map \[
\H^q(B_{W_F}(\Set),G(\overline{L}))\rightarrow\H^q(B_{\hat{W}_F},G(\overline{L})_{\mathrm{cond}}) 
\] induced by the morphism of topoi $B_{\hat{W}_F}\to B_{W_F}(\Set)$ is an isomorphism of abelian groups.
\end{prop}
\begin{proof}
For all $K/L$ finite extension, we have a morphism of topoi \[\alpha_K:B_{W_F/\Gal(\overline{L}/K)}\to B_{W_F/\Gal(\overline{L}/K)}(\Set)\] such that $\alpha_{K*}$ is exact and induces an isomorphism of abelian groups \[ \H^q(B_{W_F/\Gal(\overline{L}/K)}(\Set),G(K))\overset{\sim}{\longrightarrow}\H^q(B_{W_F/\Gal(\overline{L}/K)},\underline{G(K)^{\mathrm{top}}}) \] for all $q$ (see \cite[Proposition 2.41]{Artusa}). The result follows. 
\end{proof}
\subsection{Coefficients in tori}
A direct consequence of \cref{prop:dualtorus} is the following \begin{prop}\label{prop:structuretorus}
Let $T$ be a torus over $F$. We have \[
\HH^q(B_{\hat{W}_F},T(\overline{L})_{\mathrm{cond}})=\begin{array}{ll}
\Z^l\oplus\underline{\Z_p}^m\oplus H_0 & q=0,\\
\Z^n\oplus H_1 & q=1,\\
0 & q\ge 2.
\end{array}
\] for some $l,m,n\in\N$ and some finite abelian groups $H_0,H_1$.
\end{prop}
\begin{proof}
The proof of the first part of the statement is similar to the one of \cite[Theorem 3.31]{Artusa}. The action of $\hat{W}_F$ on $X^*(T)$ is induced by a finite quotient of $G_F$, say $G\coloneqq \Gal(F'/F)$. By \cref{prop:dualtorus}, we have a spectral sequence \[
E_2^{i,j}=\HH^i(B_G,\HH^j(B_{\hat{W}_{F'}},\underline{\Hom}(X^*(T),\overline{L}^{\times}))\implies \HH^{i+j}(B_{\hat{W}_{F}},T(\overline{L})_{\mathrm{cond}}).
\] 
As a $\hat{W}_{F'}$-module, we have $\underline{\Hom}(X^*(T),\overline{L}^{\times})\cong (\overline{L}^{\times})^r$, where $r$ is the rank of $X^*(T)$.
 By \cite[Proposition 3.12]{Artusa}, we have $E_2^{i,j}=0$ for all $j \ge 2$. In addition to that, we have \[
E_2^{i,0}=\HH^i(B_G,\underline{\Hom}(X^*(T),\underline{(F')^{\times}})), \qquad E_2^{i,1}=\HH^i(B_G,M).
\] where $M$ is a free finitely generated abelian group with an action of the finite group $G$. Consequently, $E_2^{i,1}$ is finite for all $i\ge 1$.
Let us consider the exact sequence of $G$-modules \[
0\rightarrow \underline{\Hom}(X^*(T),\underline{\O_{F'}})\rightarrow \underline{\Hom}(X^*(T),\underline{(F')^{\times}})\rightarrow \underline{\Hom}(X^*(T),\Z)\rightarrow 0
\] induced by the valuation $(F')^{\times}\to \Z$. In the proof of \cite[Theorem 3.31]{Artusa}, it is shown that the condensed abelian group $\HH^i(B_G,\underline{\Hom}(X^*(T),\underline{\O_{F'}}))$ is finite for all $i\ge 1$. Consequently, the long exact cohomology sequence implies that $E_2^{i,0}$ is finite for all $i\ge 1$.

To summarise, we showed that $E_2^{i,j}$ is finite for all $i\ge 1$ for all $j$, and vanishes for all $j\ge 2$ for all $i$. Consequently, we have exact sequences for all $q\ge 1$ \[
0\rightarrow E_{\infty}^{q-1,1}\rightarrow \HH^q(B_{\hat{W}_F},T(\overline{L})_{\mathrm{cond}})\rightarrow E_{\infty}^{q,0}\rightarrow 0
\] For $q\ge 2$, both $E_{\infty}^{q-1,1}$ and $E_{\infty}^{q,0}$ are finite, hence so is $\HH^q(B_{\hat{W}_{F}},T(\overline{L})_{\mathrm{cond}})$. This implies that the condensed abelian group $\HH^q(B_{\hat{W}_F},T(\overline{L})_{\mathrm{cond}})$ is represented by its underlying abelian group seen as a discrete topological group. Consequently we have, for all $q\ge 2$ \[\HH^q(B_{\hat{W}_F},T(\overline{L})_{\mathrm{cond}})(*)=\H^q(B_{W_F}(\Set),T(\overline{L}))=0.\]
 
For $q=1$, the condensed abelian group $E_{\infty}^{q-1,1}$ is a subgroup of $\HH^0(B_G,M)$. Consequently, it is a free finitely generated abelian group, while $E_{\infty}^{q,0}$ is finite. The fact that $\HH^1(B_{\hat{W}_F},T(\overline{L})_{\mathrm{cond}})$ is a finitely generated abelian group follows.

We just need to determine the structure of $\HH^0(B_{\hat{W}_F},T(\overline{L})_{\mathrm{cond}})=\underline{T(F)^{\mathrm{top}}}$. Since $T(F)^{\mathrm{top}}$ is isomorphic to a finite power of $F^{\times}$ as a topological abelian group, the result follows.
\end{proof}
\subsection{Coefficients in abelian varieties}
Let $A/F$ be an abelian variety. We determine the structure of $A(\overline{L})_{\mathrm{cond}}$ and of $\HH^q(B_{\hat{W}_F},A(\overline{L})_{\mathrm{cond}})$ for all $q$. We have the following 
\begin{prop}\label{prodiscreteabelianvarieties}
Let $K/L$ be a finite extension. Then the topological abelian group $A(K)^{\mathrm{top}}$ is prodiscrete.
\end{prop}
Before proving this result, we make a few comments. We observe that $A_K$ is an abelian variety over $K$ and we have $A_K(K)=A(K)$. Let $\mathcal{O}_K$ be the ring of integers of $K$, and let us call $\mathcal{A}$ the N\'eron model of $A_K$ over $\mathcal{O}_K$. We have an isomorphism of abelian groups \begin{equation}\label{neron1}
\mathcal{A}(\mathcal{O}_K)\cong A(K).
\end{equation} Moreover, the ring $\mathcal{O}_K\cong \underset{\leftarrow}{\lim}\, \mathcal{O}_K/\p_K^n$ is a complete Noetherian local ring. Consequently, the morphism of abelian groups \begin{equation}\label{neron2}
\mathcal{A}(\mathcal{O}_K)\rightarrow \underset{\substack{\leftarrow\\n}}{\lim}\, \mathcal{A}(\mathcal{O}_K/\p_K^n)
\end{equation} induced by $\mathcal{A}(\mathcal{O}_K)\rightarrow \mathcal{A}(\mathcal{O}_K/\p_K^n)$ is an isomorphism of abelian groups by \cite[\S 6]{Schemata}.  

By composing \eqref{neron1} and \eqref{neron2}, we get an isomorphism \[A(K)\cong \underset{\substack{\leftarrow\\n}}{\lim}\, \mathcal{A}(\mathcal{O}_K/\p_K^n).\] To deduce the prodiscreteness of $A(K)^{\mathrm{top}}$, we have to show that \eqref{neron1} and \eqref{neron2} are actually homeomorphism between topological abelian groups. This is expressed by the following
\begin{lem}\label{lem:homeos} The following hold true:
\begin{enumerate}[(1)]
    \item Let $R\to R'$ be a continuous map of local topological rings, and let $X$ be an $R$-scheme locally of finite type. We set $X'\coloneqq X\times_R R'$. Then the map $X(R)^{\mathrm{top}}\to X'(R')^{\mathrm{top}}=X(R')^{\mathrm{top}}$ is continuous. Moreover, if $R\to R'$ is an open embedding, then so is $X(R)^{\mathrm{top}}\to X(R')^{\mathrm{top}}$.
    \item Let $R$ be a complete discrete valuation ring with maximal ideal $\mathfrak{m}$. Let $X$ be an $R$-scheme which is locally of finite type. The morphism \[\beta:X(R)^{\mathrm{top}}\rightarrow \underset{\substack{\leftarrow\\n}}{\lim}\, X(R/\mathfrak{m}^n)^{\mathrm{top}}\] is a homeomorphism.
\end{enumerate}
 \end{lem}
\begin{rmk}\label{opencovering}
Let $R$ be a local ring and let $\{U_i\}_{i\in I}$ be an open cover of an $R$-scheme $X$. Then the family $\{U_i(R)^{\mathrm{top}}\}_{i\in I}$ is an open cover of $X(R)^{\mathrm{top}}$. Indeed, for all $i$ the morphism $U_i(R)^{\mathrm{top}}\subset X(R)^{\mathrm{top}}$ is open by the properties of Conrad's topology. Moreover, we have $\bigcup_{i\in I} U_i(R)^{\mathrm{top}} \subset X(R)$. To see that this is an equality, we take $f:\mathrm{Spec}(R)\to X$ and we show that it factors through $U_i$ for some $i$.  Let $s$ be the closed point of $\mathrm{Spec}(R)$. We have $f(s)\in U_i$ for some $i$. The subset $f^{-1}(U_i)$ is an open subset of $\mathrm{Spec}(R)$ containing $s$. Let $t$ be prime ideal of $\mathrm{Spec}(R)\setminus f^{-1}(U_i)$. The closure of $t$ is contained in $\mathrm{Spec}(R)\setminus f^{-1}(U_i)$ but contains $s$, which is a contradiction. Thus we have $f^{-1}(U_i)=\mathrm{Spec}(R)$.
\end{rmk}
\begin{proof} 
For (1), we use \cref{opencovering} to reduce to the case where $X$ is affine. Now the result follows from \cite[Example 2.2]{Conrad}.

In the hypotheses of (2), the map $\beta$ is a bijection by \cite[\S 6, Proposition 2]{Schemata}. We just need to show that $\beta$ is a local homeomorphism. Consequently, by \cref{opencovering} we can suppose that $X$ is affine. We have $X=\mathrm{Spec}(R[t_1,\dots,t_k]/I)$ for some $k\in\N$ and some ideal $I$. Then $X(R)^{\mathrm{top}}$ is the closed subset of $R^k$ defined by  $\{(x_1,\dots,x_k)\in R^k\, | f(x_1,\dots,x_k)=0 \, \forall f\in I\}$. Similarly, $X(R/\mathfrak{m}^n)^{\mathrm{top}}$ is the closed subset of $(R/\mathfrak{m}^n)^k$ given by $\{(\overline{x_1},\dots, \overline{x_k})\in (R/\mathfrak{m}^n)^k\,| \overline{f}(\overline{x_1},\dots,\overline{x_k})=0\, \forall f\in I\}$. Thus the map \[
\beta:X(R)^{\mathrm{top}}\rightarrow \underset{\substack{\leftarrow\\n}}{\lim}\, X(R/\mathfrak{m}^n)^{\mathrm{top}}
\] is an homeomorphism, being the restriction of the homeomorphism $R^k\cong \underset{\substack{\leftarrow\\n}}{\lim}\, (R/\mathfrak{m}^n)^k$ to the closed subset $X(R)^{\mathrm{top}}$.
\end{proof}
We are now ready to prove the structure result for $A(K)^{\mathrm{top}}$.
\begin{proof}[Proof of \cref{prodiscreteabelianvarieties}]
Let $\mathcal{A}$ be the N\'eron model of $A_K$ over $\O_K$. By \cref{lem:homeos} we have isomorphisms of topological abelian groups \[
A(K)^{\mathrm{top}}\cong \mathcal{A}(\O_K)^{\mathrm{top}}\cong\underset{\substack{\leftarrow\\n}}{\lim}\, \mathcal{A}(\O_K/\p_K^n)^{\mathrm{top}}.
\] Since $\O_K/\p_K^n$ is a discrete topological ring for all $n$, then $\mathcal{A}(\O_K/\p_K^n)^{\mathrm{top}}$ is a discrete topological abelian group for all $n$. The result follows.
\end{proof}

\begin{rmk}\label{exseq:opensubgr}
There is another way of seeing that $A(K)^{\mathrm{top}}$ is prodiscrete for every finite extension $K$ of $L$. Indeed, by \cite[Theorem 7]{Mattuck}, if $k$ is a complete ultrametric field and $A$ is an abelian variety of dimension $d$ defined over $k$, then the topological abelian group $A(k)^{\mathrm{top}}$ contains an \emph{open} subgroup $A^o$ which is isomorphic to $\mathcal{O}_k^d$ as a topological abelian group.
\end{rmk}

\begin{lem}\label{cor:condenseddiscrete}
The condensed abelian groups $\leftin{m}{\HH^q(B_{\hat{W}_F},A(\overline{L})_{\mathrm{cond}})}$ and $\HH^q(B_{\hat{W}_F},A(\overline{L})_{\mathrm{cond}})/m$ are finite for all $m,q\in\N$. In particular, if we denote by $A(\overline{L})^{\delta}$ the abelian group $A(\overline{L})$ with the discrete topology, then the map of condensed abelian groups \[
A(\overline{L})^{\delta}\rightarrow A(\overline{L})_{\mathrm{cond}}
\] induces equivalences \[
\HH^q(B_{\hat{W}_F},A(\overline{L})^{\delta})\otimes^L \Z/m \rightarrow \HH^q(B_{\hat{W}_F},A(\overline{L})_{\mathrm{cond}})\otimes^L \Z/m
\] for all $m,q\in \N$.
\end{lem}
\begin{proof}
By \cref{prodiscreteabelianvarieties,exactweilrealisation} we have an exact sequence in $\Ab(B_{\hat{W}_F})$ \[
0\rightarrow A_m(\overline{L})_{\mathrm{cond}}\longrightarrow A(\overline{L})_{\mathrm{cond}}\overset{\cdot m}{\longrightarrow} A(\overline{L})_{\mathrm{cond}}\rightarrow 0.
\] The associated Kummer sequence gives an exact sequence \[
0\rightarrow\HH^{q-1}(B_{\hat{W}_F},A(\overline{L})_{\mathrm{cond}})/m\rightarrow \HH^{q}(B_{\hat{W}_F},A_m(\overline{L})_{\mathrm{cond}})\rightarrow \leftin{m}{\HH^{q}(B_{\hat{W}_F},A(\overline{L})_{\mathrm{cond}})}\rightarrow 0.
\] for all $q,m\in\N$. Since $A_m(\overline{L})_{\mathrm{cond}}$ is finite, then so is $\HH^q(B_{\hat{W}_F},A_m(\overline{L})_{\mathrm{cond}})$ by \cite[Proposition 3.26]{Artusa}. The first part of the statement follows. The map \[
\HH^q(B_{\hat{W}_F},A(\overline{L})^{\delta})\otimes^L \Z/m \rightarrow \HH^q(B_{\hat{W}_F},A(\overline{L})_{\mathrm{cond}})\otimes^L \Z/m
\] induces an equivalence in $\DDD^{\mathrm{b}}(\Ab)$ between the complexes of underlying abelian groups (see \cite[Remark 3.22]{Artusa} and \cref{prop:globalsectionscohomag}).  Since the cohomology groups of both terms are finite, then the map is an equivalence in $\DDD^{\mathrm{b}}(\Ab(\CC))$. 
\end{proof}
\begin{lem}\label{lem:torsion}
For all $q\ge 1$ the condensed abelian group $\HH^q(B_{\hat{W}_F},A(\overline{L})_{\mathrm{cond}})$ is torsion, i.e.\ we have \[
\HH^q(B_{\hat{W}_F},A(\overline{L})_{\mathrm{cond}})=\underset{\substack{\rightarrow\\ m}}{\lim}\leftin{m}{\HH^q(B_{\hat{W}_F},A(\overline{L})_{\mathrm{cond}})}
\]
\end{lem}
\begin{proof}
Since the functor which takes the sections at extremally disconnected topological spaces commutes with all limits and colimits, it is enough to show that $\HH^q(B_{\hat{W}_F},A(\overline{L})_{\mathrm{cond}})(S)$ is torsion for all $S$ extremally disconnected and for all $q\ge 1$. Moreover, we can reduce to finite extensions of $L$. We call $K$ such an extension and we set $U\coloneqq \mathrm{Gal}(\overline{L}/K)$. For all $q$, we have an exact sequence \begin{equation}\label{hsabvar}\begin{split}
0\rightarrow \HH^1(B_{W_k},\HH^{q-1}(B_{I/U},\underline{A(K)^{\mathrm{top}}}))(S)\rightarrow &\HH^q(B_{W_F/U},\underline{A(K)^{\mathrm{top}}})(S)\rightarrow\\ &\rightarrow \HH^0(B_{W_k},\HH^q(B_{I/U},\underline{A(K)^{\mathrm{top}}}))(S)\rightarrow 0. \end{split}
\end{equation} By \cite[Proposition 2.41]{Artusa} we have for all $q$ and for $i=0,1$ \[\HH^i(B_{W_k},\HH^q(B_{I/U},\underline{A(K)^{\mathrm{top}}}))(S)=\H^i(B_{W_k}(\Set),\H^q(B_{I/U}(\Set),\Cont(S,A(K)^{\mathrm{top}}))),\] which is torsion for all $q\ge 1$, since $I/U$ is finite. Thus $\HH^q(B_{\hat{W}_F},\underline{A(K)^{\mathrm{top}}})(S)$ is torsion for all $q\ge 2$.

For $q=1$, it is enough to show that $\HH^1(B_{W_k},\underline{A(L)^{\mathrm{top}}})$ is represented by a discrete torsion abelian group. The following argument is inspired by the proofs of \cite[Proposition 3.8]{ADT} and \cite[Lemma 5.4.3]{Karpuk2}. Let $\mathcal{A}$ be the N\'eron model for $A$ over $\O_F$ and let $\mathcal{A}_k$ be its special fibre. Let us denote by $\mathcal{A}^0$ the open subgroup scheme of $\mathcal{A}$ whose generic fibre is $A$ and whose special fibre is $\mathcal{A}_k^{id}$, the identity component of $\mathcal{A}_k$. By \cite[Theorem 7.2.1, Corollary 7.2.2]{blrneron} the base change $\mathcal{A}\times_{\O_F} \O_L$ is the N\'eron model of $A_L=A\times_F L$ over $\O_L$, whose special fibre is $\mathcal{A}_{\overline{k}}$. Moreover, the base change $\mathcal{A}^0\times_{\O_F}\O_L$ is the open subgroup scheme of $\mathcal{A}\times_{\O_F} \O_L$ whose generic fibre is $A_L$ and whose special fibre is $\mathcal{A}_{\overline{k}}^{id}$. For this last point, we are using the fact that the identity component of an algebraic group commutes with any extension of the base field (see \cite[Proposition 1.34]{MilneAG}). Then the topological abelian group $\mathcal{A}^0(\O_L)^{\mathrm{top}}=(\mathcal{A}^0\times_{\O_F}\O_L)(\O_L)^{\mathrm{top}}$ is an open subgroup of $\mathcal{A}(\O_L)^{\mathrm{top}}=(\mathcal{A}\times_{\O_F}\O_L)(\O_L)^{\mathrm{top}}$. Since $\mathcal{A}\times_{\O_F}\O_L$ is smooth over $\O_L$, which is a complete Hausdorff local ring, Hensel's lemma implies that the reduction map $\mathcal{A}(\O_L)^{\mathrm{top}}\rightarrow \mathcal{A}_k(\overline{k})$ is surjective. It follows that the open subgroup $\mathcal{A}^0(\O_L)^{\mathrm{top}}$ is of finite index in $\mathcal{A}(\O_L)^{\mathrm{top}}$ and the quotient is isomorphic to $\pi_0(\mathcal{A}_{\overline{k}})$, the component group of $\mathcal{A}_{\overline{k}}$. Consequently, we have an exact sequence of condensed $W_k$-modules \[
0\rightarrow \underline{\mathcal{A}^0(\O_L)^{\mathrm{top}}}\rightarrow \underline{\mathcal{A}(\O_L)^{\mathrm{top}}}\rightarrow \underline{\pi_0(\mathcal{A}_{\overline{k}}})\rightarrow 0, 
\] which induces an exact sequence of condensed abelian groups \begin{equation}\label{exseqcondcohom}
\HH^1(B_{W_k},\underline{\mathcal{A}^0(\O_L)^{\mathrm{top}}})\rightarrow \HH^1(B_{W_k},\underline{\mathcal{A}(\O_L)^{\mathrm{top}}})\rightarrow \HH^1(B_{W_k},\underline{\pi_0(\mathcal{A}_{\overline{k}}}))\rightarrow 0.
\end{equation}
We show the vanishing of $\HH^1(B_{W_k},\underline{\mathcal{A}^0(\O_L)^{\mathrm{top}}})$ following \cite[Proof of Proposition 3]{Schemata2}. By \cref{lem:homeos} we have an isomorphism of topological abelian groups \[
\mathcal{A}^0(\O_L)^{\mathrm{top}}\cong \underset{\substack{\leftarrow\\n}}{\lim}\, \mathcal{A}^0(\O_L/\mathfrak{p}_L)^{\mathrm{top}},
\] compatible with the action of $W_k$. In particular, the continuous homomorphism $1-\varphi:\mathcal{A}^0(\O_L)^{\mathrm{top}}\rightarrow \mathcal{A}^0(\O_L)^{\mathrm{top}}$ is the projective limit of the continuous morphisms \[
(1-\varphi)_n:\mathcal{A}^0(\O_L/\p_L)^{\mathrm{top}}\rightarrow \mathcal{A}^0(\O_L/\p_L)^{\mathrm{top}}
\] which are surjective for all $n$ (see \cite[Proof of Proposition 3]{Schemata2}). By discreteness, we conclude that $\underline{(1-\varphi)_n}$ is an epimorphism for all $n$, and since $B_{W_k}$ is replete, $\underline{1-\varphi}$ is an epimorphism of condensed $W_k$-modules (see \cite[Proposition 3.1.8]{proetale}). Consequently, we have \[ \HH^1(B_{W_k},\underline{\mathcal{A}^0(\O_L)^{\mathrm{top}}})=\mathrm{coker}(\underline{\mathcal{A}^0(\O_L)^{\mathrm{top}}}\overset{\underline{1-\varphi}}{\longrightarrow} \underline{\mathcal{A}^0(\O_L)^{\mathrm{top}}})=0.\] Therefore, the exact sequence \eqref{exseqcondcohom} gives an isomorphism of finite abelian groups \[
\HH^1(B_{W_k},\underline{\mathcal{A}(\O_L)^{\mathrm{top}}})=\HH^1(B_{W_k},\underline{\pi_0(\mathcal{A}_{\overline{k}})}).
\] Since $\underline{\mathcal{A}(\O_L)^{\mathrm{top}}}\cong \underline{A(L)^{\mathrm{top}}}$ as condensed $W_k$-modules, the result follows.
\end{proof}
We are ready to prove the structure result for the cohomology groups of $B_{\hat{W}_F}$ with coefficients in $A(\overline{L})_{\mathrm{cond}}$.

\begin{prop}\label{thm:structurecohomologyabvar}
Let $A$ be an abelian variety over $F$. We have \[
\HH^q(B_{\hat{W}_F},A(\overline{L})_{\mathrm{cond}})=\begin{array}{ll}
    \underline{\Z_p}^n\oplus H_0 &  q=0,\\
    \Q_p/\Z_p^m\oplus H_1 & q=1,\\
     0& q\ge 2.
\end{array}
\] for some $n,m\in \N$ and some finite abelian groups $H_0,H_1$. 
\end{prop}
\begin{proof}
By \cref{lem:torsion} we have \[
\HH^q(B_{\hat{W}_F},A(\overline{L})_{\mathrm{cond}})=\underset{\substack{\rightarrow\\ m}}{\lim}\leftin{m}{\HH^q(B_{\hat{W}_F},A(\overline{L})_{\mathrm{cond}})}
\] for all $q\ge 1$. By \cref{cor:condenseddiscrete}, this coincides with \[
\underset{\substack{\rightarrow\\ m}}{\lim}\, \leftin{m}{\HH^q(B_{\hat{W}_F},A(\overline{L})^{\delta})}=\HH^q(B_{\hat{W}_F},A(\overline{L})^{\delta}).
\] Consequently, for all $q\ge 1$ we have \[
\HH^q(B_{\hat{W}_F},A(\overline{L})_{\mathrm{cond}})=\HH^q(B_{\hat{W}_F},A(\overline{L})^{\delta}).
\] Consequently $\HH^q(B_{\hat{W}_F},A(\overline{L})_{\mathrm{cond}})$ vanishes for all $q\ge 2$ and it is of finite $\Q_p/\Z_p$-type for $q=1$ (see \cite[Lemma 5.4.1, Lemma 5.4.3]{Karpuk2}).

To conclude, by \cite[Corollary 2.26, (1)]{Artusa} we have \[
\HH^0(B_{\hat{W}_F},A(\overline{L})_{\mathrm{cond}})=\underset{\substack{\rightarrow\\ K/L\\ \text{fin.\, ext.}}}{\lim}\, \HH^q(B_{W_F/U_K},\underline{A(K)^{\mathrm{top}}})=\underline{A(F)^{\mathrm{top}}},
\] which is a profinite abelian group isomorphic to the direct sum of a finite power of $\O_F$ and a finite abelian group (see \cref{exseq:opensubgr}).

\end{proof}
\subsection{Coefficients in semiabelian varieties}
Let $E$ be a semiabelian variety over $F$. We have an exact sequence of commutative algebraic groups over $F$ \begin{equation}\label{eqn:semiabelianvariety}
0\rightarrow T\rightarrow E\rightarrow A\rightarrow 0.
\end{equation} for some abelian variety $A$ and some torus $T$. We use \cref{exactweilrealisation} to determine the structure of $R\GGamma(B_{\hat{W}_F},E(\overline{L})_{\mathrm{cond}})$ in terms of the one of $R\GGamma(B_{\hat{W}_F},T(\overline{L})_{\mathrm{cond}})$ and $R\GGamma(B_{\hat{W}_F},A(\overline{L})_{\mathrm{cond}})$. The hypothesis of \cref{exactweilrealisation} is satisfied by $A$ thanks to \cref{prodiscreteabelianvarieties}. The same holds for $E$ by the following
\begin{prop}\label{prodiscretesemiabelianvarieties}
Let $K/L$ be a finite extension. Then the topological abelian group $E(K)^{\mathrm{top}}$ is prodiscrete.
\end{prop}
\begin{proof}
Since $A_K$ and $T_K$ are smooth and connected, then so is $E_K$, by \cite[Proposition 8.1]{MilneAG}. We observe that $E_K$ contains no subgroup of type $\mathbbm{G}_a$. By \cite[Ch X.2, Theorem 2]{blrneron}, $E_K$ admits a N\'eron lft-model over $\O_K$, we call it $\mathcal{E}$. 

Now we proceed as in \cref{prodiscreteabelianvarieties}. By N\'eron mapping property and by \cite[\S 6]{Schemata}, we have isomorphisms of abelian groups \begin{equation}\label{lftneron1}
\mathcal{E}(\O_K)\cong E(K).
\end{equation} \begin{equation}\label{lftneron2}
\mathcal{E}(\O_K)\cong \underset{\substack{\leftarrow\\n}}{\lim}\, \mathcal{E}(\O_K/\p_K^n).
\end{equation} By \cref{lem:homeos} we have an isomorphism of topological abelian groups \[E(K)^{\mathrm{top}}\cong \underset{\substack{\leftarrow\\n}}{\lim}\, \mathcal{E}(\O_K/\p_K^n)^{\mathrm{top}}. \] The result follows from discreteness of $\O_K/\p_K^n$.
\end{proof}

\begin{prop}\label{prop:structurecohomsab}
Let $E$ be a semiabelian variety over $F$. We have: \begin{itemize}
    \item $\HH^0(B_{\hat{W}_F},E(\overline{L})_{\mathrm{cond}})=\underline{E(F)^{\mathrm{top}}}$ is locally compact of finite ranks, extension of an open subgroup of finite index of $\underline{A(F)^{\mathrm{top}}}$ by $\underline{T(F)^{\mathrm{top}}}$.
    \item $\HH^1(B_{\hat{W}_F},E(\overline{L})_{\mathrm{cond}})$ is discrete of finite ranks, extension of $\HH^1(B_{\hat{W}_F},A(\overline{L})_{\mathrm{cond}})$ by a quotient of $\HH^1(B_{\hat{W}_F},T(\overline{L})_{\mathrm{cond}})$.
    \item $\HH^q(B_{\hat{W}_F},E(\overline{L})_{\mathrm{cond}})=0$ for all $q\neq 0,1$.
\end{itemize}
\end{prop}
\begin{proof}
By \cref{prodiscretesemiabelianvarieties} and \cref{exactweilrealisation},  we have a fibre sequence in $\DDD^{\mathrm{b}}(\CC)$ \[
R\GGamma(B_{\hat{W}_F},T(\overline{L})_{\mathrm{cond}})\rightarrow R\GGamma(B_{\hat{W}_F},E(\overline{L})_{\mathrm{cond}})\rightarrow R\GGamma(B_{\hat{W}_F},A(\overline{L})_{\mathrm{cond}}).
\] By combining \cref{thm:structurecohomologyabvar,prop:structuretorus}, we have $\HH^q(B_{\hat{W}_F},E(\overline{L})_{\mathrm{cond}})=0$ for all $q\ge 2$. Moreover, we have an exact sequence in $\Ab(\CC)$ \[\begin{split}
0&\rightarrow \underline{T(F)^{\mathrm{top}}}\rightarrow \underline{E(F)^{\mathrm{top}}}\rightarrow \underline{A(F)^{\mathrm{top}}}\rightarrow\\&\rightarrow \HH^1(B_{\hat{W}_F},T(\overline{L})_{\mathrm{cond}})\rightarrow \HH^1(B_{\hat{W}_F},E(\overline{L})_{\mathrm{cond}})\rightarrow \HH^1(B_{\hat{W}_F},A(\overline{L})_{\mathrm{cond}})\rightarrow 0. \end{split}
\] Since we have $\underline{\Hom}_{\Ab(\CC)}(\underline{A(F)^{\mathrm{top}}},\Z)=0$, the map \[
\underline{A(F)^{\mathrm{top}}}\rightarrow \HH^1(B_{\hat{W}_F},T(\overline{L})_{\mathrm{cond}})
\] has finite image. Consequently, the image of $\underline{E(F)^{\mathrm{top}}}\to \underline{A(F)^{\mathrm{top}}}$ is represented by an open subgroup of $A(F)^{\mathrm{top}}$ of finite index. The result follows.
\end{proof}

\subsection{Coefficients in 1-motives}\label{section:structurecohomologymotives}
Let $\MM=[Y\overset{u}{\rightarrow} E]$ be a 1-motive over $F$ and let $\MM(\overline{L})_{\mathrm{cond}}$ be its condensed Weil-\'etale realisation, which represents an object of $\DDD^{\mathrm{b}}(B_{\hat{W}_F})$. We determine the structure of $R\GGamma(B_{\hat{W}_F},\MM(\overline{L})_{\mathrm{cond}})$. We start with the following remark. We can't expect $\HH^0(B_{\hat{W}_F},\MM(\overline{L})_{\mathrm{cond}})$ to be in the essential image of $\LCA\subset \Ab(\CC)$ in general. Indeed, $\HH^0(B_{\hat{W}_F},\MM(\overline{L})_{\mathrm{cond}})$ contains the cokernel of the morphism $\underline{u^0}:\underline{Y(F)^{\mathrm{top}}}\rightarrow \underline{E(F)^{\mathrm{top}}}$, which is not always in $\LCA$ as the following example shows.
\begin{ex}\label{ex:nh}
Let us consider $F=\Q_p$ ($p\ge 3$), and let $\MM$ be the 1-motive $[\Z\overset{u}{\rightarrow}\mathbbm{G}_m]$ such that $u^0:\Z\to \Q_p^{\times}$ is the injection sending $1$ to $1+p$. Then, under the isomorphism $\Q_p^{\times}\cong\Z_p\oplus\F_p^{\times}\oplus \Z$, the image of $\Z$ is a dense subset of $\Z_p$. Consequently, the cokernel of the induced map \[
\underline{\Z}\rightarrow \underline{\Q_p^{\times}}
\] contains the condensed abelian group $\underline{\Z_p}/\underline{\Z}$, which is not locally compact. 
\end{ex}
 Condensed Mathematics is able to deal with non-Hausdorff situations. Indeed, we can isolate the non-Hausdorff term via the following
\begin{defn}(from \cref{decompositionsection}, \eqref{eqn:decompositiondefinition})
Let $u^0:Y(F)^{\mathrm{top}}\to E(F)^{\mathrm{top}}$ be the continuous morphism of locally compact abelian groups of finite ranks induced by $u$. Let $\overline{u^0(Y(F)^{\mathrm{top}})}$ be the closure of $u^0(Y(F)^{\mathrm{top}})$ in $E(F)^{\mathrm{top}}$. We define \[
\HH^0(B_{\hat{W}_F},\MM(\overline{L})_{\mathrm{cond}})^{\mathrm{nh}}\coloneqq \mathrm{coker}(\, \underline{u^0(Y(F)^{\mathrm{top}})}\rightarrow \underline{(\overline{u^0(Y(F)^{\mathrm{top}})})}\,)
\]  and \[
\HH^0(B_{\hat{W}_F},\MM(\overline{L})_{\mathrm{cond}})^{\mathrm{lca}}\coloneqq\HH^0(B_{\hat{W}_F},\MM(\overline{L})_{\mathrm{cond}})/\HH^0(B_{\hat{W}_F},\MM(\overline{L})_{\mathrm{cond}})^{\mathrm{nh}}.
\]\end{defn} 
In \cref{ex:nh}, we have $\HH^0(B_{\hat{W}_F},\MM(\overline{L})_{\mathrm{cond}})^{\mathrm{nh}}=\underline{\Z_p}/\underline{\Z}$. The structure of the cohomology groups of $R\GGamma(B_{\hat{W}_F},\MM(\overline{L})_{\mathrm{cond}})$ is given by the following 
\begin{thm}\label{structuremotives}
Let $\MM=[Y\overset{u}{\rightarrow} E]$ be a 1-motive over $F$. Then we have $R\GGamma(B_{\hat{W}_F},\MM(\overline{L})_{\mathrm{cond}})\in \DDD^{\mathrm{b}}(\FLCA)$. Moreover, we have \begin{itemize}
\item $\HH^{-1}(B_{\hat{W}_F},\MM(\overline{L})_{\mathrm{cond}})$ is a discrete free finitely generated abelian group.
\item $\HH^0(B_{\hat{W}_F},\MM(\overline{L})_{\mathrm{cond}})$ is an extension of $\HH^0(B_{\hat{W}_F},\MM(\overline{L})_{\mathrm{cond}})^{\mathrm{lca}}$ by $\HH^0(B_{\hat{W}_F},\MM(\overline{L})_{\mathrm{cond}})^{\mathrm{nh}}$.
\item $\HH^0(B_{\hat{W}_F},\MM(\overline{L})_{\mathrm{cond}})^{\mathrm{nh}}$ is the cokernel of a dense injective morphism of locally compact abelian groups of finite ranks.
\item  $\HH^{0}(B_{\hat{W}_F},\MM(\overline{L})_{\mathrm{cond}})^{\mathrm{lca}}$ is a locally compact abelian group of finite ranks.
\item $\HH^1(B_{\hat{W}_F},\MM(\overline{L})_{\mathrm{cond}})$ is discrete of finite ranks.
\item $\HH^q(B_{\hat{W}_F},\MM(\overline{L})_{\mathrm{cond}})=0$ for all $q\neq 0,\pm 1$.
\end{itemize}
\end{thm}
\begin{proof}
We have a fibre sequence in $\DDD^{\mathrm{b}}(\CC)$ \[
R\GGamma(B_{\hat{W}_F},E(\overline{L})_{\mathrm{cond}})\rightarrow R\GGamma(B_{\hat{W}_F},\MM(\overline{L})_{\mathrm{cond}})\rightarrow R\GGamma(B_{\hat{W}_F},Y(\overline{L})_{\mathrm{cond}})[1].
\] The result follows by taking the long exact cohomology sequence and by \cref{prop:structurecohomsab} and \cite[Theorem 3.30]{Artusa}.
\end{proof}
\begin{rmk} $\HH^0(B_{\hat{W}_F},\MM(\overline{L})_{\mathrm{cond}})^{\mathrm{lca}}$ is the maximal locally compact quotient of the condensed abelian group $\HH^0(B_{\hat{W}_F},\MM(\overline{L})_{\mathrm{cond}})$. Indeed, for all $A\in \LCA$ we have \[\Hom_{\Ab(\CC)}(\HH^0(B_{\hat{W}_F},\MM(\overline{L})_{\mathrm{cond}})^{\mathrm{nh}},A)=0.\] Thus any map $\HH^0(B_{\hat{W}_F},\MM(\overline{L})_{\mathrm{cond}})\to A$ factors through $\HH^0(B_{\hat{W}_F},\MM(\overline{L}))^{\mathrm{lca}}$.

Moreover, whenever the continuous map $u^0:Y(F)^{\mathrm{top}}\to E(F)^{\mathrm{top}}$ is closed, the condensed abelian group $\HH^0(B_{\hat{W}_F},\MM(\overline{L})_{\mathrm{cond}})^{\mathrm{nh}}$ vanishes and  $\HH^0(B_{\hat{W}_F},\MM(\overline{L})_{\mathrm{cond}})$ is locally compact of finite ranks.\end{rmk}
\begin{ex}
Let $\MM=[\Z\overset{u}{\rightarrow} \mathbbm{G}_m]$ be the 1-motive of \cref{ex:nh}, defined over $F=\Q_p$. The morphism $u(\overline{L}):\Z\to \overline{L}^{\times}$ sends $1$ to $1+p$. We determine the structure of the cohomology groups $\HH^q(B_{\hat{W}_F},\MM(\overline{L})_{\mathrm{cond}})$. We have a long exact cohomology sequence \[\begin{split}
0&\rightarrow \HH^{-1}(B_{\hat{W}_F},\MM(\overline{L})_{\mathrm{cond}})\rightarrow \Z\overset{\alpha}{\rightarrow} \Q_p^{\times}\rightarrow \HH^0(B_{\hat{W}_F},\MM(\overline{L})_{\mathrm{cond}})\rightarrow \\
& \rightarrow \Z\overset{\beta}{\rightarrow} \Z \rightarrow \HH^1(B_{\hat{W}_F},\MM(\overline{L})_{\mathrm{cond}})\rightarrow (\O_F^{\times})^{\vee} \rightarrow 0.\end{split}
\] Here the morphism $\alpha:\Z\to \Q_p^{\times}$ is the injection $1\mapsto 1+p$ and the morphism \[\beta:\HH^1(B_{\hat{W}_F},\Z)=\HH^1(B_{W_k},\Z) \rightarrow \HH^1(B_{\hat{W}_F},\overline{L}^{\times})=\HH^1(B_{W_k},L^{\times})\] is the zero morphism, since the valuation of $1+p$ is $0$ and the cohomology group $\HH^1(B_{W_k},L^{\times})$ identifies with $\HH^1(B_{W_k},\Z)\cong \Z$ via the valuation $L^{\times}\to\Z$. Consequently we have \[\begin{split}
\HH^{-1}(B_{\hat{W}_F},\MM(\overline{L})_{\mathrm{cond}})=0,\qquad \qquad & \HH^0(B_{\hat{W}_F},\MM(\overline{L})_{\mathrm{cond}})^{\mathrm{nh}}=\underline{\Z_p}/\underline{\Z},\\
 \HH^0(B_{\hat{W}_F},\MM(\overline{L})_{\mathrm{cond}})^{\mathrm{lca}}=(\Z\oplus \mathbbm{F}_p^{\times})\oplus \Z \qquad & \HH^1(B_{\hat{W}_F},\MM(\overline{L})_{\mathrm{cond}})\in \Ext_{\Ab(\CC)}((J)^{\vee},\Z),
 \end{split}
\]where $J$ is the kernel of $G_F^{ab}\to G_k$. 
\end{ex}

\subsection{The \texorpdfstring{$\R/\Z$}{RZ}-twist}\label{section:rztwist}
In the following we define a $\R/\Z$-twist for the complex $R\GGamma(B_{\hat{W}_F},\MM(\overline{L})_{\mathrm{cond}})$, and more generally for any $M\in \DDD^{\mathrm{b}}(B_{\hat{W}_F})$ such that we have $R\GGamma(B_{\hat{W}_F},M)\in\DDD^{\mathrm{b}}(\LCA)$. This operation allows us to express the duality theorem for $1$-motives as a Pontryagin duality.
\begin{defn}
Let $M\in\DDD^{\mathrm{b}}(B_{\hat{W}_F})$ such that we have $R\GGamma(B_{\hat{W}_F},M)\in \DDD^{\mathrm{b}}(\LCA)$. We define the $\R/\Z$-twist of $R\GGamma(B_{\hat{W}_F},M)$ as \[
R\GGamma(B_{\hat{W}_F},M;\R/\Z)\coloneqq R\GGamma(B_{\hat{W}_F},M)\otimes^L_{\LCA} \R/\Z,
\] where $\otimes^L_{\LCA}$ is the derived tensor product of the category $\DDD^{\mathrm{b}}(\LCA)$ (see \cite[Remark 4.3, i)]{Hoff}).
\end{defn}
\begin{rmk}
Let $A^{\bullet},B^{\bullet}$ be two objects of $\DDD^{\mathrm{b}}(\Ab(\CC))$ in the essential image of $\DDD^{\mathrm{b}}(\LCA)$. Then we have \[
(A^{\bullet}\otimes^L_{\Ab(\CC)}B^{\bullet})^{\vee}=R\underline{\Hom}_{\Ab(\CC)}(A^{\bullet},(B^{\bullet})^{\vee})=R\underline{\Hom}_{\LCA}(A^{\bullet},(B^{\bullet})^{\vee})=(A^{\bullet}\otimes^L_{\LCA}B^{\bullet})^{\vee}.
\] This shows that the Pontryagin dual of $T_{\Ab(\CC)}\coloneqq A^{\bullet}\otimes^L_{\Ab(\CC)} B^{\bullet}$ coincides with the Pontryagin dual of $T_{\LCA}\coloneqq A^{\bullet}\otimes^L_{\LCA} B^{\bullet}$. Since we have $T_{\LCA}\in \DDD^{\mathrm{b}}(\LCA)$, this object is dualisable, i.e.\ we have $((T_{\LCA})^{\vee})^{\vee}=T_{\LCA}$. Whenever $T_{\Ab(\CC)}$ is dualisable as well, we can conclude that $T_{\LCA}=T_{\Ab(\CC)}$. However, this fails in many cases, even not elaborated ones. \begin{ex}Consider, for example, $A^{\bullet}=\Z_p$ and $B^{\bullet}=\R$. Then we have \[
(\Z_p\otimes^L_{\Ab(\CC)}\R)^{\vee}=(\Z_p\otimes^L_{\LCA}\R)^{\vee}=0.
\] Consequently, we have $\Z_p\otimes^L_{\LCA}\R=0$. However, the complex $\Z_p\otimes^L_{\Ab(\CC)}\R$ is far from being zero. For example, the underlying abelian group of $\H^0(\Z_p\otimes^L_{\Ab(\CC)}\R)=\Z_p\otimes_{\Z}\R$ is an uncountably dimensional flat $\R$-module. 
\end{ex}
\end{rmk}

With coefficients in tori, abelian varieties or their Cartier duals, the effect of the $\R/\Z$ twist on cohomology groups can be described as follows: 
\begin{prop}\label{structurerztwist}
Let $G$ be either an abelian variety, a torus or a $F$-group scheme étale locally isomorphic to $\Z^r$ for some $r\in\N$. Then we have \begin{enumerate}[(a)]
    \item $\HH^{-1}(B_{\hat{W}_F},G(\overline{L})_{\mathrm{cond}};\R/\Z)=\HH^0(B_{\hat{W}_F},G(\overline{L})_{\mathrm{cond}})_{\mathrm{toptors}}$;
    \item For $q=0,1$, we have \[
    \HH^q(B_{\hat{W}_F},G(\overline{L})_{\mathrm{cond}};\R/\Z)=(\HH^q(B_{\hat{W}_F},G(\overline{L})_{\mathrm{cond}})_{\Z}\otimes \R/\Z) \oplus \HH^{q+1}(B_{\hat{W}_F},G(\overline{L})_{\mathrm{cond}})_{\mathrm{toptors}}.
    \]
    \item $\HH^q(B_{\hat{W}_F},G(\overline{L})_{\mathrm{cond}};\R/\Z)=0$ for $q\neq 0,\pm 1$.
\end{enumerate}
\end{prop}
\begin{rmk}
If $A\in\LCA$, $A_{\mathrm{toptors}}$ denotes the topological torsion of $A$ and $A_{\Z}$ denotes the discrete torsionfree quotient of $A$ defined in \cite[Proposition 2.2]{Hoff}. If $A=\HH^q(B_{\hat{W}_F},G(\overline{L})_{\mathrm{cond}})$ with $G$ as in the hypotheses of the Proposition, $A_{\Z}$ is a free finitely generated abelian group and the tensor product $A_{\Z}\otimes \R/\Z$, present in (b), is equivalently $\otimes_{\LCA}$ or $\otimes_{\Ab(\CC)}$ (they are canonically isomorphic in this case).
\end{rmk}
\begin{proof}
We have a fiber sequence in $\DDD^{\mathrm{b}}(\FLCA)$ \[
R\GGamma(B_{\hat{W}_F},G(\overline{L})_{\mathrm{cond}})\rightarrow R\GGamma(B_{\hat{W}_F},G(\overline{L})_{\mathrm{cond}};\R)\rightarrow R\GGamma(B_{\hat{W}_F},G(\overline{L})_{\mathrm{cond}};\R/\Z)
\] where the middle term is $R\GGamma(B_{\hat{W}_F},G(\overline{L})_{\mathrm{cond}};\R)\coloneqq R\GGamma(B_{\hat{W}_F},G(\overline{L})_{\mathrm{cond}})\otimes^L_{\LCA} \R$. We denote by $f^q$ the induced morphism $\HH^q(B_{\hat{W}_F},G(\overline{L})_{\mathrm{cond}})\rightarrow \HH^q(B_{\hat{W}_F},G(\overline{L})_{\mathrm{cond}};\R)$. The long exact cohomology sequence gives an exact sequence \[
0\rightarrow \mathrm{Coker}(f^q)\rightarrow \HH^q(B_{\hat{W}_F},G(\overline{L})_{\mathrm{cond}};\R/\Z)\rightarrow \mathrm{Ker}(f^{q+1})\rightarrow 0.
\]  By definition, we have\[
\HH^q(B_{\hat{W}_F},G(\overline{L})_{\mathrm{cond}};\R)=\underline{\Hom}(\underline{\Hom}(\HH^q(B_{\hat{W}_F},G(\overline{L})_{\mathrm{cond}}),\R),\R/\Z).
\] Consequently, the Pontryagin dual of $f^q$ is given by \[
(f^q)^{\vee}: \underline{\Hom}(\HH^q(B_{\hat{W}_F},G(\overline{L})_{\mathrm{cond}}),\R)\rightarrow \underline{\Hom}(\HH^q(B_{\hat{W}_F},G(\overline{L})_{\mathrm{cond}}),\R/\Z),
\] whose kernel and cokernel are $\underline{\Hom}(\HH^q(B_{\hat{W}_F},G(\overline{L})_{\mathrm{cond}}),\Z)$ and $\underline{\Ext}(\HH^q(B_{\hat{W}_F},G(\overline{L})_{\mathrm{cond}}),\Z)$ respectively. By the structure results on $\HH^q(B_{\hat{W}_F},G(\overline{L})_{\mathrm{cond}})$ (\cref{prop:structuretorus,thm:structurecohomologyabvar} and \cite[Theorem 3.30]{Artusa}) we have \[
\underline{\Ext}(\HH^q(B_{\hat{W}_F},G(\overline{L})_{\mathrm{cond}}),\Z)=\underline{\Hom}(\HH^q(B_{\hat{W}_F},G(\overline{L})_{\mathrm{cond}})_{\mathrm{toptors}},\R/\Z)\] and \[ \underline{\Hom}(\HH^q(B_{\hat{W}_F},G(\overline{L})_{\mathrm{cond}}),\Z)=\underline{\Hom}(\HH^q(B_{\hat{W}_F},G(\overline{L})_{\mathrm{cond}})_{\Z},\Z)
\] and both are locally compact abelian groups, hence dualisable for the Pontryagin duality. Consequently,  we have \[
\mathrm{Coker}(f^q)=\mathrm{Ker}((f^q)^{\vee})^{\vee}=\underline{\Hom}(\HH^q(B_{\hat{W}_F},G(\overline{L})_{\mathrm{cond}})_{\Z},\Z)^{\vee}=\HH^q(B_{\hat{W}_F},G(\overline{L})_{\mathrm{cond}})_{\Z}\otimes \R/\Z
\] and \[
\mathrm{Ker}(f^{q+1})=\mathrm{Coker}((f^{q+1})^{\vee})=\HH^q(B_{\hat{W}_F},G(\overline{L})_{\mathrm{cond}})_{\mathrm{toptors}}^{\vee\vee}=\HH^q(B_{\hat{W}_F},G(\overline{L})_{\mathrm{cond}})_{\mathrm{toptors}},
\] which concludes the proof.
\end{proof}
\begin{rmk}
The fact that $\underline{\Ext}(\HH^q(B_{\hat{W}_F},G(\overline{L})_{\mathrm{cond}}),\Z)$ is locally compact is a key point, and it follows from structure results on cohomology of abelian varieties, tori and their Cartier duals. This does not hold in general if $G$ is a semiabelian variety.
\end{rmk}
\begin{ex} We look more precisely at the $F$-group schemes considered in \cref{structurerztwist}.
\begin{enumerate}
    \item Let $T$ be a torus. \cref{prop:structuretorus,structurerztwist} give \begin{align*}
\HH^q(B_{\hat{W}_F},T(\overline{L})_{\mathrm{cond}})&=\begin{array}{ll}
    \Z^l\oplus \underline{\Z_p}^m\oplus H_0 & q= 0\\
    \Z^n\oplus H_1 & q=1\\
    0 & q\ge 2
\end{array}\\  \implies \HH^q(B_{\hat{W}_F},T(\overline{L})_{\mathrm{cond}};\R/\Z)&=\begin{array}{ll}
\underline{\Z_p}^m\oplus H_0 & q=-1\\
    (\R/\Z)^l\oplus H_1 & q= 0\\
    (\R/\Z)^n & q=1\\
    0 & q\neq 0,\pm 1.
\end{array}
\end{align*}
\item For $T=\mathbbm{G}_m$ we have (following \cite[Proposition 3.12]{Artusa}) \[
\HH^q(B_{\hat{W}_F},\overline{L}^{\times};\R/\Z)=\begin{array}{ll}
\O_F^{\times} & q=-1\\
    \R/\Z & q=0 \\
    \R/\Z & q=1\\
    0 & q\neq 0,\pm 1.
\end{array}
\]
\item Let $A$ be an abelian variety. \cref{thm:structurecohomologyabvar,structurerztwist} give  
\[ \HH^q(B_{\hat{W}_F},A(\overline{L})_{\mathrm{cond}};\R/\Z)=\HH^{q-1}(B_{\hat{W}_F},A(\overline{L})_{\mathrm{cond}})=\begin{array}{ll}
\underline{\Z_p}^{a}\oplus H'_0 & q=-1\\
    (\Q_p/\Z_p)^{b}\oplus H'_1 & q= 0\\
    0 & q\neq -1, 0.
\end{array}
\]
\item Let $Y$ be an $F$-scheme étale locally isomorphic to a free finitely generated abelian group. \cref{structurerztwist} and \cite[Theorem 3.30]{Artusa} give
\begin{align*}
\HH^q(B_{\hat{W}_F},Y(\overline{L})_{\mathrm{cond}})&=\begin{array}{ll}
    \Z^r & q= 0\\
    \Z^s\oplus H''_1 & q=1\\
    (\Q_p/\Z_p)^t\oplus H''_2 & q= 2\\
    0 & q\ge 3
\end{array}\\  \implies \HH^q(B_{\hat{W}_F},Y(\overline{L})_{\mathrm{cond}};\R/\Z)&=\begin{array}{ll}
    (\R/\Z)^r\oplus H''_1 & q= 0\\
    (\R/\Z)^s\oplus (\Q_p/\Z_p)^t\oplus H''_2 & q=1\\
    0 & q\neq 0,1.
\end{array}
\end{align*}
\item For $Y=\Z$ we have (following \cite[Lemma 3.28]{Artusa}) \[
\HH^q(B_{\hat{W}_F},\Z;\R/\Z)=\begin{array}{ll}
     \R/\Z & q=0 \\
      (W_F^{ab})^{\vee} & q=1\\
      0 & q\neq 0,1.
\end{array}
\]
\end{enumerate}
\end{ex}
\section{Duality}\label{section:duality1motives}
In this section we prove duality result for the cohomology of $B_{\hat{W}_F}$ with coefficients in some objects coming from algebraic geometry. Our dualities are expressed as Pontryagin duality between complexes in $\DDD^{\mathrm{b}}(\FLCA)$. 

\subsection{Pontryagin duality in \texorpdfstring{$\DDD^{\mathrm{b}}(\FLCA)$}{DbFLCA}}\label{decompositionsection}
Every statement in this section holds if we replace the category $\FLCA$ with $\LCA$.

 The Pontryagin duality in $\DDD^{\mathrm{b}}(\FLCA)$ is given by the equivalence of categories \[
R\underline{\Hom}(-,\R/\Z):\DDD^{\mathrm{b}}(\FLCA)\rightarrow \DDD^{\mathrm{b}}(\FLCA).
\] Our duality results are expressed as perfect pairings in $\DDD^{\mathrm{b}}(\FLCA)$ \begin{equation}\label{eqn:cdperfectpairingpontryagin}
C\otimes^L D\rightarrow \R/\Z. \end{equation} Such a perfect pairing identifies $C$ with the Pontryagin dual of $D$ and vice-versa. The goal of this section is to deduce from such a perfect pairing a similar duality on the cohomology groups of $C$ and $D$.

\begin{rmk}As explained in \cite[4.1]{Artusa}, the functor $\underline{(-)}:\DDD^{\mathrm{b}}(\FLCA)\to\DDD^{\mathrm{b}}(\Ab(\CC))$ allows us to see the first category as a full stable $\infty$-category of the second. Moreover, this functor is t-exact and induces an exact and fully faithful functor \[
\LH(\FLCA)\rightarrow \Ab(\CC), \qquad [X^{-1}\rightarrow X^0] \,\mapsto \underline{X^0}/\underline{X^{-1}}.
\] These statements are precisely \cite[Lemma 4.2]{Artusa} and \cite[Lemma 4.3]{Artusa}.\end{rmk}

Let us consider the perfect pairing \eqref{eqn:cdperfectpairingpontryagin}. If $\H^n(D)$ (and consequently $\H^n(C)$) is locally compact for all $n$, then the spectral sequence \[
E_2^{i,j}=\underline{\Ext}^i(\H^{-j}(D),\R/\Z)\implies \underline{\Ext}^{i+j}(D,\R/\Z)
\] degenerates. For all $n$, we get $\H^n(R\underline{\Hom}(D,\R/\Z))=(\H^{-n}(D))^{\vee}$, whence a perfect pairing in $\FLCA$ \[
\H^n(C)\otimes \H^{-n}(D)\rightarrow \R/\Z.
\] This identifies $\H^n(C)$ with the Pontryagin dual of $\H^{-n}(D)$ and vice-versa. 

However, in general some cohomology groups of $C$ and $D$ may not be in $\FLCA$. In this case, we can't deduce a perfect pairing between the cohomology groups so directly.

\begin{ex} \label{example:introdecomposition1}
We set \[C\coloneqq[\Z\overset{f}{\rightarrow} \Z_p], \quad D\coloneqq[\Q_p/\Z_p\overset{f^{\vee}}{\rightarrow} \R/\Z]\] where $f$ is a continuous map, $\Z$ is in degree $-1$ and $\Q_p/\Z_p$ is in degree $0$. Tautologically, we have a perfect pairing in $\DDD^{\mathrm{b}}(\Ab(\CC))$\[
C\otimes^L D\rightarrow \R/\Z
\] 
\begin{enumerate}[i)]
    \item  If $f$ is the zero map, the cohomology groups of $C$ and $D$ are locally compact, and we obtain perfect pairings \[
\Z_p\otimes \Q_p/\Z_p\rightarrow \R/\Z, \qquad \Z\otimes \R/\Z \rightarrow\R/\Z, \ 
\] which realise $\H^0(C)$ and $\H^{-1}(C)$ as the Pontryagin duals of $\H^0(D)$ and $\H^1(D)$ respectively.
\item If $f$ is the $\cdot p$-multiplication $\cdot p:\Z\to\Z_p$, then it is injective and the image is a dense subset of $p\Z_p$. Moreover, $f^{\vee}$ has a dense image, and its kernel is given by the $p$-torsion of $\Q_p/\Z_p$.  In this case, $\H^0(C)$ and $\H^1(D)$ are not locally compact. Moreover, $\H^0(C)$ is not the Pontryagin dual of $\H^0(D)$ and $\H^{-1}(C)$ is not the Pontryagin dual of $\H^1(D)$. Indeed, we have \[
\H^{-1}(C)=0, \quad \H^0(C)=\underline{\Z_p}/\underline{p\Z}, \quad \H^0(D)=\Z/p, \quad \H^1(D)=\underline{\R/\Z}/\underline{\Q_p/\Z_p}.
\] 
\end{enumerate}
\end{ex}
\vspace{0.5em}
Something similar happens in the category $\DDD^{perf}(\Z)$, where instead of the Pontryagin duality we consider the $\Z$-linear duality. 
\begin{ex}
We set \[
C\coloneqq [\Z\oplus \Z\overset{f}{\rightarrow}\Z], \quad D\coloneqq[\Z\overset{g}{\rightarrow}\Z\oplus \Z],
\] where $f$ is a morphism of abelian groups and $g$ is the $\Z$-linear dual of $f$. Here $\Z\oplus\Z$ is in degree $-1$ for $C$ and in degree $1$ for $D$. Tautologically, we have a perfect pairing in $\DDD^{perf}(\Z)$ \[
C\otimes^L D\rightarrow \Z.
\]\begin{enumerate}[i)]
    \item Let $f$ be the zero map. Both cohomology groups of $C$ and $D$ are free, and we obtain $\H^0(D)$ as the $\Z$-linear dual of $\H^0(C)$ and $\H^1(D)$ as the $\Z$-linear dual of $\H^{-1}(C)$.
    \item Let $f$ be the morphism $f(n,m)\coloneqq p\cdot m$ (and thus let $g$ be given by $g(n)\coloneqq(0,p\cdot n)$). Then we have \[
    \H^{-1}(C)=\Z, \quad \H^0(C)=\Z/p, \quad \H^0(D)=0, \quad \H^1(D)=\Z\oplus \Z/p.
    \] The behaviour is different from the case i) but we still obtain a duality. Indeed, the maximal free quotient $\H^1(D)_{\Z}$ is the $\Z$-linear dual of $\H^{-1}(C)$, while the torsion subgroup $\H^1(D)_{\mathrm{tors}}$ is determined by $\H^0(C)$ via $\H^1(D)_{\mathrm{tors}}=\Ext(\H^0(C),\Z)$. 
\end{enumerate}
\end{ex} In this example, we see how the duality can be recovered by separating, for any finitely generated abelian group, its torsion subgroup from its maximal free quotient. The goal of this section is to do something similar in the category $\LH(\FLCA)$. The role of the torsion subgroup is taken by the \emph{non-Hausdorff} subgroup, and the role of the maximal free quotient is taken by the maximal \emph{locally compact} quotient. 

We define, for any $M\in \Ab(\CC)$ which is in the essential image of $\LH(\FLCA)$, a unique (up to isomorphism) exact sequence \begin{equation}\label{eqn:preamble:nhlca}
0\rightarrow M^{\mathrm{nh}}\rightarrow M\rightarrow M^{\mathrm{lca}}\rightarrow 0,
\end{equation} and we prove the following \begin{lem}\label{dualitynhlca}
Let $C,D\in \DDD^{\mathrm{b}}(\Ab(\CC))$ be two objects in the essential image of $\DDD^{\mathrm{b}}(\FLCA)\to \DDD^{\mathrm{b}}(\Ab(\CC))$. Suppose that we have a perfect pairing \[
C\otimes^L D\rightarrow \R/\Z.
\] Then for all $n\in \N$ we have induced perfect pairings in $\Ab(\CC)$ \[
\H^n(C)^{\mathrm{lca}}\otimes \H^{-n}(D)^{\mathrm{lca}}\rightarrow \R/\Z
\] and perfect pairings in $\DDD^{\mathrm{b}}(\CC)$\[ \H^n(C)^{\mathrm{nh}}\otimes^L \H^{-n+1}(D)^{\mathrm{nh}}[-1]\rightarrow \R/\Z.
\]\end{lem} The first pairing identifies $\H^n(C)^{\mathrm{lca}}$ with $\underline{\Hom}(\H^{-n}(D)^{\mathrm{lca}},\R/\Z)$, while the second one identifies $\H^n(C)^{\mathrm{nh}}$ with $\underline{\Ext}(\H^{-n+1}(D),\R/\Z)$ and viceversa.
\vspace{0.5em}

We go back to the second item of \cref{example:introdecomposition1}. The decomposition \eqref{eqn:preamble:nhlca} for $\H^0(C)$ becomes \[
0\rightarrow \underline{p\Z_p}/\underline{p\Z}\rightarrow \H^0(C)\rightarrow \Z/p\rightarrow 0.
\] 

We obtain the perfect pairings 
\[\underline{p\Z_p}/\underline{p\Z}\otimes^L (\underline{\R/\Z})/(\underline{\Q_p}/\underline{\Z_p})[-1]\rightarrow \R/\Z,
\] and \[
\Z/p \otimes \Z/p\rightarrow \R/\Z.
\] The first one identifies $\H^1(D)=\H^1(D)^{\mathrm{nh}}$ with $\underline{\Ext}(\H^0(C)^{\mathrm{nh}},\R/\Z)$ and vice-versa, while the second one is the Pontryagin duality between the locally compact abelian groups $\H^0(C)^{\mathrm{lca}}$ and $\H^0(D)^{\mathrm{lca}}=\H^0(D)$.

\vspace{1em}
Let us develop the tools necessary for the proof of \cref{dualitynhlca}.
\begin{defn}
An object $X\in\LH(\FLCA)$ is \emph{non-Hausdorff} if it satisfies $\underline{\Hom}(X,\R/\Z)=0$. We define $\NH$ as the full subcategory of $\LH(\FLCA)$ of non-Hausdorff objects. \end{defn}
\begin{prop}\label{lem:nhlca}
The following facts hold: \begin{enumerate}[i)]
\item $\FLCA$ is the full subcategory of $\LH(\FLCA)$ of the objects $X\in\LH(\FLCA)$ such that $\underline{\Ext}(X,\R/\Z)=0$.
\item Both $\FLCA$ and $\NH$ are stable by extensions in $\LH(\FLCA)$. Moreover, $\FLCA$ is stable by kernels and $\NH$ is stable by cokernels.
\item We have $\NH\cap \FLCA=0$.
\item If we have $X\in \NH$ and $Y\in\FLCA$, then we have $\Hom(X,Y)=0$.
\item If $X\in \NH$, then so is $X^{\vee}[1]=\underline{\Ext}(X,\R/\Z)$.
\item The Pontryagin duality $R\underline{\Hom}(-,\R/\Z):\DDD^{\mathrm{b}}(\FLCA)^{\mathrm{op}}\rightarrow \DDD^{\mathrm{b}}(\FLCA)$ induces equivalence of categories $\FLCA^{\mathrm{op}}\overset{\sim}{\rightarrow} \FLCA$ and $\NH^{\mathrm{op}}\overset{\sim}{\rightarrow} \NH[-1]$.
\end{enumerate}
\end{prop}
\begin{proof}
To prove i), we just need to show that if $X\in\LH(\FLCA)$ satisfies $\underline{\Ext}(X,\R/\Z)=0$, then it is a locally compact abelian group. By \cite[Corollary 1.2.21]{Schn}, we have \[X=\mathrm{cofib}(X^{-1}\rightarrow X^0),\] where $X^{-1},X^0\in\FLCA$ and the morphism $X^{-1}\rightarrow X^0$ is a monomorphism. Consequently, we have an exact sequence in $\LH(\LCA)$ \[
0\rightarrow \underline{\Hom}(X,\R/\Z)\rightarrow \underline{\Hom}(X^0,\R/\Z)\rightarrow \underline{\Hom}(X^{-1},\R/\Z)
\] which tells us that $\underline{\Hom}(X,\R/\Z)$ is locally compact. Since $\underline{\Ext}(X,\R/\Z)=0$, we have \[
X=R\underline{\Hom}(R\underline{\Hom}(X,\R/\Z),\R/\Z)=R\underline{\Hom}(\underline{\Hom}(X,\R/\Z),\R/\Z).
\] Since $\underline{\Hom}(X,\R/\Z)$ is locally compact, then we have $\underline{\Ext}(\underline{\Hom}(X,\R/\Z),\R/\Z)=0$ and we conclude that \[
X=\underline{\Hom}(\underline{\Hom}(X,\R/\Z),\R/\Z)
\] is locally compact as well.

If we have $0\rightarrow X'\rightarrow X \rightarrow X''\rightarrow 0$ an exact sequence in $\LH(\FLCA)$, we get a fibre sequence in $\DDD^{\mathrm{b}}(\FLCA)$
\[
X''^{\vee}\rightarrow X^{\vee}\rightarrow X'^{\vee}.
\] If we take the long exact cohomology sequence, ii) follows.

For iii), we take $X\in \NH\cap \FLCA$. Then we have $X^{\vee}=R\underline{\Hom}(X,\R/\Z)=0$. Consequently, we have $X=X^{\vee\vee}=0^{\vee}=0$. To prove iv), we observe that if $f:X\to Y$ is a morphism in $\LH(\FLCA)$ with $X\in\NH$ and $Y\in\FLCA$, then the morphism $f^{\vee}:Y^{\vee}\rightarrow X^{\vee}$ must be 0, since $Y^{\vee}$ belongs to $\DDD^{\mathrm{b}}(\FLCA)^{\le 0}$ and $X^{\vee}$ belongs to $\DDD^{\mathrm{b}}(\FLCA)^{\ge 1}$. Consequently, we have $f=f^{\vee\vee}=0$.

Let us prove v). Using \cite[Corollary 1.2.21]{Schn}, we represent $X\in\NH$ as the cofiber in $\DDD^{\mathrm{b}}(\FLCA)$ of the monomorphism $f:X^{-1}\to X^0$, where $X^{-1},X^0$ are locally compact. Then we have a fibre sequence in $\DDD^{\mathrm{b}}(\FLCA)$ \[
(X^0)^{\vee}\overset{f^{\vee}}{\rightarrow}(X^{-1})^{\vee}\rightarrow X^{\vee}[1].
\] Since we have $\underline{\Hom}(X,\R/\Z)=0$, the morphism $f^{\vee}$ is a monomorphism and $X^{\vee}[1]\in\LH(\LCA)$. Moreover, we have \[\begin{split}
\underline{\Hom}(X^{\vee}[1],\R/\Z)&=\underline{\Hom}(\underline{\Ext}(X,\R/\Z),\R/\Z)=\\&=\H^{-1}(R\underline{\Hom}(R\underline{\Hom}(X,\R/\Z)),\R/\Z)=\H^{-1}(X)=0. \end{split}
\] As a consequence, $X^{\vee}[1]$ is an object of $\NH$. Part vi) follows.
\end{proof}
\begin{rmk}
Item iv) shows that for $X\in\LH(\FLCA)$ it is enough to have no non-trivial continuous homomorphism $X\to\R/\Z$ to guarantee that there are no non-trivial continuous homomorphisms $X\to A$, where $A$ is locally compact.
\end{rmk}
\begin{rmk}
By \cite[Corollary 1.2.21]{Schn}, for every object $X$ of $\LH(\FLCA)$ we have a fibre sequence in $\DDD^{\mathrm{b}}(\FLCA)$\[
X'\rightarrow X''\rightarrow X,
\] where $X',X''$ are objects of $\LH(\FLCA)$ which are acyclic for $R\underline{\Hom}(-,\R/\Z)$. This is analogous to the two-term free resolution of an abelian group. \end{rmk}
\begin{prop}[Decomposition in $\LH(\FLCA)$]
Let $A\in\Ab(\CC)$ be in the essential image of $\LH(\FLCA)$. Then we have a decomposition \[
0\rightarrow A^{\mathrm{nh}}\rightarrow A\rightarrow A^{\mathrm{lca}}\rightarrow 0,
\] where $A^{\mathrm{nh}}$ is in the essential image of $\NH$ and $A^{\mathrm{lca}}$ is in the essential image of $\FLCA$. This decomposition is functorial in $A$ and unique up to isomorphism.
\end{prop}
\begin{proof}
Let us prove the existence first. By hypothesis, we have $A=\underline{X}$, where $X\in\LH(\FLCA)$ is the cofibre in $\DDD^{\mathrm{b}}(\FLCA)$ of a monomorphism $f:X'\rightarrow X''$ of locally compact abelian groups. We denote by $\overline{f(X')}$ the closure of $f(X')$ in $X''$, and we set \begin{equation}\label{eqn:decompositiondefinition}
X^{\mathrm{nh}}\coloneqq \mathrm{cofib}(X'\to \overline{f(X')}), \qquad X^{\mathrm{lca}}\coloneqq \mathrm{cofib}(\overline{f(X')}\rightarrow X'').
\end{equation} Then we have an exact sequence in $\LH(\FLCA)$ \[
0\rightarrow X^{\mathrm{nh}}\rightarrow X\rightarrow X^{\mathrm{lca}}\rightarrow 0.
\] We just need to show that we have $X^{\mathrm{nh}}\in\NH$ and $X^{\mathrm{lca}}\in \FLCA$. Since $X'$ is dense in $\overline{f(X')}$, every continuous morphism $\overline{f(X')}\to\R/\Z$ which vanishes on $X'$ is the zero morphism. This implies that the map $\underline{\Hom}(\overline{f(X')},\R/\Z)\to\underline{\Hom}(X',\R/\Z)$ is a monomorphism in $\LH(\FLCA)$. Consequently, we have $\underline{\Hom}(X^{\mathrm{nh}},\R/\Z)=0$, i.e.\ $X^{\mathrm{nh}}$ is non-Hausdorff. Moreover, since the morphism $\overline{f(X)}\to X''$ is a closed immersion, it is strict. By \cite[Proposition 1.2.29, a)]{Schn}, its cokernel computed in $\LH(\FLCA)$ belongs to $\FLCA$. This proves that $X^{\mathrm{lca}}$ belongs to $\FLCA$.

We now define $A^{\mathrm{nh}}\coloneqq \underline{X^{\mathrm{nh}}}=\underline{\overline{f(X')}}/\underline{X'}$ and $A^{\mathrm{lca}}\coloneqq \underline{X^{\mathrm{lca}}}=\underline{X''}/\underline{\overline{f(X')}}$. Since the functor $\LH(\FLCA)\rightarrow \Ab(\CC)$ is exact, we get the desired decomposition.

For uniqueness, let us consider another decomposition, say $0\to B\to A\to C\to 0$, with $B=\underline{Y}$ and $C=\underline{Z}$, $Y\in\NH$ and $Z\in\FLCA$. By fully faithfulness of the functor $\LH(\FLCA)\to\Ab(\CC)$, and by item iv) of \cref{lem:nhlca}, we have \[
\Hom_{\Ab(\CC)}(B,A^{\mathrm{lca}})=\Hom_{\LH(\FLCA)}(Y,X^{\mathrm{lca}})=0.
\] Hence the composition $B\to A\to A^{\mathrm{lca}}$ vanishes, inducing a morphism of exact sequences \[
\begin{tikzcd}
0\ar[r]&B\ar[r]\ar[d] &A\ar[d,equal]\ar[r]&C\ar[d]\ar[r]&0\\
0\ar[r]&A^{\mathrm{nh}}\ar[r]& A\ar[r]&A^{\mathrm{lca}}\ar[r]&0
\end{tikzcd}
\] By the Snake Lemma, the morphism $B\to A^{\mathrm{nh}}$ is injective and the morphism $C\to A^{\mathrm{lca}}$ is surjective. Moreover, the cokernel of $B\to A^{\mathrm{nh}}$ coincides with the kernel of $C\to A^{\mathrm{lca}}$. By item ii) of \cref{lem:nhlca}, this cokernel is in the essential image of $\NH\cap \FLCA$, which is $0$ by item iii) of \cref{lem:nhlca}. Consequently, both $B\to A^{\mathrm{nh}}$ and $C\to A^{\mathrm{lca}}$ are isomorphisms. 

In the same way we prove the functoriality of the decomposition. Indeed, if we have $A,B\in \Ab(\CC)$ in the essential image of $\LH(\FLCA)$, let us consider the decompositions \[
0\rightarrow B^{\mathrm{nh}}\rightarrow B\rightarrow B^{\mathrm{lca}}\rightarrow 0,\quad 0\rightarrow A^{\mathrm{nh}}\rightarrow A\rightarrow A^{\mathrm{lca}}\rightarrow 0.
\] Since we have $\Hom_{\Ab(\CC)}(B^{\mathrm{nh}},A^{\mathrm{lca}})=0$, we obtain a morphism between the exact sequences.
\end{proof}
\begin{ex}
If $X=[\Z\overset{\cdot p^n}{\longrightarrow} \Z_p]$ and $A=\underline{X}$, then the closure of $p^n\Z$ is $p^n\Z_p$. Consequently, we have \[
A^{\mathrm{nh}}=\underline{p^n\Z_p}/\underline{p^n\Z}, \quad A^{\mathrm{lca}}=\Z/p^n.
\]
\end{ex}
\begin{rmk}
The locally compact abelian group $A^{\mathrm{lca}}$ is the maximal locally compact quotient of $A$. Indeed, for all $B\in\LCA$ receiving a map from $A$, we have $\Hom(A^{\mathrm{nh}},B)=0$. Consequently, the map $A\to B$ factors through $A^{\mathrm{lca}}$.
\end{rmk}
\begin{rmk}
The decomposition $0\rightarrow A^{\mathrm{nh}}\rightarrow A\rightarrow A^{\mathrm{lca}}\rightarrow 0$ plays for the Pontryagin duality in $\DDD^{\mathrm{b}}(\FLCA)$ the same role that the decomposition $0\rightarrow M_{\mathrm{tors}}\to M\to M_{free}\to 0$ plays for the $\Z$-linear duality in $\DDD^{perf}(\Z)$. 
\end{rmk}
\begin{rmk}
If we consider the right t-structure on $\DDD^{\mathrm{b}}(\FLCA)$, and we denote its heart by $\RH(\FLCA)$, we get a full subcategory of $\DDD^{\mathrm{b}}(\Ab(\CC))$. In general, this is not contained in $\Ab(\CC)$, since the t-structure on $\DDD^{\mathrm{b}}(\Ab(\CC))$ does not coincide with the right t-structure on $\DDD^{\mathrm{b}}(\FLCA)$. Consequently,  for $B\in\DDD^{\mathrm{b}}(\Ab(\CC))$ wihich is in the essential image of $\RH(\FLCA)$ we could have $\H^1(B)\neq 0$. Nonetheless, we obtain a similar unique decomposition as in $\LH(\FLCA)$. In particular, for $B\in\DDD^{\mathrm{b}}(\FLCA)$ which is in the essential image of $\RH(\FLCA)$, there exists a unique and functorial fibre sequence \[
B^{\mathrm{lca}}\rightarrow B\rightarrow B^{\mathrm{nh}},
\] where we have $B^{\mathrm{lca}}\in \FLCA$ and $B^{\mathrm{nh}}\in \NH[-1]$. Indeed, any object $Y\in\RH(\LCA)$ is the fibre of a continuous homomorphism $\pi:Y'\to Y''$ of locally compact abelian groups with a dense image. To obtain the decomposition, we set $B^{\mathrm{lca}}\coloneqq \underline{\mathrm{Ker}(\pi)}\in\FLCA$ and $B^{\mathrm{nh}}\coloneqq \underline{Y''}/\underline{\pi(Y')}[-1]\in \NH[-1]$. This fibre sequence is unique, up to isomorphism.
\end{rmk}
The Pontryagin duality realises an equivalence of categories $\LH(\FLCA)^{\mathrm{op}} \overset{\sim}{\rightarrow} \RH(\FLCA)$. We can see this fact under a new light. Indeed, these categories admit ``decompositions" \[
\NH\rightarrow \LH(\FLCA)\rightarrow \FLCA, \qquad \FLCA\rightarrow \RH(\FLCA)\rightarrow \NH[-1]
\] and the Pontryagin duality induces equivalences of categories $\NH^{\mathrm{op}}\overset{\sim}{\rightarrow} \NH[-1]$ and $\FLCA^{\mathrm{op}}\overset{\sim}{\rightarrow}\FLCA$. 

We are ready to prove \cref{dualitynhlca}.
\begin{proof}[Proof of \cref{dualitynhlca}]
The perfect pairing induces an equivalence \[
C\overset{\sim}{\rightarrow} R\underline{\Hom}(D,\R/\Z)
\] in $\DDD^{\mathrm{b}}(\FLCA)$. We have a spectral sequence for the cohomology of $R\underline{\Hom}(D,\R/\Z)$ \[
E_2^{i,j}=\underline{\Ext}^i(\H^{-j}(D),\R/\Z)\implies \H^{i+j}(R\underline{\Hom}(D,\R/\Z)).
\] Since we have $\H^{-j}(D)\in\LH(\FLCA)$ for all $j$, the term $E_2^{i,j}$ vanishes for all $i\neq 0,1$. Consequently, for all $n$, we have an exact sequence in $\LH(\FLCA)$ \begin{equation}\label{eqn:1s1}
0\rightarrow \underline{\Ext}(\H^{-n+1}(D),\R/\Z)\rightarrow \H^n(R\underline{\Hom}(D,\R/\Z))\rightarrow \underline{\Hom}(\H^{-n}(D),\R/\Z)\rightarrow 0.
\end{equation} We observe that we have \[\begin{split}\underline{\Ext}(\H^{-n+1}(D),\R/\Z)&=\underline{\Ext}(\H^{-n+1}(D)^{\mathrm{nh}},\R/\Z)=(\H^{-n+1}(D)^{\mathrm{nh}})^{\vee}[1]\in\NH, \\ \underline{\Hom}(\H^{-n}(D),\R/\Z)&=\underline{\Hom}(\H^{-n}(D)^{\mathrm{lca}},\R/\Z)=(\H^{-n}(D)^{\mathrm{lca}})^{\vee}\in\FLCA.\end{split}\] We have for all $n$  \[\H^n(C)\cong \H^n(R\underline{\Hom}(D,\R/\Z)). \] By uniqueness of the nh-lca decomposition, the exact sequence \eqref{eqn:1s1} induces isomorphisms \[
\H^n(C)^{\mathrm{nh}}\cong(\H^{-n+1}(D)^{\mathrm{nh}})^{\vee}[1]\cong (\H^{-n+1}(D)^{\mathrm{nh}}[-1])^{\vee}, \qquad \H^{n}(C)^{\mathrm{lca}}\cong (\H^{-n}(D)^{\mathrm{lca}})^{\vee}
\] for all $n$. \end{proof}


\subsection{Duality with coefficients in tori}\label{subsection:dualitytori}
Let $T$ be a torus over $F$ and let $Y^*$ be its Cartier dual. The bilinear map in $\Ab(B_{\hat{W}_F})$ introduced in  \cref{example:charactertorus} \[
T(\overline{L})_{\mathrm{cond}}\times X^*(T)\rightarrow \overline{L}^{\times}.
\] induces a cup-product pairing \begin{equation} 
\label{eqn:cupproductpairingtori}
R\GGamma(B_{\hat{W}_F},T(\overline{L})_{\mathrm{cond}})\otimes^L R\GGamma(B_{\hat{W}_F},X^*(T))\rightarrow \HH^1(B_{\hat{W}_F},\overline{L}^{\times})[-1]=\Z[-1].
\end{equation}
Thanks to \cite[Theorem 4.27]{Artusa}, we prove the following
\begin{prop}[Condensed duality for tori]\label{tatenakayama}
The cup-product pairing \eqref{eqn:cupproductpairingtori} is a perfect pairing in $\DDD^{\mathrm{b}}(\FLCA)$. 
\end{prop}
 In order to proceed with the proof, we need two preliminary results
\begin{lem}\label{lem:doubledualcoeffz}
The canonical morphism \[
\varphi_{X^*(T)}:R\GGamma(B_{\hat{W}_F},X^*(T))\rightarrow R\underline{\Hom}(R\underline{\Hom}(R\GGamma(B_{\hat{W}_F},X^*(T)),\Z),\Z)
\] is an equivalence in $\DDD^{\mathrm{b}}(\CC)$.
\end{lem}
\begin{proof}
We have a fiber sequence \[
\tau^{\le 1}R\GGamma(B_{\hat{W}_F},X^*(T))\rightarrow R\GGamma(B_{\hat{W}_F},X^*(T))\rightarrow \HH^2(B_{\hat{W}_F},X^*(T))[-2]
\] in $\DDD^{\mathrm{b}}(\CC)$. By \cite[Theorem 3.30]{Artusa}, we have \[\tau^{\le 1}R\GGamma(B_{\hat{W}_F},X^*(T))\in\DDD^{perf}(\Z)\] and \[\HH^2(B_{\hat{W}_F},X^*(T))=(\Q_p/\Z_p)^n\oplus H\] for some $n\in\N$ and some finite abelian group $H$. Consequently, we have equivalences \[
\tau^{\le 1}R\GGamma(B_{\hat{W}_F},X^*(T))\overset{\sim}{\rightarrow} R\underline{\Hom}(R\underline{\Hom}(\tau^{\le 1}R\GGamma(B_{\hat{W}_F},X^*(T)),\Z),\Z)
\] and \[
\HH^2(B_{\hat{W}_F},X^*(T))\overset{\sim}{\rightarrow}R\underline{\Hom}(R\underline{\Hom}(\HH^2(B_{\hat{W}_F},X^*(T)),\Z),\Z).
\] The result follows.
\end{proof}

\begin{lem}\label{lem:cohomologytensorr}
Let $M$ be a free finitely generated abelian group with a continuous action of a finite quotient $G$ of $G_F$. Then $M$ is flat in $B_{\hat{W}_F}$ and the canonical morphism \[
R\GGamma(B_{\hat{W}_F},M)\otimes^L \R\rightarrow R\GGamma(B_{\hat{W}_F},M\otimes^L \R)
\] is an equivalence in $\DDD(\CC)$.
\end{lem}
\begin{rmk}
In the statement of \cref{lem:cohomologytensorr}, the derived tensor products are computed in $\DDD^{\mathrm{b}}(\CC)$.
 \end{rmk}
 \begin{rmk}\label{rmk:ontensorr}
If $C\in\DDD^{\mathrm{b}}(\FLCA)$ is such that $\H^i(C)$ is a discrete torsion abelian group for all $i$, then we have $C\otimes^L \R=0$ in $\DDD^{\mathrm{b}}(\CC)$. To show this, we reduce to the case where $C=\Z/n[0]$. We consider the flat resolution $\cdot n:\Z\to \Z$ of $\Z/n$. Then we have \[
 \Z/n\otimes^L \R=[\Z\otimes \R\overset{\cdot n}{\rightarrow} \Z\otimes\R],
 \] which is equivalent to the $0$ complex in $\DDD^{\mathrm{b}}(\CC)$.
 \end{rmk}
\begin{proof}
We first observe that we have $M_{|EG}\cong \Z^n$ for some $n\in\N$, which is flat in $B_{\hat{W}_F}/EG$. Consequently, $M$ is flat and we have $M\otimes^L \R=M\otimes \R$, which is a finitely generated $\R$-vector space with a continuous action of $G$. By \cite[Proposition 2.30]{Artusa} we have \[
R\GGamma(B_{\hat{W}_F},M\otimes^L \R)=R\GGamma(B_{W_k},(M\otimes \R)^H)
\] where $H$ is the image of the inertia subgroup via the quotient $G_F\to G$. 

\vspace{0.5em}
On the other hand, we have a fiber sequence in $\DDD^{\mathrm{b}}(\Ab(\CC))$ \[
R\GGamma(B_{W_k},M^H)\rightarrow R\GGamma(B_{\hat{W}_F},M)\rightarrow R\GGamma(B_{W_k},\tau^{\ge 1}R\GGamma(B_{\hat{I}},M)).
\] By \cite[Example 2.42]{Artusa}, we have \[R\GGamma(B_{W_k},M^H)=[M^H\overset{1-\varphi}{\longrightarrow} M^H]\in \DDD^{perf}(\Z)\subset \DDD^{\mathrm{b}}(\FLCA).\]  Moreover, by \cite[Theorem 4.27]{Artusa}, we have $R\GGamma(B_{\hat{W}_F},M)\in\DDD^{\mathrm{b}}(\FLCA)$. Consequently, we have \[
C\coloneqq R\GGamma(B_{W_k},\tau^{\ge 1}R\GGamma(B_{\hat{I}},M))\in \DDD^{\mathrm{b}}(\FLCA).
\] Moreover, since $\HH^q(B_{\hat{I}},M)=\underline{\H^q(B_I(\Set),M)}$ is torsion for all $q\ge 1$, by \cref{rmk:ontensorr}, we have \[
C\otimes^L \R=0
\] in $\DDD^{\mathrm{b}}(\CC)$. Consequently,
the map \[
R\GGamma(B_{W_k},M^H)\otimes^L \R\rightarrow R\GGamma(B_{\hat{W}_F},M)\otimes^L \R
\] is an equivalence in $\DDD^{\mathrm{b}}(\CC)$. Since the complex $R\GGamma(B_{W_k},M^H)$ belongs to $\DDD^{perf}(\Z)$, we have $R\GGamma(B_{W_k},M^H)\otimes^L \R=R\GGamma(B_{W_k},M^H)\otimes \R$. 

\vspace{0.5em} 
Therefore we need to check that the induced map \[
R\GGamma(B_{W_k},M^H)\otimes \R \rightarrow R\GGamma(B_{W_k},(M\otimes \R)^H)
\] is an equivalence. Both condensed abelian groups belong to $\DDD^{perf}(\R)$, hence it is enough to check that the morphism is an equivalence on the underlying abelian groups.
This reduces to the commutativity of the diagram of abelian groups \[
\begin{tikzcd}
M^H\otimes \R \ar[r,"(1-\varphi)\otimes \R"]\ar[d,"\sim"] & M^H\otimes \R\ar[d, "\sim"]\\
(M\otimes \R)^{H}\ar[r, "1-\varphi"] & (M\otimes \R)^{H},
\end{tikzcd}
\] where the horizontal arrows are\[
(1-\varphi)\otimes \R: M^H\otimes \R \rightarrow M^H\otimes \R, \qquad m\otimes r\mapsto (m-\varphi(m))\otimes r,
\]\[
1-\varphi: (M\otimes \R)^H\rightarrow (M\otimes \R)^H, \qquad m\otimes r\mapsto m\otimes r- \varphi(m\otimes r)= m\otimes r- \varphi(m)\otimes r.
\] This is clear.
\end{proof}
\begin{proof}[Proof of \cref{tatenakayama}]
Following the notations of \cite{Artusa}, we set \[\R/\Z(1)\coloneqq \mathrm{cofib}(\overline{L}^{\times}[-1]\overset{val}{\rightarrow} \R[-1])\in\DDD^{\mathrm{b}}(B_{\hat{W}_F}).\]By \cref{prop:dualtorus}, we have a fiber sequence \[ X^*(T)^D\rightarrow T(\overline{L})_{\mathrm{cond}}\rightarrow \underline{\Hom}(X^*(T),\R) \] in $\DDD^{\mathrm{b}}(B_{\hat{W}_F})$, where $X^*(T)^D\coloneqq R\underline{\Hom}(X(T),\R/\Z(1))$.  We get a morphism of fiber sequences
\[
\begin{tikzcd}
R\GGamma(B_{\hat{W}_F},X(T)^D)\ar[d,"\psi_{D}"]\ar[r] & *\ar[d]\ar[r]& R\GGamma(B_{\hat{W}_F}, \underline{\Hom}(X^*(T),\R))\ar[d,"\psi_{\R}"]\\
R\underline{\Hom}(R\GGamma(B_{\hat{W}_F},X(T)),\R/\Z[-2])\ar[r]&*\ar[r]& R\underline{\Hom}(R\GGamma(B_{\hat{W}_F},X^*(T)),\R[-1]),
\end{tikzcd}
\] where the middle map is the morphism \[
\psi_T:R\GGamma(B_{\hat{W}_F},T(\overline{L})_{\mathrm{cond}})\rightarrow R\underline{\Hom}(R\GGamma(B_{\hat{W}_F},X^*(T)),\Z[-1])
\] coming from \eqref{eqn:cupproductpairingtori}. To show that $\psi_T$ is an equivalence, it is enough to show that the same holds for $\psi_{D}$ and $\psi_{\R}$. The map $\psi_{D}$ is an equivalence by \cite[Theorem 4.27]{Artusa}. Concerning $\psi_{\R}$, we have \[\begin{split}
R\underline{\Hom}(R\GGamma(B_{\hat{W}_F},X^*(T)),\R[-1])&=R\underline{\Hom}(R\GGamma(B_{\hat{W}_F},X^*(T))\otimes^L \R,\R/\Z[-1])=\\&=R\underline{\Hom}(R\GGamma(B_{\hat{W}_F},X^*(T)\otimes^L \R),\R/\Z[-1]), \end{split}
\] where the last equality is \cref{lem:cohomologytensorr}. By \cite[Lemma 3.19]{Artusa} we have $R\underline{\Hom}(\R,\overline{L}^{\times})=0$, thus we obtain \[
(X(T)\otimes^L \R)^D[1]=R\underline{\Hom}(X^*(T),R\underline{\Hom}(\R,\R/\Z(1)))[1]=\underline{\Hom}(X^*(T),\R).
\] Consequently, the fact that $\psi_{\R}$ is an isomorphism follows from \cite[Theorem 4.27]{Artusa} with $M=X^*(T)\otimes^L \R$.

We just need to show that the other morphism induced from \cref{eqn:cupproductpairingtori}, i.e.\ \[
\psi_{X^*(T)}:R\GGamma(B_{\hat{W}_F},X^*(T))\rightarrow R\underline{\Hom}(R\GGamma(B_{\hat{W}_F},T(\overline{L})_{\mathrm{cond}}),\Z[-1])
\] is an equivalence. Up to equivalence, the map $\psi_{X^*(T)}$ is given by \[\begin{split}
R\GGamma(B_{\hat{W}_F},X^*(T))&\overset{\varphi_{X*(T)}}{\longrightarrow} R\underline{\Hom}(R\underline{\Hom}(R\GGamma(B_{\hat{W}_F},X^*(T)),\Z[-1]),\Z[-1])\overset{-\circ \psi_T}{\longrightarrow}\\
&\overset{-\circ \psi_T}{\longrightarrow}R\underline{\Hom}(R\GGamma(B_{\hat{W}_F},T(\overline{L})_{\mathrm{cond}}),\Z[-1]).
\end{split}
\] The second morphism is an equivalence since $\psi_T$ is. The first morphism is an equivalence by \cref{lem:doubledualcoeffz}. Consequently, the morphism $\psi_{X^*(T)}$ is an equivalence.

\vspace{0.5em}

\end{proof}

\subsection{Duality with coefficients in abelian varieties}\label{subsection:dualitycoefficientsabelianvarieties}
Let $A$ be an abelian variety over $F$, and let $A^*$ be the dual abelian variety. As explained in the appendix (\cref{example:poincareabelianvar}), the Poincaré biextension $\mathcal{P}_0$ of $(A,A^*)$ by $\mathbbm{G}_m$ induces a pairing $\psi_{\mathcal{P}_0}:A\otimes^L A^*\rightarrow \mathbbm{G}_m[1]$ in $\DDD^{\mathrm{b}}(\TT_{\mathrm{fppf}})$ and a cup-product pairing in $\DDD^{\mathrm{b}}(\Ab)$ \begin{equation}\label{noncondensedcppairing:av} R\Gamma(B_{G_F}(\Set),A(\overline{F}))\otimes^L R\Gamma(B_{G_F}(\Set),A^*(\overline{F}))\rightarrow \H^2(G_F(\Set),\overline{F}^{\times})[-1]=\Q/\Z[-1].
\end{equation} By \cref{condensedbiext}  $\mathcal{P}_0(\overline{L})_{\mathrm{cond}}$ is a biextension of $(A(\overline{L})_{\mathrm{cond}},A^*(\overline{L})_{\mathrm{cond}})$ by $\overline{L}^{\times}$, inducing a pairing in $\DDD^{\mathrm{b}}(B_{\hat{W_F}})$ \[\psi_{\mathcal{P}_0(\overline{L})_{\mathrm{cond}}}:A(\overline{L})_{\mathrm{cond}}\otimes^L A^*(\overline{L})_{\mathrm{cond}}\rightarrow \overline{L}^{\times}[1]\] and a condensed cup-product pairing in $\DD^{\mathrm{b}}(\Ab(\CC))$ \begin{equation}\label{condensedcppairing:av}
R\GGamma(B_{\hat{W}_F},A(\overline{L})_{\mathrm{cond}})\otimes^L R\GGamma(B_{\hat{W}_F},A^*(\overline{L})_{\mathrm{cond}})\rightarrow \HH^1(B_{\hat{W_F}},\overline{L}^{\times})[0]=\Z.
\end{equation} Moreover, the underlying pairing in $\DDD^{\mathrm{b}}(\Ab)$ is compatible with \eqref{noncondensedcppairing:av}.

We prove the Weil version of the Tate duality with coefficients in abelian varieties.
\begin{thm}\label{thm:condenseddualityav}(Tate duality for abelian varieties)
The cup-product pairing \eqref{condensedcppairing:av} is a perfect pairing in $\DDD^{\mathrm{b}}(\FLCA)$ and factors through $\R/\Z[-1]$. In particular, for $q=0,1$, we get a perfect pairing of locally compact abelian group of finite ranks \[
\HH^q(B_{\hat{W}_F},A(\overline{L})_{\mathrm{cond}})\otimes \HH^{1-q}(B_{\hat{W}_F},A^*(\overline{L})_{\mathrm{cond}})\rightarrow \R/\Z.
\]\end{thm}
\begin{proof}
By \cref{thm:structurecohomologyabvar}, we have $R\underline{\Hom}(R\GGamma(B_{\hat{W}_F},A(\overline{L})_{\mathrm{cond}}),\R)=0$, which implies that the map of \eqref{condensedcppairing:av} factors through $\R/\Z[-1]$. Let us consider the map \[
\tau(A):R\GGamma(B_{\hat{W}_F},A^*(\overline{L})_{\mathrm{cond}})\rightarrow R\underline{\Hom}(R\GGamma(B_{\hat{W}_F},A(\overline{L})_{\mathrm{cond}}),\R/\Z[-1]).
\] 
Since the cohomology groups of the complexes $R\GGamma(B_{\hat{W}_F},A^*(\overline{L})_{\mathrm{cond}})$ and $R\GGamma(B_{\hat{W}_F},A(\overline{L})_{\mathrm{cond}})$ are either profinite or torsion, we observe that \begin{itemize} \item $\tau(A)$ is an equivalence if and only if $\tau(A)(*)$ is an equivalence in $\DDD^{\mathrm{b}}(\Ab)$;
\item We have a canonical isomorphism \[R\Hom_{\Ab(\CC)}(R\GGamma(B_{\hat{W}_F},A(\overline{L})_{\mathrm{cond}}),\R/\Z)\cong R\Hom_{\Ab}(R\Gamma(B_{\hat{W}_F},A(\overline{L})_{\mathrm{cond}}),\Q/\Z) .\]
\end{itemize} Thus it is enough to show that the pairing \[
R\Gamma(B_{\hat{W}_F},A(\overline{L})_{\mathrm{cond}})\otimes^L R\Gamma(B_{\hat{W}_F},A^*(\overline{L})_{\mathrm{cond}})\rightarrow \Q/\Z[-1]
\] is perfect.

By compatibility with \eqref{noncondensedcppairing:av} we have a commutative diagram \[
\begin{tikzcd}
R\Gamma(B_{G_F}(\Set),A(\overline{F}))\otimes^L R\Gamma(B_{G_F},A^*(\overline{F}))\ar[r]\ar[d]& \Q/\Z[-1]\ar[d,equal]\\
R\Gamma(B_{\hat{W}_F},A(\overline{L})_{\mathrm{cond}})\otimes^L R\Gamma(B_{\hat{W}_F},A^*(\overline{L})_{\mathrm{cond}})\ar[r]& \Q/\Z[-1].
\end{tikzcd}
\] By \cite[Lemma 5.4.3]{Karpuk} and \cref{prop:globalsectionscohomag} the left vertical arrow is an isomorphism. By Tate's duality theorem with coefficients in abelian varieties (see \cite[Corollary 3.4]{ADT}) the upper pairing is perfect. The result follows.
\end{proof} 
\subsection{Duality with coefficients in \texorpdfstring{$1$}{1}-motives}\label{subsection:duality1motives}
Let $\MM=[Y\overset{u}{\rightarrow} E]$ be a $1$-motive over $F$, where $E$ is an extension of the abelian variety $A$ by the torus $T$. In \cite[10.2.11]{Deligne} a dual $1$-motive $\MM^*=[Y^*\overset{u^*}{\rightarrow}E^*]$ is associated to $\MM$, where $Y^*\coloneqq \mathcal{X}(T)$ is the Cartier dual of $T$, while $E^*$ is a semiabelian variety, extension of the dual abelian variety $A^*$ of $A$ by the torus $T^*$ with Cartier dual $Y$. A Poincaré biextension $\mathcal{P}$ of $(\MM,\MM^*)$ by $\mathbbm{G}_m$ is defined. This gives a pairing in $\DDD^{\mathrm{b}}(\TT_{\mathrm{fppf}})$ \[
\psi_{\mathcal{P}}:\MM\otimes^L \MM^*\rightarrow \mathbbm{G}_m[1].
\] and a cup-product pairing for Galois cohomology \[
CP_{\mathcal{P}}:R\Gamma(B_{G_F}(\Set),\MM(\overline{F}))\otimes^L R\Gamma(B_{G_F}(\Set),\MM^*(\overline{F}))\rightarrow \H^2(B_{G_F}(\Set),\overline{F}^{\times})[-1]=\Q/\Z[-1].
\]
The construction of the dual $1$-motive, of the Poincaré pairing, as well as the compatibility with the weight filtration on $\MM$ and $\MM^*$ are recalled  in \cref{example:poincarebiext}. 
Moreover, in \cref{condensedbiextensions1motives} we show that $\mathcal{P}(\overline{L})_{\mathrm{cond}}$ is a biextension of $(\MM(\overline{L})_{\mathrm{cond}},\MM^*(\overline{L})_{\mathrm{cond}})$ by $\overline{L}^{\times}$, called \emph{condensed Poincaré biextension}. We obtain a condensed Poincaré pairing \[
\psi_{\mathcal{P}(\overline{L})_{\mathrm{cond}}}:\MM(\overline{L})_{\mathrm{cond}}\otimes^L \MM^*(\overline{L})_{\mathrm{cond}}\rightarrow \overline{L}^{\times}[1],
\] and a condensed cup-product pairing 
\begin{equation}\label{eqn:condensedcppairing1motives}
   CP_{\mathcal{P}(\overline{L})_{\mathrm{cond}}}: R\GGamma(B_{\hat{W}_F},\MM(\overline{L})_{\mathrm{cond}})\otimes^L R\GGamma(B_{\hat{W}_F},\MM^*(\overline{L})_{\mathrm{cond}})\rightarrow \HH^1(B_{\hat{W}_F},\overline{L}^{\times})=\Z.
\end{equation}
 which are compatible with the condensed weight filtration as proved in \cref{lem:lastlemma}. 
We are ready to prove the following \begin{thm}[Condensed duality with coefficients in $1$-motives]\label{thm:duality1motives1}
The cup-product pairing \eqref{eqn:condensedcppairing1motives} is a perfect pairing in $\DDD^{\mathrm{b}}(\FLCA)$.
\end{thm}
\begin{proof}
By \cref{structuremotives}, we have $R\GGamma(B_{\hat{W}_F},\MM(\overline{L})_{\mathrm{cond}})\in \DDD^{\mathrm{b}}(\FLCA)$ and $R\GGamma(B_{\hat{W}_F},\MM^*(\overline{L})_{\mathrm{cond}})\in \DDD^{\mathrm{b}}(\FLCA)$. We need to show that the morphisms induced by the cup-product pairing $CP_{\mathcal{P}(\overline{L})_{\mathrm{cond}}}$ \[
R\GGamma(B_{\hat{W}_F},\MM(\overline{L})_{\mathrm{cond}})\overset{\alpha}{\rightarrow} R\underline{\Hom}(R\GGamma(B_{\hat{W}_F},\MM^*(\overline{L})_{\mathrm{cond}}),\Z) \] \[ R\GGamma(B_{\hat{W}_F},\MM^*(\overline{L})_{\mathrm{cond}})\overset{\beta}{\rightarrow} R\underline{\Hom}(R\GGamma(B_{\hat{W}_F},\MM(\overline{L})_{\mathrm{cond}}),\Z)
\] are equivalences. 

\vspace{0.5em}
\emph{Reduction to the case $T=0$.} We have fiber sequences \[
T(\overline{L})_{\mathrm{cond}}\rightarrow \MM(\overline{L})_{\mathrm{cond}}\rightarrow \MM/T(\overline{L})_{\mathrm{cond}}, \qquad E^*(\overline{L})_{\mathrm{cond}}\rightarrow \MM^*(\overline{L})_{\mathrm{cond}}\rightarrow Y^*(\overline{L})_{\mathrm{cond}}[1].
\] We denote by $\mathcal{P'}$ the Poincaré biextension of $(\MM/T,E^*)$ by $\mathbbm{G}_m$. By \cref{lem:lastlemma} a) and b) the map $\alpha$ induces a morphism of fiber sequences 
\[
\begin{tikzcd}
R\GGamma(B_{\hat{W}_F},T(\overline{L})_{\mathrm{cond}})\ar[r]\arrow[d,"\alpha'"]
&*\ar[r]\arrow[d,"\alpha"]
&R\GGamma(B_{\hat{W}_F},\MM/T(\overline{L})_{\mathrm{cond}})\arrow[d,"\alpha''"] 
\\
R\underline{\Hom}(R\GGamma(B_{\hat{W}_F},Y^*(\overline{L})_{\mathrm{cond}})[1],\Z)\ar[r]
&*\ar[r]
&R\underline{\Hom}(R\GGamma(B_{\hat{W}_F},E^*(\overline{L})_{\mathrm{cond}}),\Z)
\end{tikzcd}
\] where $\alpha'$, resp.\ $\alpha''$, is the morphism induced by the cup-product pairing \eqref{eqn:cupproductpairingtori}, resp.\ $CP_{\mathcal{P}'(\overline{L})_{\mathrm{cond}}}$.
 The morphism $\alpha$ is an equivalence if and only if $\alpha'$ and $\alpha''$ are. The same holds for $\beta$. Consequently, to show that $CP_{\mathcal{P}(\overline{L})_{\mathrm{cond}}}$ is a perfect pairing, it is enough to show that the cup-product pairings \eqref{eqn:cupproductpairingtori} and 
\[
CP_{\mathcal{P}'(\overline{L})_{\mathrm{cond}}}: R\GGamma(B_{\hat{W}_F},\MM/T(\overline{L})_{\mathrm{cond}})\otimes^L R\GGamma(B_{\hat{W}_F},E^*(\overline{L})_{\mathrm{cond}})\rightarrow \Z
\] are perfect. The first one is perfect by \cref{tatenakayama}. The only thing left to show is that $CP_{\mathcal{P}'(\overline{L})_{\mathrm{cond}}}$ is perfect. 

\vspace{0.5em}
\emph{The case $T=0$.} We reduced to the case $T=0$. We have fiber sequences \[
A(\overline{L})_{\mathrm{cond}}\rightarrow \MM(\overline{L})_{\mathrm{cond}}\rightarrow Y(\overline{L})_{\mathrm{cond}}[1], \qquad T^*(\overline{L})_{\mathrm{cond}}\rightarrow \MM^*(\overline{L})_{\mathrm{cond}}\rightarrow A^*(\overline{L})_{\mathrm{cond}}.
\] By \cref{lem:lastlemma} i) and ii) $\alpha$ induces a morphism of fiber sequences 
\[
\begin{tikzcd}
R\GGamma(B_{\hat{W}_F},A(\overline{L})_{\mathrm{cond}})\ar[r]\arrow[d,"\alpha'"]
&*\ar[r]\arrow[d,"\alpha"]
&R\GGamma(B_{\hat{W}_F},Y(\overline{L})_{\mathrm{cond}})[1]\arrow[d,"\alpha''"] 
\\
R\underline{\Hom}(R\GGamma(B_{\hat{W}_F},A^*(\overline{L})_{\mathrm{cond}}),\Z)\ar[r]
&*\ar[r]
&R\underline{\Hom}(R\GGamma(B_{\hat{W}_F},T^*(\overline{L})_{\mathrm{cond}}),\Z)
\end{tikzcd}
\] where $\alpha'$ and $\alpha''$ are induced by the cup-product pairings \eqref{condensedcppairing:av} and \eqref{eqn:cupproductpairingtori} (and the same for $\beta$). As before, it is enough to show that the cup-product pairings \eqref{condensedcppairing:av} and \eqref{eqn:cupproductpairingtori}, i.e.\ \[
R\GGamma(B_{\hat{W}_F},A(\overline{L})_{\mathrm{cond}})\otimes^L R\GGamma(B_{\hat{W}_F},A^*(\overline{L})_{\mathrm{cond}})\rightarrow \Z
\] and \[
R\GGamma(B_{\hat{W}_F},Y(\overline{L})_{\mathrm{cond}})\otimes^L R\GGamma(B_{\hat{W}_F},T^*(\overline{L})_{\mathrm{cond}})\rightarrow \Z[-1],
\] are perfect. The first fact follows by \cref{thm:condenseddualityav}, while the second is again \cref{tatenakayama}.
\end{proof}
To express the duality as a Pontryagin duality on the cohomology groups, we use the $\R/\Z$-twist defined in \cref{section:rztwist}. Using this definition and \cref{dualitynhlca} we reformulate the duality as follows
\begin{thm}\label{duality1motives2}
Let $\MM$ be a $1$-motive over $F$ with dual $1$-motive $\MM^*$. We have a perfect cup-product pairing in $\DDD^{\mathrm{b}}(\FLCA)$ \begin{equation}\label{eqn:duality1motives2}
R\GGamma(B_{\hat{W}_F},\MM(\overline{L})_{\mathrm{cond}})\otimes^LR\GGamma(B_{\hat{W}_F},\MM^*(\overline{L})_{\mathrm{cond}};\R/\Z)\rightarrow \R/\Z,
\end{equation} which induces perfect cup-product pairings in $\FLCA$ \[\begin{split}
\HH^{-1}(B_{\hat{W}_F},\MM(\overline{L})_{\mathrm{cond}})\otimes \HH^{1}(B_{\hat{W}_F},\MM^*(\overline{L})_{\mathrm{cond}};\R/\Z)^{\mathrm{lca}}&\rightarrow \R/\Z, \\
\HH^{0}(B_{\hat{W}_F},\MM(\overline{L})_{\mathrm{cond}})^{\mathrm{lca}}\otimes \HH^{0}(B_{\hat{W}_F},\MM^*(\overline{L})_{\mathrm{cond}};\R/\Z)&\rightarrow \R/\Z,\\
\HH^{1}(B_{\hat{W}_F},\MM(\overline{L})_{\mathrm{cond}})\otimes \HH^{-1}(B_{\hat{W}_F},\MM^*(\overline{L})_{\mathrm{cond}};\R/\Z)&\rightarrow \R/\Z\\
\end{split}
\] and a perfect pairing in $\DDD^{\mathrm{b}}(\FLCA)$ \[
\HH^0(B_{\hat{W}_F},\MM(\overline{L})_{\mathrm{cond}})^{\mathrm{nh}}\otimes^L \HH^1(B_{\hat{W}_F},\MM^*(\overline{L})_{\mathrm{cond}};\R/\Z)^{\mathrm{nh}}[-1]\rightarrow \R/\Z.
\]
\end{thm}
\begin{proof}
The complex $R\GGamma(B_{\hat{W}_F},\MM^*(\overline{L})_{\mathrm{cond}};\R/\Z)$ is in $\DDD^{\mathrm{b}}(\FLCA)$, hence we have \[\begin{split}
&R\underline{\Hom}_{\Ab(\CC)}(R\GGamma(B_{\hat{W}_F},\MM^*(\overline{L})_{\mathrm{cond}};\R/\Z),\R/\Z)\cong \\ \cong &R\underline{\Hom}_{\FLCA}(R\GGamma(B_{\hat{W}_F},\MM^*(\overline{L})_{\mathrm{cond}};\R/\Z),\R/\Z)\cong\\ \cong & R\underline{\Hom}_{\FLCA}(R\GGamma(B_{\hat{W}_F},\MM^*(\overline{L})_{\mathrm{cond}}),\Z)\cong R\underline{\Hom}_{\Ab(\CC)}(R\GGamma(B_{\hat{W}_F},\MM^*(\overline{L})_{\mathrm{cond}}),\Z).\end{split}
\] Therefore, by \cref{thm:duality1motives1}, the pairing \eqref{eqn:condensedcppairing1motives} induces an equivalence \[
\gamma:R\GGamma(B_{\hat{W}_F},\MM(\overline{L})_{\mathrm{cond}})\rightarrow R\underline{\Hom}(R\GGamma(B_{\hat{W}_F},\MM^*(\overline{L})_{\mathrm{cond}};\R/\Z),\R/\Z)
\] in $\DDD^{\mathrm{b}}(\Ab(\CC))$. Let us consider the cup-product pairing \[
cp_{\gamma}:R\GGamma(B_{\hat{W}_F},\MM(\overline{L})_{\mathrm{cond}})\otimes^L R\GGamma(B_{\hat{W}_F},\MM^*(\overline{L})_{\mathrm{cond}};\R/\Z)\rightarrow \R/\Z
\] in $\DDD^{\mathrm{b}}(\Ab(\CC))$ determined by this equivalence. This cup-product pairing is perfect. To see this, let us denote \[
A\coloneqq R\GGamma(B_{\hat{W}_F},\MM(\overline{L})_{\mathrm{cond}}), \qquad B\coloneqq R\GGamma(B_{\hat{W}_F},\MM^*(\overline{L})_{\mathrm{cond}};\R/\Z).
\] Then $cp_{\gamma}$ induces morphisms \[
\gamma_A:A\rightarrow R\underline{\Hom}(B,\R/\Z), \qquad \gamma_B:B\rightarrow R\underline{\Hom}(A,\R/\Z).
\] The morphism $\gamma_A$ tautologically coincides with the morphism $\gamma$, which is an equivalence. Moreover, the morphism $\gamma_B$ coincides, up to equivalence, to the composition \[
B\longrightarrow R\underline{\Hom}(R\underline{\Hom}(B,\R/\Z),\R/\Z)\overset{-\circ \gamma}{\longrightarrow} R\underline{\Hom}(A,\R/\Z).
\] The second morphism is an equivalence since $\gamma$ is, while the first morphism is an equivalence since $B\in\DDD^{\mathrm{b}}(\FLCA)$ is dualisable for the Pontryagin duality. Consequently, $\gamma_B$ is an equivalence, and the cup-product pairing $cp_{\gamma}$ is a perfect pairing in $\DDD^{\mathrm{b}}(\Ab(\CC))$ and in $\DDD^{\mathrm{b}}(\FLCA)$.  

\vspace{0.5em}

The rest of the proof follows from the structure result on $R\GGamma(B_{\hat{W}_F},\MM(\overline{L})_{\mathrm{cond}})$, i.e. \cref{structuremotives}, and from \cref{dualitynhlca}.

\end{proof}

\subsection{Comparison with Galois cohomology}
We show how to deduce the first statement of \cite[Theorem 0.1]{HarariSzamuely} from \cref{duality1motives2} in the $p$-adic field case. To do this, we compare the Weil $\R/\Z$-twisted cohomology $R\GGamma(B_{\hat{W}_F},\MM^*(\overline{L})_{\mathrm{cond}};\R/\Z)$ with the Galois cohomology $R\GGamma(B_{\hat{G}_F},\MM^*(\overline{F})_{\mathrm{cond,\acute{e}t}})$. 

\vspace{0.5em}
We consider the condensed étale realisation and the condensed Weil-étale realisation of \cref{section:realisation} \[
-(\overline{F})_{cond,\acute{e}t}:\Sch^{\mathrm{lft}}_F\rightarrow B_{\hat{G}_F}, \qquad -(\overline{L})_{\mathrm{cond}}:\Sch^{\mathrm{lft}}_F\rightarrow B_{\hat{W}_F}.
\]The morphism of topoi $f:B_{\hat{W}_F}\rightarrow B_{\hat{G}_F}$ induces a natural transformation \[
\eta:f^* \circ -(\overline{F})_{cond,\acute{e}t}\implies -(\overline{L})_{\mathrm{cond}}
\] of functors $\Sch^{\mathrm{lft}}_{F}\to B_{\hat{W}_F}$. This natural transormation is not an isomorphism in general. For example, we have $f^*\mathbbm{G}_m(\overline{F})_{cond,\acute{e}t}=\overline{F}^{\times}$ together with its $W_F$ action induced by $G_F$ and $\mathbbm{G}_m(\overline{L})_{\mathrm{cond}}=\overline{L}^{\times}$. However, $\eta_X$ is an isomorphism in some cases.
\begin{ex}
Let $X\in\Sch^{\mathrm{lft}}_F$ be a scheme which is locally constant for the étale topology, i.e.\ there exists $F'$ a finite extension of $F$ such that $X_{F'}$ is isomorphic to the constant scheme $M_F$ for some abelian group $M$. Then the map \[
\eta_{X}:f^*X(\overline{F})_{\acute{e}t,cond}\rightarrow X(\overline{L})_{\mathrm{cond}}
\] is an isomorphism in $B_{\hat{W}_F}$, and both condensed $\hat{W}_F$-modules coincide with the discrete abelian group $X(\overline{F})=X(\overline{L})=M$ with its action of $G_F$ which factors through $\Gal(F'/F)$ (see \cref{example:locallyconstantrealisation}).
\end{ex}

Let $G\in\Sch^{\mathrm{lft}}_F$ be a commutative group scheme over $F$ which is locally of finite type. Composing the canonical map \[R\GGamma(B_{\hat{G}_F},G(\overline{F})_{cond,\acute{e}t})\rightarrow R\GGamma(B_{\hat{W}_F},f^*G(\overline{F})_{cond,\acute{e}t})\] with the map induced by $\eta_{G}$, we obtain a canonical morphism \[
R\GGamma(B_{\hat{G}_F},G(\overline{F})_{cond,\acute{e}t})\rightarrow R\GGamma(B_{\hat{W}_F},G(\overline{L})_{\mathrm{cond}}).
\] The same holds if we replace $G$ by a $1$-motive $\MM$.

Before studying the properties of this morphism, we analyse the structure of the condensed cohomology groups $\HH^q(B_{\hat{G}_F},G(\overline{F})_{cond,\acute{e}t})$ in the case where $G$ is either an algebraic group or a commutative group scheme which is étale locally constant.
\begin{prop}\label{condensedgaloiscohomology}
\begin{enumerate}[1)]
\item Let $G$ be a commutative group scheme locally of finite type over $F$. The canonical map \[
\H^q(B_{G_F}(\Set),G(\overline{F}))\rightarrow \H^q(B_{\hat{G}_F},G(\overline{F})_{cond,\acute{e}t})
\] is an isomorphism of abelian groups for all $q$.
\item If $E$ is a semiabelian variety, the condensed abelian group $\HH^q(B_{\hat{G}_F},E(\overline{F})_{cond,\acute{e}t})$ is a discrete torsion abelian group for all $q\ge 1$. Consequently, it coincides with the discrete abelian group $\H^q(B_{G_F}(\Set),E(\overline{F}))$ for all $q\ge 1$.
\item If $Y$ is étale locally constant, the condensed abelian group $\HH^q(B_{\hat{G}_F},Y(\overline{F})_{cond,\acute{e}t})$ is discrete for all $q\ge 0$ and torsion for all $q\ge 1$. Consequently, it coincides with the discrete abelian group $\H^q(B_{G_F}(\Set),Y(\overline{F}))$ for all $q$.
\item If $\MM$ is a $1$-motive over $F$ and $q\ge 1$, the condensed abelian group $\HH^q(B_{\hat{G}_F},\MM(\overline{F})_{cond,\acute{e}t})$ is discrete torsion, and coincides with $\H^q(B_{G_F}(\Set),\MM(\overline{F}))$.
\end{enumerate}
\end{prop}
\begin{proof}
The proof of 1) is exactly the same of \cref{prop:globalsectionscohomag}.

\vspace{0.5em}
The proof of 2) is an easier analogue of the proof of \cref{thm:structurecohomologyabvar}. In particular, the result is the consequence of the following facts \begin{enumerate}[(i)]
    \item For all $q$, the condensed abelian group $\leftin{m}{\HH^q(B_{\hat{G}_F},E(\overline{F})_{cond,\acute{e}t})}$ is discrete and finite.
    \item For all $q\ge 1$, the abelian group $\HH^q(B_{\hat{G}_F},E(\overline{F})_{cond,\acute{e}t})(S)$ is torsion for all $S$ extremally disconnected.
\end{enumerate}
To prove (i), we observe that for all $m\in\N$ we have an exact sequence in $\mathbf{Gp}_F$ \[
0\rightarrow E_m\rightarrow E\rightarrow E\rightarrow 0,
\] where $E_m$ is a finite group scheme. By \cref{prodiscretesemiabelianvarieties,thm:exactrealisation} we have \[
E(\overline{F})_{cond,\acute{e}t}\otimes^L \Z/m=E_m(\overline{F})_{cond,\acute{e}t}[1],
\] which is a finite Galois module. The Kummer sequence \[
0\rightarrow \HH^{q-1}(B_{\hat{G}_F},E(\overline{F})_{cond,\acute{e}t})\rightarrow \HH^q(B_{\hat{G}_F},E_m(\overline{F})_{cond,\acute{e}t})\rightarrow \leftin{m}{\HH^{q}(B_{\hat{G}_F},E(\overline{F})_{cond,\acute{e}t})}\rightarrow 0
\] implies that $\leftin{m}{\HH^q(B_{\hat{G}_F},E(\overline{F})_{cond,\acute{e}t})}$ is finite for all $q$. 

To prove (ii), it is enough to show that for all $F'$ finite extension of $F$, the abelian group $\HH^q(B_{\Gal(F'/F)},\underline{E(F')^{\mathrm{top}}})$ is torsion. By \cite[Proposition 2.41]{Artusa} we have \[
\HH^q(B_{\Gal(F'/F)},\underline{E(F')^{\mathrm{top}}})(S)=\H^q(B_{\Gal(F'/F)}(\Set),\Cont(S,E(F')^{\mathrm{top}})).
\] The result follows since higher cohomology groups of the finite group $\Gal(F'/F)$ are killed by the order of $\Gal(F'/F)$.

\vspace{0.5em}
The proof of 3) follows from \cite[Remark 3.22]{Artusa} with $I$ replaced by $G_F$. Item 4) follows from 1), 2) and 3).
\end{proof}

We analyse the relation between the condensed Galois cohomology and the condensed Weil cohomology of $1$-motives.
\begin{lem}\label{lem:galoisweilcomparison1motives}
Let $\MM=[Y\to E]$ be a $1$-motive over $F$. Then for all $m\in\N$ the canonical morphism \[
R\GGamma(B_{\hat{G}_F},\MM(\overline{F})_{cond,\acute{e}t})\otimes^L\Z/m\rightarrow R\GGamma(B_{\hat{W}_F},\MM(\overline{L})_{\mathrm{cond}})\otimes^L \Z/m
\] is an equivalence of complexes with finite cohomology groups in $\DDD^{\mathrm{b}}(\CC)$.Consequently, we obtain an equivalence \[
R\GGamma(B_{\hat{G}_F},\MM(\overline{F})_{cond,\acute{e}t})\otimes^L \Q/\Z\rightarrow R\GGamma(B_{\hat{W}_F},\MM(\overline{L})_{\mathrm{cond}})\otimes^L \Q/\Z 
\] in $\DDD^{\mathrm{b}}(\LCA)$.
\end{lem}
\begin{proof}
Since we have
\[
\MM(\overline{F})_{cond,\acute{e}t}=[Y(\overline{F})_{cond,\acute{e}t}\rightarrow E(\overline{F})_{cond,\acute{e}t}],\qquad \MM(\overline{L})_{\mathrm{cond}}=[Y(\overline{L})_{\mathrm{cond}}\rightarrow E(\overline{L})_{\mathrm{cond}}],
\] it is enough to show the result for $\MM=[Y\to 0]$ and for $\MM=[0\to E]$. 

\vspace{0.5 em}
Let $\MM=[Y\to 0]$. The morphism 
$f^*Y(\overline{F})_{cond,\acute{e}t}\to Y(\overline{L})_{\mathrm{cond}}$ is an isomorphism since $Y$ is étale locally constant. Consequently, it is enough to show that the canonical morphism \[
R\GGamma(B_{\hat{G}_F},M\otimes^L \Z/m)\rightarrow R\GGamma(B_{\hat{W}_F},f^*(M\otimes^L \Z/m)) 
\] is an equivalence of complexes with finite cohomology groups for $M=Y(\overline{F})_{cond,\acute{e}t}$. Since $M$ is free and finitely generated with a continuous action of a finite quotient of $G_F$, then $M\otimes^L \Z/m$ is a finite $G_F$-module, and the result follows from \cite[Proposition 3.26]{Artusa}.

\vspace{0.5em}
Let $\MM=[0\to E]$. 
As we already observed, we have \[
E(\overline{F})_{cond,\acute{e}t}\otimes^L\Z/m=E_m(\overline{F})_{cond,\acute{e}t}[1]\] and \[ E(\overline{L})_{\mathrm{cond}}\otimes^L\Z/m=E_m(\overline{L})_{\mathrm{cond}}[1].
\] Consequently, it is enough to show that the canonical morphism \[
R\GGamma(B_{\hat{G}_F},E_m(\overline{F})_{cond,\acute{e}t})\rightarrow R\GGamma(B_{\hat{W}_F},E_m(\overline{L})_{\mathrm{cond}})
\] is an equivalence of complexes with finite cohomology groups. Again, this follows from the fact that $E_m$ is étale locally constant and from \cite[Proposition 3.26]{Artusa}.
\end{proof}

\begin{cor}\label{comparisongaloisweil1motives}
Let $\MM=[Y\to E]$ be a $1$-motive over $F$. For all $m\in\N$, we have a canonical equivalence \[
R\GGamma(B_{\hat{W}_F},\MM(\overline{L})_{\mathrm{cond}};\R/\Z)\otimes^L \Z/m \overset{\sim}{\rightarrow} R\GGamma(B_{\hat{G}_F},\MM(\overline{F})_{cond,\acute{e}t})\otimes^L \Z/m[1]
\] of complexes with finite cohomology groups in $\DDD^{\mathrm{b}}(\CC)$. Consequently, we obtain an equivalence \[
R\GGamma(B_{\hat{W}_F},\MM(\overline{L})_{\mathrm{cond}};\R/\Z)\otimes^L \Q/\Z \rightarrow R\GGamma(B_{\hat{G}_F},\MM(\overline{F})_{cond,\acute{e}t})\otimes^L \Q/\Z[1]
\] in $\DDD^{\mathrm{b}}(\LCA)$.
\end{cor}
\begin{proof}
We remark that for every $M\in\DDD^{\mathrm{b}}(\FLCA)$ we have \[
M\otimes^L_{\Ab(\CC)}\Z/m=M\otimes^L_{\FLCA}\Z/m=\mathrm{cofib}(M\overset{\cdot m}{\rightarrow}M).
\] Consequently, we have \[
\begin{split}
R\GGamma(B_{\hat{W}_F},\MM(\overline{L})_{\mathrm{cond}};\R/\Z)\otimes^L \Z/m&= R\GGamma(B_{\hat{W}_F},\MM(\overline{L})_{\mathrm{cond}})\otimes^L_{\FLCA} (\R/\Z\otimes^L_{\FLCA} \Z/m)=\\&=R\GGamma(B_{\hat{W}_F},\MM(\overline{L})_{\mathrm{cond}})\otimes^L \Z/m[1].
\end{split}
\] The result now follows from \cref{lem:galoisweilcomparison1motives}.
\end{proof}
\begin{rmk}
We recall that if $C\in \DDD^{\mathrm{b}}(\CC)$, we have a Kummer sequence \[
0\rightarrow \H^q(C)/m\rightarrow \H^q(C\otimes^L \Z/m) \rightarrow \leftin{m}{\H^{q+1}(C)}\rightarrow 0.\] If we take the filtered colimit along integers $m$, we obtain an exact sequence \begin{equation}\label{kummerqz}
    0\rightarrow \H^q(C)\otimes \Q/\Z \rightarrow \H^q(C\otimes^L \Q/\Z)\rightarrow \H^{q+1}(C)_{\mathrm{tors}}\rightarrow 0
\end{equation} If $\leftin{m}{\H^q(C)}$ is finite for all $m$, taking the cofiltered limit along $m$ gives an exact sequence (according to the diagram \cite[(4.5)]{Artusa})
\begin{equation}\label{kummerzhat}
0\rightarrow \H^q(C)^{\wedge}\rightarrow \H^q(C\otimes^L \hat{\Z})\rightarrow T\H^{q+1}(C)\rightarrow 0.
\end{equation}
\end{rmk}
\begin{rmk}
Let us consider the equivalence \[
R\GGamma(B_{\hat{W}_F},\MM(\overline{L})_{\mathrm{cond}};\R/\Z)\otimes^L \Q/\Z \rightarrow R\GGamma(B_{\hat{G}_F},\MM(\overline{F})_{cond,\acute{e}t})\otimes^L \Q/\Z[1].
\] By \cref{condensedgaloiscohomology}, the condensed abelian group $\HH^q(B_{\hat{G}_F},\MM(\overline{F})_{cond,\acute{e}t})$ are torsion for $q=1,2$. Consequently, using the exact sequence \eqref{kummerqz} we obtain an identification \[
\H^0(R\GGamma(B_{\hat{G}_F},\MM(\overline{F})_{cond,\acute{e}t})\otimes^L \Q/\Z[1])=\HH^2(B_{\hat{G}_F},\MM(\overline{F})_{cond,\acute{e}t}).
\] and an exact sequence
\begin{equation}\label{eqn:comparison021}
  \begin{split}  0&\rightarrow \HH^0(B_{\hat{W}_F},\MM(\overline{L})_{\mathrm{cond}};\R/\Z)\otimes \Q/\Z\rightarrow\\&\rightarrow \HH^2(B_{\hat{G}_F},\MM(\overline{F})_{cond,\acute{e}t})\rightarrow \HH^1(B_{\hat{W}_F},\MM(\overline{L})_{\mathrm{cond}};\R/\Z)_{\mathrm{tors}}\rightarrow 0.
  \end{split}
\end{equation}
\end{rmk}

We are ready to prove the following corollary to \cref{duality1motives2}
\begin{cor}\label{cor:harariszamuely1}
Let $\MM=[Y\overset{u}{\rightarrow} E]$ be a $1$-motive over $F$, and let $\MM^*=[Y^*\overset{u^*}{\rightarrow} E^*]$ be its Cartier dual. The perfect pairing \eqref{eqn:duality1motives2} induces a perfect cup-product pairing in $\FLCA$ \begin{equation}\label{eqn:topologisedhs0}
\HH^{-1}(F,\MM)_{\wedge}\otimes \HH^2(F,\MM^*)\rightarrow \R/\Z,
\end{equation} where we set \[\HH^{-1}(F,\MM)_{\wedge}\coloneqq \mathrm{Ker}((\underline{Y(F)^{\mathrm{top}}})^{\wedge}\rightarrow (\underline{E(F)^{\mathrm{top}}})^{\wedge})\] and \[
\HH^2(F,\MM^*)\coloneqq \HH^2(B_{\hat{G}_F},\MM^*(\overline{F})_{cond,\acute{e}t}).
\]
\end{cor}
\begin{proof} 
We set  \[ \begin{split}C&\coloneqq \mathrm{cofib}(\underline{Y(F)^{\mathrm{top}}}\rightarrow \underline{E(F)^{\mathrm{top}}}),\\ D&\coloneqq \mathrm{fib}(\HH^0(B_{\hat{W}_F},\MM^*/T^*(\overline{L})_{\mathrm{cond}};\R/\Z)\rightarrow \HH^1(B_{\hat{W}_F},T^*(\overline{L})_{\mathrm{cond}};\R/\Z))\end{split}\] in $\DDD^{\mathrm{b}}(\Ab(\CC))$. By \cref{duality1motives2}, we have $C^{\vee}=D$. Consequently, we have \[
C\otimes^L \hat{\Z}=(D\otimes^L \Q/\Z)^{\vee} [1].
\] Let us consider the exact sequences \eqref{kummerzhat}  \[
0\rightarrow \H^q(C)^{\wedge}\rightarrow \H^q(C\otimes^L \hat{\Z}) \rightarrow T\H^{q+1}(C)\rightarrow 0
\] and \eqref{kummerqz} \[
0\rightarrow \H^q(D)\otimes \Q/\Z \rightarrow \H^q(D\otimes^L \Q/\Z) \rightarrow \H^{q+1}(D)_{\mathrm{tors}}\rightarrow 0.
\] We observe that $\H^q(C)^{\wedge}$ and $T\H^{q+1}(C)$ are profinite, while $\H^q(D)\otimes \Q/\Z$ and $\H^{q+1}(D)_{\mathrm{tors}}$ are torsion abelian groups. In particular, they are locally compact. Since $\LCA$ is stable by extensions in $\Ab(\CC)$ (\cite[Proposition 4.5]{Artusa}), both condensed abelian groups $\H^q(D\otimes^L\Q/\Z)$ and $\H^q(C\otimes^L \hat{\Z})$ are locally compact. By \cref{dualitynhlca}, the equivalence \[
C\otimes^L \hat{\Z}=(D\otimes^L \Q/\Z)^{\vee} [1].\]
 gives a perfect cup-product pairing in $\LCA$ 

\[
\H^{-1}(C\otimes^L \hat{\Z})\otimes \H^0(D\otimes^L \Q/\Z)\rightarrow \R/\Z.
\]
We show that we have \[
\H^{-1}(C\otimes^L \hat{\Z})=\mathrm{Ker}((\underline{Y(F)^{\mathrm{top}}})^{\wedge}\rightarrow (\underline{E(F)^{\mathrm{top}}})^{\wedge}), \qquad \H^0(D\otimes^L \Q/\Z)=\HH^2(B_{\hat{G}_F},\MM(\overline{F})_{cond,\acute{e}t}).
\] 

Firstly, we have a fibre sequence in $\DD^{\mathrm{b}}(\CC)$ \[
\underline{Y(F)^{\mathrm{top}}}\otimes^L \hat{\Z}\rightarrow \underline{E(F)^{\mathrm{top}}}\otimes^L \hat{\Z}\rightarrow C\otimes^L \hat{\Z}.
\] By \eqref{kummerzhat} we observe that both $\underline{Y(F)^{\mathrm{top}}}\otimes^L\hat{\Z}$ and $\underline{E(F)^{\mathrm{top}}}\otimes^L \hat{\Z}$ are concentrated in cohomological degrees $0$ and $1$. The equality \[\H^{-1}(C\otimes^L\hat{\Z})=\mathrm{Ker}((\underline{Y(F)^{\mathrm{top}}})^{\wedge}\rightarrow (\underline{E(F)^{\mathrm{top}}})^{\wedge})\] follows from the long exact cohomology sequence \[
T(\underline{Y(F)^{\mathrm{top}}})\to T(\underline{E(F)^{\mathrm{top}}})\to \H^{-1}(C\otimes^L \hat{\Z})\to (\underline{Y(F)^{\mathrm{top}}})^{\wedge}\to (\underline{E(F)^{\mathrm{top}}})^{\wedge} \to  \H^{0}(C\otimes^L \hat{\Z})\to 0.
\] Indeed, the structure theorem of $\underline{E(F)^{\mathrm{top}}}$ (see \cref{prop:structurecohomsab}) implies the vanishing of $T(\underline{E(F)^{\mathrm{top}}})$.

Let us consider the exact sequence in $\FLCA$ \[
0\rightarrow \H^0(D)\otimes \Q/\Z \rightarrow \H^0(D\otimes^L \Q/\Z) \rightarrow \H^{1}(D)_{\mathrm{tors}}\rightarrow 0.
\] By definition of $D$, the long exact cohomology sequence associated to 
\[
 R\GGamma(B_{\hat{W}_F},T^*(\overline{L})_{\mathrm{cond}};\R/\Z)\rightarrow R\GGamma(B_{\hat{W}_F},\MM^*(\overline{L})_{\mathrm{cond}};\R/\Z)\rightarrow R\GGamma(B_{\hat{W}_F},\MM^*/T^*(\overline{L})_{\mathrm{cond}};\R/\Z)
\] gives an epimorphism \[ \HH^0(B_{\hat{W}_F},\MM^*(\overline{L})_{\mathrm{cond}};\R/\Z)\twoheadrightarrow \H^0(D)\] and an isomorphism \[\H^1(D)\overset{\sim}{\rightarrow} \HH^1(B_{\hat{W}_F},\MM^*(\overline{L})_{\mathrm{cond}};\R/\Z).\] By \cref{duality1motives2}, we have \[
\HH^0(B_{\hat{W}_F},\MM^*(\overline{L})_{\mathrm{cond}};\R/\Z)\otimes \Q/\Z=(T\HH^0(B_{\hat{W}_F},\MM(\overline{L})_{\mathrm{cond}})^{\mathrm{lca}})^{\vee}=0,
\] where the vanishing of the Tate module follows from the structure of $\HH^0(B_{\hat{W}_F},\MM(\overline{L})_{\mathrm{cond}})^{\mathrm{lca}})$, see \cref{structuremotives}. This implies that \[
\H^0(D)\otimes \Q/\Z=0,
\] and we have \[
\H^0(D\otimes^L \Q/\Z)=\H^{1}(D)_{\mathrm{tors}}=\HH^1(B_{\hat{W}_F},\MM^*(\overline{L})_{\mathrm{cond}};\R/\Z)_{\mathrm{tors}}.
\]
The result now follows from the exact sequence \eqref{eqn:comparison021}, which gives the identification \[
\HH^1(B_{\hat{W}_F},\MM^*(\overline{L})_{\mathrm{cond}};\R/\Z)_{\mathrm{tors}}=\HH^2(B_{\hat{G}_F},\MM^*(\overline{F})_{cond,\acute{e}t}).
\]
\end{proof}
We conclude with a further analysis of the perfect pairing \[
\HH^1(F,\MM)_{\wedge}\otimes \HH^2(F,\MM^*)\rightarrow \R/\Z,
\] whose induced pairing on the underlying abelian groups is item 1 of \cite[Theorem 2.3]{HarariSzamuely}. First, we make the following observation. 
\begin{rmk}\label{rmk:final} If $A\in\LH(\FLCA)$, then the condensed abelian groups $\leftin{m}{A}$ and $A/m$ are finite for all $m$. This can be seen, for example, by taking the $2$-term resolution \[
A_0\rightarrow A_1\rightarrow A,
\] with $A_0$ and $A_1$ locally compact of finite ranks. If moreover $A$ is non-Hausdorff, we have $A/m\in \NH\cap \FLCA=0$ (see \cref{lem:nhlca}, ii) and iii)). It follows that for all $A\in\LH(\FLCA)$ the derived completion of the canonical decomposition \[
0\rightarrow A^{\mathrm{nh}}\rightarrow A\rightarrow A^{\mathrm{lca}}\rightarrow 0
\] gives an exact sequence of profinite abelian groups \[
0\rightarrow TA^{\mathrm{nh}}\rightarrow TA\rightarrow TA^{\mathrm{lca}}\rightarrow 0
\] and an isomorphism \[
A^{\wedge}\cong (A^{\mathrm{lca}})^{\wedge}.
\] Applying this observation to $\HH^0(B_{\hat{W}_F},\MM(\overline{L})_{\mathrm{cond}})$, we obtain \[
T\HH^0(B_{\hat{W}_F},\MM(\overline{L})_{\mathrm{cond}})=T\HH^0(B_{\hat{W}_F},\MM(\overline{L})_{\mathrm{cond}})^{\mathrm{nh}}\]and\[ \HH^0(B_{\hat{W}_F},\MM(\overline{L})_{\mathrm{cond}})^{\wedge}=(\HH^0(B_{\hat{W}_F},\MM(\overline{L})_{\mathrm{cond}})^{\mathrm{lca}})^{\wedge}.
\] \end{rmk}
\vspace{0.5em}
We recall that the cup-product pairing \eqref{eqn:topologisedhs0} can be expressed as \[
\H^{-1}(C\otimes^L \hat{\Z})\otimes \H^0(D\otimes^L \Q/\Z)\rightarrow \Q/\Z,
\] where $C$ and $D$ are defined in the proof of \cref{cor:harariszamuely1}.

Applying \cref{rmk:final} to the exact sequence 
\[
    0\rightarrow \H^{-1}(C)^{\wedge}\rightarrow \H^{-1}(C\otimes^L \hat{\Z})\rightarrow T\H^0(C)\rightarrow 0,
\] we obtain an exact sequence 
\begin{equation}\tag{*}\label{eqn:final1}
0\rightarrow \HH^{-1}(B_{\hat{W}_F},\MM(\overline{L})_{\mathrm{cond}})^{\wedge}\rightarrow \HH^{-1}(F,\MM)_{\wedge}\rightarrow T\HH^0(B_{\hat{W}_F},\MM(\overline{L})_{\mathrm{cond}})^{\mathrm{nh}}\rightarrow 0.
\end{equation}
Moreover, taking the torsion of the canonical decomposition \[\begin{split}
0&\rightarrow \HH^1(B_{\hat{W}_F},\MM^*(\overline{L})_{\mathrm{cond}};\R/\Z)^{\mathrm{nh}}\rightarrow\\&\rightarrow \HH^1(B_{\hat{W}_F},\MM^*(\overline{L})_{\mathrm{cond}};\R/\Z)\rightarrow \HH^1(B_{\hat{W}_F},\MM^*(\overline{L})_{\mathrm{cond}};\R/\Z)^{\mathrm{lca}}\rightarrow 0\end{split}
\] gives the exact sequence \begin{equation}\tag{**}\label{eqn:final2}\begin{split}
    0&\rightarrow \HH^1(B_{\hat{W}_F},\MM^*(\overline{L})_{\mathrm{cond}};\R/\Z)^{\mathrm{nh}}_{\mathrm{tors}}\rightarrow\\&\rightarrow \HH^2(F,\MM^*)\rightarrow \HH^1(B_{\hat{W}_F},\MM^*(\overline{L})_{\mathrm{cond}};\R/\Z)^{\mathrm{lca}}_{\mathrm{tors}}\rightarrow 0. \end{split}
\end{equation}

Combining \eqref{eqn:final1} and \eqref{eqn:final2}, we observe that the perfect cup-product pairing \eqref{eqn:topologisedhs0} can be decomposed into the two perfect cup-product pairings \[
\HH^{-1}(B_{\hat{W}_F},\MM(\overline{L})_{\mathrm{cond}})^{\wedge}\otimes \HH^1(B_{\hat{W}_F},\MM^*(\overline{L})_{\mathrm{cond}};\R/\Z)^{\mathrm{lca}}_{\mathrm{tors}}\rightarrow \R/\Z
\] and \[
T\HH^{0}(B_{\hat{W}_F},\MM(\overline{L})_{\mathrm{cond}})^{\mathrm{nh}}\otimes \HH^1(B_{\hat{W}_F},\MM^*(\overline{L})_{\mathrm{cond}};\R/\Z)^{\mathrm{nh}}_{\mathrm{tors}}\rightarrow \R/\Z,
\] which can be immediately deduced from \cref{duality1motives2}.
\appendix
\section{Biextensions and pairings}\label{appendix}
The goal of this appendix is to prove the existence of some condensed cup-product pairings which are needed in \cref{section:duality1motives}. In particular we need a condensed Poincaré pairing for the Weil-étale realisation of abelian varieties and for the Weil-étale realisation of 1-motives. To do so, we use the notion of \emph{biextensions} developed in \cite[VII]{SGA7I}. To show that these pairings are compatible with the non-condensed ones, we recall some results on cup-product pairings and their compatibility with respect to morphisms of topoi.

\subsection{Cup-product pairings}
Let $\TT$ be a topos and let $X,Y\in\DDD^{\mathrm{b}}(\TT)$. We denote by $\Gamma(\TT,-):\TT\rightarrow \Set$ the global section functor. There is a canonical pairing in $\DDD^{\mathrm{b}}(\Ab)$ \begin{equation}\label{eqn:cupproductdefinition}
\varphi_{X,Y}:R\Gamma(\TT,X)\otimes^L R\Gamma(\TT,Y)\rightarrow R\Gamma(\TT,X\otimes^L Y)
\end{equation} which can be described via the cup-product pairings of abelian groups \[
\varphi_{X,Y}^{r,s}:\H^r(\TT,X)\otimes \H^s(\TT,Y)\rightarrow \H^{r+s}(\TT,X\otimes^L Y)
\] as follows. We have a natural isomorphism of abelian groups \[
\H^r(\TT,X)=\Hom_{\DDD^{\mathrm{b}}(\TT)}(\Z,X[r])
\] and similarly for $Y$ and $X\otimes^L Y$. We represent elements $\alpha\in\H^r(\TT,X)$ and $\beta\in\H^s(\TT,Y)$ as homotopy classes of maps $\alpha:\Z\to I_X[r]$ and $\beta:\Z\to I_Y[s]$, where $I_X$ and $I_Y$ are injective resolutions of $X$ and $Y$ respectively. Then $\varphi^{r,s}_{X,Y}(\alpha\otimes \beta)$ is represented by \[
\Z=\Z\otimes \Z\rightarrow I_X[r]\otimes I_Y[s]\overset{\sim}{\leftarrow} P(I_X)\otimes I_Y[r+s],
\] where $P(I_X)$ is a flat resolution of $I_X$.
\begin{defn}
Let $X,Y,Z\in \DDD^{\mathrm{b}}(\TT)$ and let $\psi: X\otimes^L Y \rightarrow Z$ be a morphism in $\DDD^{\mathrm{b}}(\TT)$. We define the cup-product pairing induced by $\psi$ as the composition \begin{equation}
    R\Gamma(\TT,X)\otimes^L R\Gamma(\TT,Y)\overset{\varphi_{X,Y}}{\longrightarrow} R\Gamma(\TT,X\otimes^L Y)\overset{R\Gamma(\TT,\psi)}{\longrightarrow} R\Gamma(\TT,Z)
\end{equation} and we denote it $CP_{\psi}$.
\end{defn}
 Let $f:\TT_1\rightarrow \TT_2$ be a morphism of topoi. The left adjoint $f^*$ is exact and induces a functor \[
f^*:\DDD^{\mathrm{b}}(\TT_2)\rightarrow \DDD^{\mathrm{b}}(\TT_1).
\] For $X,Y\in \DDD^{\mathrm{b}}(\TT_2)$, we have a canonical morphism \begin{equation}\label{eqn:functorialrhom}
R\Hom_{\DDD^{\mathrm{b}}(\TT_2)}(X,Y)\rightarrow R\Hom_{\DDD^{\mathrm{b}}(\TT_1)}(f^*X,f^*Y)
\end{equation} functorial in both terms and described as follows. Represent an element $\alpha:R\Hom_{\DDD^{\mathrm{b}}(\TT_2)}(X,Y)$ as the homotopy class of a map $\alpha:X\to I_Y$. Since $f^*$ is exact, $f^*I_Y$ is a resolution of $f^*Y$. The image of $\alpha$ is represented by \[
f^*X\overset{f^*\alpha}{\longrightarrow}f^*I_Y \overset{\sim}{\longrightarrow}I_{f^*Y},
\] where $I_{f^*Y}$ is an injective resolution of $f^*I_Y$.
For $X=\Z$, we get the canonical morphism \[
R\Gamma(\TT_2,Y)\rightarrow R\Gamma(\TT_1, f^* Y)
\] which is functorial in $Y$.

Let $\psi:X\otimes^L Y \rightarrow Z$ and $\psi':f^*X\otimes^L f^*Y\rightarrow f^*Z$ be morphisms in $\DDD^{\mathrm{b}}(\TT_2)$ and $\DDD^{\mathrm{b}}(\TT_1)$ respectively. Then we have a diagram \begin{equation}\label{compatibilitydiagram} \begin{tikzcd}
R\Gamma(\TT_2,X)\otimes^L R\Gamma(\TT_2,Y)\ar[d]\ar[r,"CP_{\psi}"]& R\Gamma(\TT_2,Z)\ar[d]\\
R\Gamma(\TT_1,f^*X)\otimes^L R\Gamma(\TT_1,f^*Y)\ar[r,"CP_{\psi'}"] & R\Gamma(\TT_1,f^*Z).
\end{tikzcd} \end{equation}
\begin{defn}
We say that $\psi$ and $\psi'$ induce compatible cup-product pairings if \eqref{compatibilitydiagram} is commutative.
\end{defn}
\begin{rmk}\label{rmk:pullbacktensor}
For $X,Y\in \DDD^{\mathrm{b}}(\TT_2)$, we have a canonical isomorphism \[
f^*X\otimes^L f^*Y \overset{\sim}{\rightarrow} f^*(X\otimes^L Y).
\] This follows from the isomorphism between non-derived objects and from flatness of $f^*P$ for a flat $P$ (see \cite[IV.13.4 and V.1.7.1]{SGA4}).
\end{rmk}
Consequently, given a pairing $\psi:X\otimes^L Y\rightarrow Z$ in $\DDD^{\mathrm{b}}(\TT_2)$ we have an induced pairing \[f^*\psi:f^*X\otimes^L f^* Y \overset{\sim}{\rightarrow} f^*(X\otimes^L Y)\rightarrow f^*Z\] in $\DDD^{\mathrm{b}}(\TT_1)$.
\begin{prop}\label{prop:compatiblepairings}
The morphisms $\psi$ and $f^*\psi$ induce compatible cup-product pairings.
\end{prop}
\begin{proof}
We decompose the diagram \eqref{compatibilitydiagram} as follows \[
\begin{tikzcd}
R\Gamma(\TT_2,X)\otimes^L R\Gamma(\TT_2,Y)\ar[d]\ar[r,"\varphi_{X,Y}"]&R\Gamma(\TT_2,X\otimes^L Y)\ar[r]\ar[d]&R\Gamma(\TT_2,Z)\ar[d]\\
R\Gamma(\TT_1,f^*X)\otimes^L R\Gamma(\TT_1,f^*Y)\ar[r] & R\Gamma(\TT_1,f^*(X\otimes^L Y))\ar[r]&R\Gamma(\TT_1,f^*Z),
\end{tikzcd}
\] where the first arrow on the second line is the composite \[
R\Gamma(\TT_1,f^*X)\otimes^L R\Gamma(\TT_1,f^*Y)\overset{\varphi_{f^*X,f^*Y}}{\longrightarrow} R\Gamma(\TT_1,f^*X\otimes^L f^*Y)\overset{\sim}{\rightarrow}R\Gamma(\TT_1,f^*(X\otimes^L Y)).
\] The commutativity of the right square follows from functoriality of \eqref{eqn:functorialrhom} in the second argument. We are left to show the commutativity of the left square, which is clear from the definitions of $\varphi_{X,Y}^{r,s}$ and $\varphi_{f^*X,f^*Y}^{r,s}$.
\end{proof}

\subsection{Biextensions in a topos}\label{section:biextensionsinatopos}
The definitions and results of this section are mostly contained in \cite[VII]{SGA7I}.
In what follows, $\TT$ denotes a topos. If $X,Y\in \TT$, we denote by $X_Y\in \TT_{/Y}$ the inverse image of $X$ via the localisation morphism $j_Y:\TT_{/Y}\rightarrow \TT$, i.e.\ the object $X\times Y\rightarrow Y$. We recall the definitions of a torsor and of a biextension in $\TT$, as well as the relation between biextensions and pairings.
\begin{defn}
Let $G$ be a group object of $\TT$, and let $T\rightarrow X$ be a morphism in $\TT$. We say that $T$ is a $G_X$-torsor over $X$ if we have an action of $G_X$ on $T$ over $X$ which is locally trivial, i.e.\ there exists a covering $\{U_i\to X\}_{i\in I}$ in $\TT$ such that for all $i$ the base change $T\times_X U_i$ is isomorphic to $G\times U_i$, where the action of $G$ is given by the left multiplication.
\end{defn}
\begin{rmk}\label{rmk:torsorepi}
If $\pi:T\to X$ is a $G_X$-torsor over $X$, then $\pi$ is an epimorphism in $\TT$. Indeed, by definition there is a covering $\{U_i\to X\}_{i\in I}$ such that the restriction $\pi_{|U_i}:T_{|U_i}\cong G\times U_i\rightarrow U_i$ is an epimorphism for all $i$. Since the morphism $\coprod U_i \rightarrow X$ is an epimorphism, the localisation functor $j:\TT_{/\coprod U_i}\rightarrow \TT_{/X}$ is such that $j^*$ is conservative. The result follows.
\end{rmk}
For any $X\in\TT$ and any $G\in \Ab(\TT)$, we denote by $\mathbf{Tors}(X,G_X)$ the category of $G_X$-torsors over $X$, where morphisms of torsors are morphisms over $X$ compatible with the action of $G_X$.
\begin{defn}\label{defn:biextension}
Let $A,B$ and $G$ be abelian group objects of $\TT$. A \emph{biextension} of $(A,B)$ by $G$ is an object $W$ of $\TT$ together with a surjective morphism $\pi:W\rightarrow A\times B$ endowed with the following structure:
\begin{enumerate}[a)]
    \item An action $W\times_{A\times B}G_{A\times B}\rightarrow W$ of $G_{A\times B}$ on $W$ making $W$ into a $G_{A\times B}$-torsor.
    \item An $A$-morphism $m_A:W\times_A W\rightarrow W$ and a section $e_A:A\rightarrow W$ of the structure morphism $W\to A$, making $W$ into an abelian group object of $\TT_{/A}$.
    \item A $B$-morphism $m_B:W\times_B W\rightarrow W$ and a section $e_B:B\rightarrow W$ of the structure morphism $W\to B$, making $W$ into an abelian group object of $\TT_{/B}$.
\end{enumerate}
These structures are to satisfy the following conditions:
\begin{enumerate}[i)]
    \item if $G_A\rightarrow W$ is the map $(g,a)\mapsto e_A(a)\cdot g$, then the sequence \[
    0\rightarrow G_A\rightarrow W\overset{\pi}{\rightarrow} B_A\rightarrow 0
    \] is an exact sequence in $\Ab(\TT_{/A})$.
    \item if $G_B\rightarrow W$ is the map $(g,b)\mapsto e_B(b)\cdot g$, then the sequence \[
    0\rightarrow G_B\rightarrow W\overset{\pi}{\rightarrow} A_B\rightarrow 0
    \] is an exact sequence in $\Ab(\TT_{/B})$.
    \item (compatibility between $m_A$ and $m_B$) the following diagram commutes 
    \[\begin{tikzcd}[row sep=4em, column sep=6em]
    (W\times_A W)\times_{B\times B} (W\times_A W)\ar[r,"m_A\times_{B\times B} m_A"]\ar[d,"\sim"]&  W\times_{B} W \ar[r,"m_B"]\ar[d] & W\ar[d,equal]\\
    (W\times_B A)\times_{A\times A} (W\times_B W)\ar[r,"m_B\times_{A\times A} m_B"]& W\times_A W\ar[r,"m_A"]& W.
    \end{tikzcd}\] where the isomorphism on the left is $(w_1,w_2;w_3,w_4)\mapsto (w_1,w_3;w_2,w_4)$.
\end{enumerate}
\end{defn}
 For any $A,B,G\in\Ab(\TT)$ we denote by $\mathbf{Biext}(A,B;G)$ the category of biextensions of $(A,B)$ by $G$, where morphisms of biextensions are morphisms in $\TT$ which yield both a morphism of extensions over $A$ and over $B$, as defined in \cite[VII,2.3]{SGA7I}

In \cite[VII,2.5]{SGA7I}, a product of biextensions is defined. This makes the monoid of endomorphism of some $W\in\mathbf{Biext}(A,B;G)$ canonically isomorphic to the monoid of endomorphism of the trivial biextension $A\times B\times G$.  We denote this monoid by $\mathrm{Biext}^0(A,B;G)$ and we observe that it is an abelian group. Moreover, there is a canonical bijection of abelian groups (\cite[VII,2.5.4]{SGA7I})\begin{equation}\label{biextension0pairing}
 \mathrm{Biext}^0(A,B;G)\overset{\sim}{\rightarrow} \Hom_{\Ab(\TT)}(A\otimes B,G), \quad [u]\mapsto \psi_u
\end{equation} described as follows. Let $\pi:W\to A\times B$ be a biextension of $(A,B)$ by $G$. Every automorphism $u$ of $W$ is in particular an automorphism of the underlying $G_{A\times B}$-torsor, and it is induced by a morphism \[
u_0:A\times B\to G
\] via the formula $u(x)=u_0(\pi(x))\cdot x$. In order for $u$ to be an automorphism of biextensions, $u_0$ needs to be bilinear. Consequently, it defines the desired morphism $A\otimes B\to G$.

\vspace{0.5em}
Let $\mathrm{Biext}^1(A,B;G)$ denote the set of isomorphism classes of biextensions of $(A,B)$ by $G$. The product defined in \cite[VII,2.5]{SGA7I} makes it an abelian group. One has a canonical isomorphism of abelian groups (\cite[VII, Corollary 3.6.5]{SGA7I})
\begin{equation}\label{biextensionpairing}
    \mathrm{Biext}^1(A,B;G)\overset{\sim}{\longrightarrow} \Hom_{\DDD(\TT)}(A\otimes^L B,G[1]),\quad [W]\mapsto \psi_W.
\end{equation}

\vspace{0.5em}
\noindent  \begin{rmk}\label{rmk:functorialitycompatibilitybiextensions2}(Functoriality) The isomorphisms \eqref{biextension0pairing} and \eqref{biextensionpairing} are functorial \begin{enumerate}[a)]
    \item (\cite[VII, 3.7]{SGA7I}) in $A,B$ and $G$. Let $f:A'\to A$ and $g:B'\to B$ be morphisms in $\Ab(\TT)$. One has a canonical functor
    \[
-\times_{A\times B}(A'\times B'):\mathbf{Biext}(A,B;G)\rightarrow \mathbf{Biext}(A',B';G)
\] inducing morphisms of abelian groups $
\mathrm{Biext}^i(A,B;G)\rightarrow \mathrm{Biext}^i(A',B';G)
$ for $i=0,1$. Functoriality in $A$ and $B$ is the commutativity of the following diagram \[\begin{tikzcd}
\mathrm{Biext}^i_{\TT}(A,B;G)\ar[r,"\sim"]\ar[d,"-\times_{A\times B}(A'\times B')"] & \Hom_{\DDD(\TT)}(A\otimes^L B,G[i])\ar[d,"-\circ (f\otimes g)"]\\
\mathrm{Biext}^i_{\TT}(A',B';G)\ar[r,"\sim"] & \Hom_{\DDD(\TT)}(A'\otimes^L B',G[i]).
\end{tikzcd}\]
    \item (\cite[VII,2.8, VIII,2.4]{SGA7I}) with respect to the change of topoi. Let $f:\TT_1\rightarrow \TT_2$ be a morphism of topoi.  Since biextensions are defined using finite limits and colimits, we have a canonical functor \[
    f^*:\mathbf{Biext}(A,B;G)\rightarrow \mathbf{Biext}(f^*A,f^*B;f^*G)
    \] which induces homomorphisms of abelian groups $
\mathrm{Biext}^i(A,B;G)\rightarrow \mathrm{Biext}^i(f^*A,f^*B;f^*G)$ for $i=0,1$.  By \cref{rmk:pullbacktensor}, we also have a morphism of abelian groups \[
f^*-:\Hom_{\DDD(\TT_2)}(A\otimes^L B, G[i])\rightarrow \Hom_{\DDD(\TT_1)}(f^*A\otimes^L f^*B,f^*G[i]), \qquad \psi\mapsto f^*\psi.
\]  for $i=0,1$. The functoriality consists in the commutativity of the diagram \[
\begin{tikzcd}
\mathrm{Biext}^i_{\TT_2}(A,B;G)\ar[r,"\sim"]\ar[d,"f^*-"] & \Hom_{\DDD(\TT_2)}(A\otimes^L B,G[i])\ar[d,"f^*-"]\\
\mathrm{Biext}^i_{\TT_1}(f^*A,f^*B;f^*G)\ar[r,"\sim"] & \Hom_{\DDD(\TT_1)}(f^*A\otimes^L f^*B,f^*G[i])
\end{tikzcd}
\] for $i=0,1$. By \cref{prop:compatiblepairings} if $[u]\in \mathrm{Biext}^0(A,B;G)$ (resp.\ if $[W]\in\mathrm{Biext}^1(A,B,G)$), $\psi_{u}$ and $\psi_{f^*u}$ (resp.\ $\psi_W$ and $\psi_{f^*W}$) induce compatible cup-product pairings.
\end{enumerate} \end{rmk}

\subsection{Biextensions of commutative group schemes}\label{appendixsection:commutativegroupschemes}
Let $F$ be a $p$-adic field. We recall that $\TT_{\mathrm{fppf}}$ and $\TT_{\mathrm{\acute{E}t}}$ denote the fppf topos over $F$ and the big étale topos over $F$ respectively. We have a morphism of topoi \[
g:\TT_{\mathrm{fppf}}\rightarrow \TT_{\mathrm{\acute{E}t}}.
\] For all $X\in \TT_{\mathrm{\acute{E}t}}$ and $A,B,G\in \Ab(\TT_{\mathrm{\acute{E}t}})$ we have functors  \[
\mathbf{Tors}(X,G_X)\rightarrow \mathbf{Tors}(g^*X,g^*G_{g^*X}), \quad \mathbf{Biext}(A,B;G)\rightarrow \mathbf{Biext}(g^*A,g^*B;g^*G).
\]

\begin{prop}\label{fppfetaletorsors}
Let $X$ be an $F$-scheme locally of finite type and let $G$ be an affine commutative algebraic group over $F$. Let $T$ be a $G_{X}$-torsor over $X$ for the fppf topology. Then we have \begin{enumerate}[(i)]
\item $T$ is representable. 
\item there exists an \'etale cover $\{f_i:U_i\rightarrow X\}$ of $X$ which trivialises $T$ as a $G_X$-torsor.  \end{enumerate}
In particular, the functor \[
\mathbf{Tors}_{\mathrm{\acute{E}t}}(X,G_X)\rightarrow \mathbf{Tors}_{\mathrm{fppf}}(X,G_X)
\] is an equivalence of categories, and the morphism \[
\mathrm{Tors}_{\mathrm{\acute{E}t}}(X,G_X)\rightarrow \mathrm{Tors}_{\mathrm{fppf}}(X,G_X)
\] is an isomorphism of abelian groups.
\end{prop}
\begin{proof}
The proof of (i) and (ii) is contained in \cite[III, Section 4]{MilneEC}. The functor \[
\mathbf{Tors}_{\mathrm{\acute{E}t}}(X,G_X)\rightarrow \mathbf{Tors}_{\mathrm{fppf}}(X,G_X)
\] is fully faithful by (i), and it is essentially surjective by (ii).
\end{proof}
\begin{cor}\label{fppfetalebiextensions}
Let $A,B$ be commutative group schemes which are locally of finite type over $F$, and let $G$ be an affine commutative algebraic group over $F$. Let $W$ be a biextension of $(A,B)$ by $G$ for the fppf topology. Then we have 
\begin{enumerate}[(i)]
    \item $W$ is representable
    \item $W$ is a biextension of $(A,B)$ by $G$ in $\TT_{\mathrm{\acute{E}t}}$.
\end{enumerate}
In particular, the functor \[
\mathbf{Biext}_{\mathrm{\acute{E}t}}(A,B;G)\rightarrow \mathbf{Biext}_{\mathrm{fppf}}(A,B;G)
\] is an equivalence of categories, and the morphism
\[
\mathrm{Biext}^i_{\mathrm{\acute{E}t}}(A,B;G)\rightarrow \mathrm{Biext}^i_{\mathrm{fppf}}(A,B;G)
\] is an isomorphism of abelian groups for $i=0,1$.
\end{cor}
\begin{proof}
By \cref{fppfetaletorsors}, $W$ is representable and it is a $G_{A\times B}$-torsor over $A\times B$ for the \'etale topology. To show that $W$ is a biextension of $(A,B)$ by $G$ in $\TT_{\mathrm{\acute{E}t}}$, the only non-trivial conditions to check are i) and ii) of \cref{defn:biextension}. The surjectivity of $\pi:W\to A\times B$ for the \'etale topology is guaranteed by the fact that this morphisms realises $W$ as a $G_{A\times B}$-torsor over $A\times B$ (see \cref{rmk:torsorepi}). We need to show that \[
\mathrm{Ker}(W\rightarrow A_B)=G_B
\] in $\Ab(\TT_{\mathrm{\acute{E}t}}{}_{/B})$. We obtain the isomorphism after application of $g^*$, which is fully faithful on representable objects of $\TT_{\mathrm{\acute{E}t}}{}/B$. Items (i) and (ii) follow.

Item (i) assures that the functor \[
\mathbf{Biext}_{\mathrm{\acute{E}t}}(A,B;G)\rightarrow \mathbf{Biext}_{\mathrm{fppf}}(A,B;G)
\] is fully faithful, and item (ii) assures its essential surjectivity.
\end{proof}
As a consequence, when $A,B,G$ are as in the hypotheses of \cref{fppfetalebiextensions}, the study of étale biextensions and fppf biextensions is equivalent. We call such objects just biextensions of $(A,B)$ by $G$, without specifying the topology.

The cohomology of $\TT_{\mathrm{\acute{E}t}}$ coincides with Galois cohomology. Consequently, any biextension $W$ of $(A,B)$ by $G$, resp.\ any automorphism $u$ of a biextension, induces a canonical cup-product pairing \[
CP_{W}:R\Gamma(B_{G_F}(\Set),A(\overline{F}))\otimes^L R\Gamma(B_{G_F}(\Set),B(\overline{F}))\rightarrow R\Gamma(B_{G_F}(\Set),G(\overline{F})[1]),
\] resp.\ \[
CP_{u}:R\Gamma(B_{G_F}(\Set),A(\overline{F}))\otimes^L R\Gamma(B_{G_F}(\Set),B(\overline{F}))\rightarrow R\Gamma(B_{G_F}(\Set),G(\overline{F})).
\] which is compatible with the correspondent fppf-pairing by \cref{rmk:functorialitycompatibilitybiextensions2}, b). 

\subsubsection{The Weil-étale realisation of biextensions}\label{section:weiletalebiext}
In this section we let $A,B$ be commutative group schemes locally of finite type over $F$, and we let $G$ be a commutative affine algebraic group over $F$.

The Weil-étale realisation functor \[
-(\overline{L}):\Sch^{\mathrm{lft}}_F\rightarrow B_{W_F}(\Set), \qquad M\mapsto M(\overline{L})
\] respects products and sends étale coverings to epimorphic families of $B_{W_F}(\Set)$. In order to see this last point, it is enough to observe that if $U\rightarrow X$ is an étale surjective morphism in $\Sch^{\mathrm{lft}}_F$, then the morphism $U(\overline{L})\to X(\overline{L})$ is surjective as well. Since étale surjective morphisms are stable under base change, the morphism $U_L\to X_L$ is étale and surjective. Consequently, the induced morphism on $\overline{L}$-points is surjective. Therefore, the Weil-\'etale realisation induces a morphism of topoi \[
f_{W,\acute{E}t}:B_{W_F}(\Set)\rightarrow \TT_{\mathrm{\acute{E}t}},
\] such that $f^*_{W,\acute{E}t}(M)=M(\overline{L})$ for every $M\in\Sch^{\mathrm{lft}}_F$. Consequently, we have a functor \[
\mathbf{Biext}(A,B;G)\rightarrow \mathbf{Biext}(A(\overline{L}),B(\overline{L});G(\overline{L})), \quad W\mapsto f^*_{W,\acute{E}t}(W)=W(\overline{L})
\] which induces homomorphisms of abelian groups \[
\mathrm{Biext}^i(A,B;G)\rightarrow \mathrm{Biext}^i(A(\overline{L}),B(\overline{L});G(\overline{L}))
\] for $i=0,1$. By \cref{rmk:functorialitycompatibilitybiextensions2}, b), a biextension $W$ and its Weil-étale realisation $W(\overline{L})$ induce compatible cup-product pairings, i.e.\ the diagram \begin{equation}
\begin{tikzcd}[row sep=2em, column sep=3em]
R\Gamma(B_{G_F}(\Set),A(\overline{F}))\otimes^L R\Gamma(B_{G_F}(\Set),B(\overline{F}))\ar[d]\ar[r,"CP_{W}"]&R\Gamma(B_{G_F}(\Set),G(\overline{F}))[1]\ar[d]\\
R\Gamma(B_{W_F}(\Set),A(\overline{L}))\otimes^L R\Gamma(B_{W_F}(\Set),B(\overline{L}))\ar[r,"CP_{W(\overline{L})}"] & R\Gamma(B_{W_F}(\Set),G(\overline{L}))[1]
\end{tikzcd}
\end{equation} is commutative. Similarly, an automorphism $u$ of a biextension $W$ and its Weil-étale realisation induce compatible cup-product pairings.

\subsubsection{The condensed Weil-étale realisation of biextensions}\label{section:condensedweiletalebiext}
Let $A,B$ and $G$ as in the previous section. To define the condensed Weil-étale realisation of a biextension, the first idea would be to follow the previous section, where the role of the Weil-\'etale realisation is taken by the condensed Weil-\'etale realisation $-(\overline{L})_{\mathrm{cond}}:\Sch^{\mathrm{lft}}_F\rightarrow B_{\hat{W}_F}$ defined in \cref{section:realisation}. However, it is not clear (for the moment) whether the condensed Weil-étale realisation of an étale surjective map $U\rightarrow X$ gives an epimorphism in $B_{\hat{W}_F}$. Nonetheless, we are still able to prove the following in the case where $G=\mathbbm{G}_m$
\begin{thm}\label{condensedbiext}
Let $A,B$ be commutative group schemes locally of finite type over $F$, and let $W$ be a biextension of $(A,B)$ by $\mathbbm{G}_m$. Then $W(\overline{L})_{\mathrm{cond}}$ is a biextension of $(A(\overline{L})_{\mathrm{cond}},B(\overline{L})_{\mathrm{cond}})$ by $\overline{L}^{\times}$.

If $W'$ is another biextension of $(A,B)$ by $\mathbbm{G}_m$ and $u:W\to W'$ is a morphism of biextensions, then $u(\overline{L})_{\mathrm{cond}}:W(\overline{L})_{\mathrm{cond}}\to W'(\overline{L})_{\mathrm{cond}}$ is a morphism of biextensions.
\end{thm}
We need the following 
\begin{lem}\label{localsectioncover}
Let $\{f_i:X_i\rightarrow X\}_{i\in I}$ be a local section cover of topological spaces, i.e.\ a collection of continuous maps such that $\sqcup_{i\in I}X_i\to X$ admits local sections. Then the map $\coprod_{i\in I}\underline{X_i}\rightarrow \underline{X}$ is an epimorphism in $\CC$.
\end{lem}
\begin{proof}
Since the condensification functor $\underline{(-)}:\Top\to \CC$ respects coproducts (see for example \cite[Remark 2.32]{Artusa}), we have \[
\coprod_{i\in I}\underline{X_i}=\underline{\sqcup_{i\in I} X_i}.
\]
Let $S$ be a compact Hausdorff topological space and let $f:S\rightarrow X$ be a continuous map. We need to show that there exists $S'$ compact Hausdorff together with a continuous surjection $S'\twoheadrightarrow S$ and a continuous map $g:S'\rightarrow \sqcup_{i\in I} X_i$ such that the following diagram commutes\[
\begin{tikzcd}
\sqcup_{i\in I} X_i\ar[r,"p"]& X\\
S'\ar[u,"g"]\ar[r,two heads] & S\ar[u,"f"].
\end{tikzcd}
\] By hypothesis, there exists an open covering $\{U_j\}_{j\in J}$ of $X$ such that $p_{|p^{-1}(U_j)}:h^{-1}(U_j)\to U_j$ admits a section $s_j$. For all $j\in J$, we set $V_j\coloneqq f^{-1}(U_j)$. By compactness, $S$ is covered by a finite number of such $f^{-1}(U_j)$'s. By regularity of $S$, we can refine $V_j$ to a finite covering $\{S_1,\dots,S_n\}$ of closed subsets of $S$ such that for all $k=1,\dots n$, we have $S_k\subset f^{-1}(U_{j_k})$ for some $j\in J$.   The result now follows by setting \[
S'\coloneqq S_1\sqcup \dots \sqcup S_n, \qquad g_k:S_k\rightarrow \sqcup_{i\in I} X_i, \qquad g_k\coloneqq s_{j_k}\circ f, \, \forall k=1,\dots,n.
\]
\end{proof}
\begin{proof}[Proof of \cref{condensedbiext}]
Firstly, we set $X\coloneqq A\times_F B$ and we show that $W(\overline{L})_{\mathrm{cond}}$ is a $\overline{L}^{\times}_{X(\overline{L})_{\mathrm{cond}}}$-torsor over $X(\overline{L})_{\mathrm{cond}}$. To do this, we exhibit an epimorphic family  $\{X_i\rightarrow X(\overline{L})_{\mathrm{cond}}\}$ in $B_{\hat{W}_F}$ which realises \begin{equation}\label{eqn:torsorcondition:poinc}
W(\overline{L})_{\mathrm{cond}}\times_{X(\overline{L})_{\mathrm{cond}}} X_i\cong \overline{L}^{\times} \times X_i.
\end{equation} Let $\{U_i\rightarrow X\}$ be an étale cover of $X$ which trivialises $W$ in $\TT_{\mathrm{\acute{E}t}}$. Since the condensed Weil-\'etale realisation functor respects fiber products, the condition \eqref{eqn:torsorcondition:poinc} is satisfied with $X_i\coloneqq U_i(\overline{L})_{\mathrm{cond}}$. It remains to show that $\{U_i(\overline{L})_{\mathrm{cond}}\rightarrow X(\overline{L})_{\mathrm{cond}}\}$ is an epimorphic family. By \cite[XI, Proposition 5.1]{SGA1}, we can suppose $\{U_i\to X\}$ to be a Zariski cover of $X$. Consequently, for all $K$ finite extension of $L$, we get an open cover $\{U_i(K)^{\mathrm{top}}\rightarrow X(K)^{\mathrm{top}}\}$. By \cref{localsectioncover}, the morphism \[
\coprod_{i} \underline{U_i(K)^{\mathrm{top}}}\rightarrow \underline{X(K)^{\mathrm{top}}}
\] is an epimorphism after localisation $j^*_{EW_F/\Gal(\overline{L}/K)}:B_{W_F/\Gal(\overline{L}/K)}\rightarrow \CC$. By conservativity of $j^*_{EW_F/\Gal(\overline{L}/K)}$, we conclude that $\{\underline{U_i(K)^{\mathrm{top}}}\rightarrow \underline{X(K)^{\mathrm{top}}}\}$ is a cover in $B_{W_F/\Gal(\overline{L}/K)}$. Since the colimit of epimorphisms is an epimorphism, the family $\{U_i(\overline{L})_{\mathrm{cond}}\rightarrow X(\overline{L})_{\mathrm{cond}}\}$ is epimorphic. This shows that $W(\overline{L})_{\mathrm{cond}}$ is a $\overline{L}^{\times}_{A(\overline{L})_{\mathrm{cond}}\times B(\overline{L})_{\mathrm{cond}}}$-torsor. In particular, the morphism $\pi:W(\overline{L})_{\mathrm{cond}}\rightarrow A(\overline{L})_{\mathrm{cond}}\times B(\overline{L})_{\mathrm{cond}}$ is an epimorphism in $B_{\hat{W}_F}$. The other conditions that $W(\overline{L})_{\mathrm{cond}}$ needs to satisfy in order to be a biextension of $(A(\overline{L})_{\mathrm{cond}},B(\overline{L})_{\mathrm{cond}})$ are expressed as fiber products. Since the condensed Weil-étale realisation functor respects fiber products, we conclude.

\vspace{0.5em}
In order to prove the last statement, we let $W,W'$ be biextensions of $(A,B)$ by $\mathbbm{G}_m$. In order for a morphisms $W\to W'$ to be a morphism of biextensions, it needs to satisfy conditions that only involve fiber products. Since the Weil-étale realisation functor respects fiber products, the result follows.
\end{proof}
\begin{rmk}
The fact that $G=\mathbbm{G}_m$ is crucial since there exists a Zariski cover of $A\times B$ which trivialises $W$ as a $\mathbbm{G}_{m,A\times B}$-torsor (see \cite[XI, Proposition 5.1]{SGA1}). For more general affine $G$, the same proof holds for those biextension $W$ of $(A,B)$ by $G$ that can be trivialised by a Zariski cover of $A\times_F B$. 
\end{rmk}
We introduce a morphism of topoi $\alpha:B_{W_F}(\Set)\rightarrow B_{\hat{W}_F}$  as follows. For every finite extension $K$ of $L$, we set $U_K\coloneqq \Gal(\overline{L}/K)$. We have a morphism of topoi $\beta_K:B_{W_F/U_K}\rightarrow B_{W_F/U_K}(\Set)$ whose right adjoint \[
\beta_{K*}:B_{W_F/U_K}\rightarrow B_{W_F/U_K}(\Set), \quad M\mapsto M(*)
\] sends a condensed $W_F/U_K$-module to its underlying $W_F/U_K$-set. Since the functor $-(*):\CC\to \Set$ commutes with all limits and colimits, then so does the functor $\beta_{K*}$. Consequently, there is a section of $\beta_K$ \[
\alpha_K:B_{W_F/U_K}(\Set)\rightarrow B_{W_F/U_K}.
\] with $\alpha_K^*=\beta_{K*}$.  Taking the inverse limit over finite extensions of $L$, we obtain morphisms \[
\beta:B_{\hat{W}_F}\rightarrow B_{W_F}(\Set), \qquad \alpha:B_{W_F}(\Set)\rightarrow B_{\hat{W}_F},
\] such that $\alpha$ is a section of $\beta$. Moreover, the composition \[
\alpha^*\circ -(\overline{L})_{\mathrm{cond}}:\Sch_F^{\mathrm{lft}}\rightarrow B_{W_F}(\Set)
\] coincides with the Weil-étale realisation $-(\overline{L}):\Sch_F^{\mathrm{lft}}\to B_{W_F}(\Set)$.

\begin{prop}\label{compatibility3biextensions}
Let $A,B$ be commutative group schemes locally of finite type over $F$, and let $W$ be a biextension of $(A,B)$ by $\mathbbm{G}_m$. Then $W$, its Weil-étale realisation $W(\overline{L})$ and its condensed Weil-étale realisation $W(\overline{L})_{\mathrm{cond}}$ induce compatible cup-product pairings. In other words, the following diagram has commutative squares \[
\begin{tikzcd}
R\Gamma(B_{G_F}(\Set),A(\overline{F}))\otimes^L R\Gamma(B_{G_F}(\Set),B(\overline{F}))\ar[d]\ar[r,"CP_{W}"]&R\Gamma(B_{G_F}(\Set),\mathbbm{G}_m(\overline{F}))[1]\ar[d]\\
R\Gamma(B_{W_F}(\Set),A(\overline{L}))\otimes^L R\Gamma(B_{W_F}(\Set),B(\overline{L}))\ar[r,"CP_{W(\overline{L})}"] & R\Gamma(B_{W_F}(\Set),\mathbbm{G}_m(\overline{L}))[1]\\
R\Gamma(B_{\hat{W}_F},A(\overline{L})_{\mathrm{cond}})\otimes^L R\Gamma(B_{\hat{W}_F},B(\overline{L})_{\mathrm{cond}})\ar[r,"CP_{W(\overline{L})_{\mathrm{cond}}}"]\ar[u,"\sim"] & R\Gamma(B_{\hat{W}_F},\overline{L}^{\times})[1]\ar[u,"\sim"]
\end{tikzcd}
\]
\end{prop}
\begin{proof}
We have morphisms of topoi \[
f_{W,\acute{e}t}:B_{W_F}(\Set)\rightarrow \TT_{\mathrm{\acute{E}t}}, \qquad \alpha:B_{W_F}(\Set)\rightarrow B_{\hat{W}_F}.
\] The biextension $W(\overline{L})$ coincides with the inverse image of $W$ via $f_{W,\acute{e}t}$ and also with the inverse image of $W(\overline{L})_{\mathrm{cond}}$ via $\alpha$. We conclude by \cref{rmk:functorialitycompatibilitybiextensions2}, b). 
\end{proof}
\begin{rmk}
In the lower commutative square the vertical arrows are isomorphisms. This follows from \cref{prop:globalsectionscohomag}, or equivalently by the exactness of $\beta_*=\alpha^*:B_{\hat{W}_F}\rightarrow B_{W_F}(\Set)$.
\end{rmk}
\begin{ex}(The Poincaré biextension, \cite[Example C.1]{ADT})\label{example:poincareabelianvar}
Let $A$ be an abelian variety over $F$ and $A^*$ be its dual abelian variety. Then there exists an essentially unique biextension $\mathcal{P}_0$ of $(A,A^*)$ by $\mathbbm{G}_m$ such that the correspondent pairing \[
A\otimes^L A^*\rightarrow \mathbbm{G}_m[1]
\] in $\DDD^{\mathrm{b}}(\TT_{\mathrm{fppf}})$ identifies $A^*$ with $\underline{\Ext}_{\mathrm{fppf}}(A,\mathbbm{G}_m)$. We call this extension \emph{Poincaré biextension}. By \cref{compatibility3biextensions} we get a diagram with commutative squares \[ \begin{tikzcd}
R\Gamma(B_{G_F}(\Set),A(\overline{F}))\otimes^L R\Gamma(B_{G_F}(\Set),A^*(\overline{F}))\ar[d,"\sim"]\ar[r,"CP_{\mathcal{P}_0}"]&R\Gamma(B_{G_F}(\Set),\mathbbm{G}_m(\overline{F})[1]\ar[d]\\
R\Gamma(B_{W_F}(\Set),A(\overline{L}))\otimes^L R\Gamma(B_{W_F}(\Set),A^*(\overline{L}))\ar[r,"CP_{\mathcal{P}_0(\overline{L})}"] & R\Gamma(B_{W_F}(\Set),\mathbbm{G}_m(\overline{L}))[1]\\
R\Gamma(B_{\hat{W}_F},A(\overline{L})_{\mathrm{cond}})\otimes^L R\Gamma(B_{\hat{W}_F},A^*(\overline{L})_{\mathrm{cond}})\ar[r,"CP_{\mathcal{P}_0(\overline{L})_{\mathrm{cond}}}"]\ar[u,"\sim"] & R\Gamma(B_{\hat{W}_F},\overline{L}^{\times})[1]\ar[u,"\sim"].
\end{tikzcd} \] 
Here the upper-left vertical morphism is an isomorphism by \cite[Lemma 5.4.3]{Karpuk2}.
\end{ex}

\subsection{Biextensions of 1-motives}\label{appendixsection:biextension1motives}
In \cite{Deligne}, Deligne generalises the concept of biextension in a topos $\TT$ to two-term complexes as follows.
\begin{defn}\label{defn:biextensiondelignegeneralisation}
Let $K_1=[A_1\rightarrow B_1]$ and $K_2=[A_2\rightarrow B_2]$ be two complexes concentrated in degrees $-1$ and $0$ in $C^{\mathrm{b}}(\Ab(\TT))$, and let $H$ be an abelian group object of $\TT$. A biextension $W$ of $(K_1,K_2)$ by $H$ consists in 
\begin{enumerate}[a)]
    \item A biextension $W_1$ of $B_1$ and $B_2$ by $H$;
    \item A trivialisation (i.e.\ a biadditive section) of the biextension $W_1'$ of $(A_1,B_2)$ by $H$ obtained from $W_1$ by pullback along $A_1\times B_2\to B_1\times B_2$.
    \item A trivialisation of the biextension $W''_1$ of $(B_1,A_2)$ by $H$ obtained from $W_1$ by pullback along $B_1\times A_2\to B_1\times B_2$.
\end{enumerate}
We ask that the two trivialisations coincide on $A_1\times A_2$. If we denote by $s_1$ and $s_2$ the trivialisations, this means that the following diagram \[
\begin{tikzcd}[row sep=0.5em, column sep=1.5em]
&A_1\times B_2\ar[r,"s_1"]& W'_1\ar[rd] &\\
A_1\times A_2\ar[ru]\ar[rd] & & & W_1\\
& B_1\times A_2\ar[r,"s_2"]& W''_1\ar[ru]& 
\end{tikzcd}
\] commutes.
\end{defn}
\begin{defn}
Let $W$ be a biextension of $(K_1,K_2)$ by $H$ as in \cref{defn:biextensiondelignegeneralisation}. An automorphism of $W$ is given by an automorphism of $W_1$, say $\varphi:W_1\rightarrow W_1$ , which is compatible with the biadditive sections $s_1$ and $s_2$, i.e.\ such that $\varphi \circ s_i=s_i$ for $i=1,2$.
\end{defn}
The monoid of automorphisms of any $W$ is canonically isomorphic to the monoid of automorphisms of the trivial biextension. We denote it by $\mathrm{Biext}^0(K_1,K_2;H)$ and we observe that it is an abelian group. We have a canonical isomorphism of abelian groups (\cite[10.2.1]{Deligne}) \begin{equation}\label{generalisedbiextension0pairing}
\mathrm{Biext}^0(K_1,K_2;H)\overset{\sim}{\rightarrow}\Hom_{\Ab(\TT)}(\H^0(K_1)\otimes \H^1(K_2),H)
\end{equation} that can be described as follows. Every automorphism of $W$ is given by an automorphism $\varphi$ of $W_1$ compatible with the biadditive sections. As before, $\varphi$ is determined by a biadditive map $u_0:B_1\times B_2\rightarrow H$. This defines a morphism $B_1\otimes B_2\to H$. The compatibility with the biadditive sections $s_1$ and $s_2$ precisely states that the restriction of $u_0$ to $A_1\times B_2$ and $B_1\times A_2$ is the zero morphism. Hence the morphism induced by $u_0$ factors through $\H^0(K_1)\otimes \H^0(K_2)$.

The set of isomorphism classes of biextensions of $(K_1,K_2)$ by $H$ has an abelian group structure, and we denote it by $\mathrm{Biext}^1(K_1,K_2;H)$. As in \eqref{biextensionpairing}, we have an isomorphism \begin{equation}\label{generalisedbiextensionpairing}
\mathrm{Biext}^1(K_1,K_2;H)\overset{\sim}{\longrightarrow} \Hom_{\DDD(\TT)}(K_1\otimes^L K_2, H[1]).
\end{equation}
\begin{rmk}\label{rmk:functorialities:biextensions:complexes}
The same functoriality of \cref{rmk:functorialitycompatibilitybiextensions2} holds for \eqref{generalisedbiextension0pairing} and \eqref{generalisedbiextensionpairing}.
\end{rmk}
\begin{rmk}\label{rmk:biextensionspecialcase}
We study three special cases:
\begin{enumerate}[a)]
\item if $K_1=[0\to B_1]$ and $K_2=[0\to B_2]$, then \eqref{generalisedbiextension0pairing} and  \eqref{generalisedbiextensionpairing} canonically identify with \eqref{biextension0pairing} and \eqref{biextensionpairing}.
\item If $K_1=[A_1\to 0]$ and $K_2=[0\to B_2]$, we have a canonical isomorphism \[
\mathrm{Biext}^0(A_1,B_2;H)\overset{\sim}{\longrightarrow} \mathrm{Biext}^1(K_1,K_2;H)
\] which sends the class of a biadditive map $\sigma:A_1\times B_2\to H$ to the class of the couple \[(W_1\to 0\times B_2,s:A_1\times B_2\to W_1),\] where $W_1\coloneqq 0\times B_2\times H$ is the trivial biextension and the biadditive section $s:A_1\times B_2\to W_1$ is given by $(0\times \mathrm{id}_{B_2},\sigma)$. Under \eqref{generalisedbiextension0pairing} and \eqref{generalisedbiextensionpairing}, this isomorphism identifies with \[
-[1]:\Hom(A_1\otimes^L B_2,H)\overset{\sim}{\rightarrow}{\Hom}(K_1 \otimes^L K_2,H[1]).
\]
\item Let $K_1=[A_1\overset{u_1}{\to} B_1]$ and $K_2=[0\to B_2]$. By the previous items we have an exact sequence of abelian groups \[
\mathrm{Biext}^0(A_1,K_2;H)\overset{\alpha}{\longrightarrow} \mathrm{Biext}^1(K_1,K_2;H)\overset{\beta}{\longrightarrow} \mathrm{Biext}^1(B_1,K_2;H).
\] Here $\alpha$ and $\beta$ can be described as follows \begin{itemize}
\item $\alpha$ sends the class of $\sigma:A_1\times B_2\to H$ to the class of \[(W_1\to B_1\times B_2,s:A_1\times B_2\to W_1),\] where $W_1\coloneqq B_1\times B_2\times H$ is the trivial biextension and $s:A_1\times B_2\to W_1$ is the biadditive section $(u_1\times \mathrm{id}_{B_2},\sigma)$. 
\item $\beta$ sends the class of $(W_1\to B_1\times B_2,\sigma:A_1\times B_2\to W_1)$ to the class of $W_1$.
\end{itemize}
Under \eqref{generalisedbiextension0pairing} and \eqref{generalisedbiextensionpairing}, this exact sequence canonically identifies with the usual exact sequence \[
\Hom(A_1\otimes^L B_2,H)\rightarrow \Hom(K_1\otimes^L B_2,H[1])\rightarrow \Hom(B_1\otimes^L B_2,H[1]).
\]
\end{enumerate}
\end{rmk}
\vspace{0.5em} 

This generalisation to two-term complexes applies to $1$-motives. Let $\MM_1=[Y_1\rightarrow E_1]$ and $\MM_2=[Y_2\rightarrow E_2]$ be $1$-motives over a $p$-adic field $F$, as defined in \cref{section:realisation1motives}. If $H$ is smooth and affine, it is equivalent to study biextensions of $(\MM_1,\MM_2)$ by $H$ in $\TT_{\mathrm{fppf}}$ and $\TT_{\mathrm{\acute{E}t}}$ by \cref{fppfetalebiextensions}. Moreover, any such biextension is representable. We use \cref{condensedbiext} to prove the following 
\begin{cor}\label{condensedbiextensions1motives}
Let $\MM_1=[Y_1\to E_1]$ and $\MM_2=[Y_2\to E_2]$ be two $1$-motives over the $p$-adic field $F$, and let $W$ be a biextension of $(\MM_1,\MM_2)$ by $\mathbbm{G}_m$. Then $W(\overline{L})_{\mathrm{cond}}$ is a biextension of $(\MM_1(\overline{L})_{\mathrm{cond}},\MM_2(\overline{L})_{\mathrm{cond}})$ by $\overline{L}^{\times}$.
\end{cor}
\begin{proof}
Let $W_1,s_1$ and $s_2$ be respectively the biadditive section of $(E_1,E_2)$ by $\mathbbm{G}_m$, the biadditive section $ Y_1\times E_2\rightarrow W_1$ and the biadditive section $ E_1\times Y_2\rightarrow W_1$ which define $W$.

By \cref{condensedbiext}, the morphism \[
W_1(\overline{L})_{\mathrm{cond}}\rightarrow E_1(\overline{L})_{\mathrm{cond}}\times E_2(\overline{L})_{\mathrm{cond}}
\] realises $W_1(\overline{L})_{\mathrm{cond}}$ as a biextension of $(E_1(\overline{L})_{\mathrm{cond}},E_2(\overline{L})_{\mathrm{cond}})$ by $\overline{L}^{\times}$. Moreover, since the Weil-étale realisation functor respects products, the morphisms $s_1(\overline{L})_{\mathrm{cond}}$ and $s_2(\overline{L})_{\mathrm{cond}}$ are biadditive sections, and they coincide on $Y_1(\overline{L})_{\mathrm{cond}}\times Y_2(\overline{L})_{\mathrm{cond}}$. Consequently, the data of $W(\overline{L})_{\mathrm{cond}}\coloneqq (W_1(\overline{L})_{\mathrm{cond}},s_1(\overline{L})_{\mathrm{cond}}),s_2(\overline{L})_{\mathrm{cond}})$ define a biextension of $(\MM_1(\overline{L})_{\mathrm{cond}},\MM_2(\overline{L})_{\mathrm{cond}})$ by $\overline{L}^{\times}$.
\end{proof}
\begin{rmk}
Let $W$ be a biextension of $(\MM_1,\MM_2)$ by $\mathbbm{G}_m$. 
By \cite[VIII,3.5]{SGA7I}, its restriction $W_1$ to  $E_1\times E_2$ is the pullback of a biextension $W_0$ of $(A_1,A_2)$ by $\mathbbm{G}_m$ along the morphism $E_1\times E_2\rightarrow A_1\times A_2$. Thus $W_1(\overline{L})_{\mathrm{cond}}$ is the pullback of $W_0(\overline{L})_{\mathrm{cond}}$, biextension of $(A_1(\overline{L})_{\mathrm{cond}},A_2(\overline{L})_{\mathrm{cond}})$ by $\overline{L}^{\times}$.
\end{rmk}
As a consequence of \cref{condensedbiextensions1motives}, if $W$ is a biextension of $(\MM_1,\MM_2)$ by $\mathbbm{G}_m$ inducing a pairing \[
\psi_{W}:\MM_1\otimes^L \MM_2\rightarrow\mathbbm{G}_m[1]
\] in $\DDD^{\mathrm{b}}(\TT_{\mathrm{fppf}})$ (or $\DDD^{\mathrm{b}}(\TT_{\mathrm{\acute{E}t}})$), we get a biextension $W(\overline{L})_{\mathrm{cond}}$ inducing a pairing \[
\psi_{W(\overline{L})_{\mathrm{cond}}}:\MM_1(\overline{L})_{\mathrm{cond}}\otimes^L \MM_2(\overline{L})_{\mathrm{cond}}\rightarrow \overline{L}^{\times}[1]
\] in $\DDD^{\mathrm{b}}(B_{\hat{W}_F})$. As in Section \ref{section:weiletalebiext}, we also have an induced biextension $W(\overline{L})$ of $(\MM_1(\overline{L}),\MM_2(\overline{L}))$ by $\mathbbm{G}_m(\overline{L})$ and a pairing \[
\psi_{W(\overline{L})}:\MM_1(\overline{L})\otimes^L \MM_2(\overline{L})\rightarrow \mathbbm{G}_m(\overline{L})[1].
\] 
By adapting the proof of \cref{compatibility3biextensions}, one shows that the induced cup-product pairings are compatible, i.e.\ we have the following commutative diagram \[
\begin{tikzcd}
R\Gamma(B_{G_F}(\Set),\MM_1(\overline{F}))\otimes^L R\Gamma(B_{G_F}(\Set),\MM_2(\overline{F}))\ar[d]\ar[r,"CP_{W}"]&R\Gamma(B_{G_F}(\Set),\mathbbm{G}_m(\overline{F}))[1]\ar[d]\\
R\Gamma(B_{W_F}(\Set),\MM_1(\overline{L}))\otimes^L R\Gamma(B_{W_F}(\Set),\MM_2(\overline{L}))\ar[r,"CP_{W(\overline{L})}"] & R\Gamma(B_{W_F}(\Set),\mathbbm{G}_m(\overline{L}))[1]\\
R\Gamma(B_{\hat{W}_F},\MM_1(\overline{L})_{\mathrm{cond}})\otimes^L R\Gamma(B_{\hat{W}_F},\MM_2(\overline{L})_{\mathrm{cond}})\ar[r,"CP_{W(\overline{L})_{\mathrm{cond}}}"]\ar[u,"\sim"] & R\Gamma(B_{\hat{W}_F},\overline{L}^{\times})[1]\ar[u,"\sim"]
\end{tikzcd}
\]

\subsubsection{The Poincaré biextension}\label{example:poincarebiext}
Let $\MM=[Y\overset{u}{\to} E]$ be a $1$-motive over $F$. In \cite[10.2.11]{Deligne} a dual $1$-motive $\MM^*=[Y^*\overset{u^*}{\to} E^*]$ and a biextension $\mathcal{P}$ of $(\MM,\MM^*)$ by $\mathbbm{G}_m$ are constructed. This biextension is such that its restriction to $E\times E^*$ is the pullback of the Poincaré biextension $\mathcal{P}_0$ of $(A,A^*)$ by $\mathbbm{G}_m$ discussed in \cref{example:poincareabelianvar}. We call $\mathcal{P}$ Poincaré biextension as well. We recall these constructions and we give a proof of the well-known fact that $\mathcal{P}$ is compatibile with the weight filtration. 

\vspace{0.5em} 
\textbf{Construction} (Case $T=0$). We suppose firstly $T=0$. Then we have $\MM=[Y\overset{u}{\rightarrow} A]$. We define the dual $1$-motive $\MM^*\coloneqq[0\rightarrow E^*]$, where $E^*$ is the algebraic group representing the fppf-sheaf $X\mapsto \Ext^1_X(\MM_{|X},\mathbbm{G}_{m|X})$. In particular, an element of $E^*(X)$ is a couple $(\mathcal{E},\eta)$, where $\mathcal{E}$ is an extension of $A_X$ by $\mathbbm{G}_{m,X}$ and $\eta$ is a trivialisation of $\mathcal{E}$ along $Y_X\overset{u_X}{\rightarrow}A_X$. If we set $T^*\coloneqq \underline{\Hom}_{\mathrm{fppf}}(Y,\mathbbm{G}_m)$ and $A^*\coloneqq \underline{\Ext}_{\mathrm{fppf}}(A,\mathbbm{G}_m)$, we have an exact sequence in $\Ab(\TT_{\mathrm{fppf}})$\[
0\rightarrow T^*\rightarrow E^*\rightarrow A^*\rightarrow 0.
\] The two maps can be described as follows \[\begin{split}
T^*(X)=\Hom_X(Y_X,\mathbbm{G}_{m,X})\rightarrow E^*(X)=\Ext_X(\MM_X,\mathbbm{G}_{m,X}), \qquad &\varphi\mapsto (A_X\times \mathbbm{G}_{m,X},u_X\times \varphi),\\
E^*(X)=\Ext_X(\MM_X,\mathbbm{G}_{m,X})\rightarrow A^*(X)=\Hom_X(A_X,\mathbbm{G}_{m,X}), \qquad & (\mathcal{E},\eta)\mapsto \mathcal{E}.
\end{split}
\] 

Let $\mathcal{P}_0$ be the Poincaré biextension of $(A,A^*)$ by $\mathbbm{G}_m$. For all $X$ and for all $\mathcal{E}\in A^*(X)$ we denote by $\mathcal{P}_{0,\mathcal{E}}$ the pullback $\mathcal{P}_0\times_{A^*} X$. By definition of $\mathcal{P}_0$ we have $\mathcal{P}_{0,\mathcal{E}}=\mathcal{E}$, which is an extension of $A_X$ by $\mathbbm{G}_{m,X}$.

The pullback of $\mathcal{P}_0\to A\times A^*$ by the map $A\times E^*\rightarrow A\times A^*$ defines a biextension of $(A,E^*)$ by $\mathbbm{G}_m$. We denote it by $\mathcal{P}_1$. In order to define a biextension of $(\MM,E^*)$ by $\mathbbm{G}_m$, we have to exhibit a biadditive section $Y\times E^*\to \mathcal{P}_1$. We firstly observe that for all $X\in \Sch^{\mathrm{lft}}_F$ and for all $(\mathcal{E},\eta)\in E^*(X)$, we have \[
\mathcal{P}_{1,(\mathcal{E},\eta)}\coloneqq\mathcal{P}_1\times_{E^*}X=(\mathcal{P}_0\times_{A^*}E^*)\times_{E^*}X=\mathcal{P}_{0,\mathcal{E}}=\mathcal{E},
\] which is an extension of $A_X$ by $\mathbbm{G}_{m,X}$. Moreover, $\eta$ gives a trivialisation of such extension along $Y_X\to A_X$. Indeed, we have the following diagram \[
\begin{tikzcd}
& & & Y_X\arrow[d,"u_X"]\arrow[dl,"\eta"]&\\
0\ar{r}&\mathbbm{G}_{m,X}\ar{r}&\mathcal{E}\ar{r}&A_X\ar{r}&0.
\end{tikzcd}
\] We define a biadditive section $\sigma:Y\times E^*\rightarrow \mathcal{P}_1$ as follows. For all $X\in\Sch^{\mathrm{lft}}_F$, we set \begin{equation}\label{defn:biadditivesection1}
\sigma(X):Y(X)\times E^*(X)\rightarrow \mathcal{P}_1(X), \quad (y,(\mathcal{E},\eta))\mapsto \eta(X)(y)\in \mathcal{E}(X)=\mathcal{P}_{1,(\mathcal{E},\eta)}(X)\subset \mathcal{P}_1(X).
\end{equation} \begin{defn}\label{defn:poincarebiextension:t0} (Case $T=0$) The Poincaré biextension $\mathcal{P}$ of $(\MM,\MM^*)$ by $\mathbbm{G}_m$ consists in the following data:\begin{itemize}
    \item the biextension $\mathcal{P}_1$ of $(A,E^*)$ by $\mathbbm{G}_m$, pullback of $\mathcal{P}_0$, the Poincaré biextension of $(A,A^*)$ by $\mathbbm{G}_m$.
    \item the biadditive section $\sigma:Y\times E^* \rightarrow \mathcal{P}_1$ defined in \cref{defn:biadditivesection1}.
\end{itemize}  \end{defn}
\begin{rmk}\label{rmk:biadditivesection1}
We determine the restriction of $\sigma:Y\times E^*\rightarrow \mathcal{P}_1$ to $Y\times T^*$. We fix $X\in\Sch^{\mathrm{lft}}_F$.  Following \eqref{defn:biadditivesection1} and composing with the morphism $T^*(X)\to E^*(X)$ we have 
\[
\sigma_{|}(X): Y(X)\times T^*(X)\rightarrow \mathcal{P}(X), \qquad (y,\varphi)\mapsto (u(X)(y),\varphi(X)(y))\in A(X)\times \mathbbm{G}_m(X).
\] Thus $\sigma_{|Y\times T^*}$ is determined by the biadditive map $C_{\sigma}:Y\times T^*\rightarrow \mathbbm{G}_m$ given for $X\in\Sch^{\mathrm{lft}}_F$ by \[
C_{\sigma}(X):Y(X)\times T^*(X)\rightarrow \mathbbm{G}_m(X), \quad (y,\varphi)\mapsto \varphi(X)(y).
\] $C_{\sigma}$ is exactly $\gamma_{T^*}$, the map inducing Cartier duality between $Y$ and $T^*$ (see \cref{example:charactertorus}).
\end{rmk}
\vspace{0.5em}
\textbf{Compatibility with weight filtration} (Case $T=0$). The Poincaré biextension $\mathcal{P}$ induces a morphism $\psi_{\mathcal{P}}:\MM\otimes^L \MM^*\rightarrow \mathbbm{G}_m[1]$. The weight filtrations on $\MM$ and $\MM^*$ reduce to the fiber sequences \[
A\rightarrow \MM \rightarrow Y[1]
\] and \[
T^*\rightarrow \MM^* \rightarrow A^*
\] respectively. We have the following maps:\begin{itemize} \item $\psi_{\mathcal{P}_0}:A\otimes^L A^*\rightarrow \mathbbm{G}_m[1]$, the map induced by the Poincaré biextension $\mathcal{P}_0$ of $(A,A^*)$ by $\mathbbm{G}_m$.
\item $\gamma_{T^*}:Y\otimes^L T^*\rightarrow \mathbbm{G}_m$, induced by the biadditive map of Cartier duality $Y\times T^*\rightarrow \mathbbm{G}_m$, which coincides with $C_{\sigma}$ by \cref{rmk:biadditivesection1}. \end{itemize}
Compatibility of $\psi_{\mathcal{P}}:\MM\otimes^L\MM^*\rightarrow \mathbbm{G}_m[1]$ with weight filtration is expressed by the following
\begin{lem}\label{lem:poincarecompatibilitymotivest0}
Let $\MM=[Y\to A]$ be a $1$-motive (with $T=0$) and let $\mathcal{P}$ be the Poincaré biextension of $(\MM,\MM^*)$ by $\mathbbm{G}_m$, inducing $\psi_{\mathcal{P}}:\MM\otimes^L\MM^*\to\mathbbm{G}_m[1]$. The following facts hold:
\begin{enumerate}[a)]
\item The composition $A\otimes^L \MM^*\rightarrow \MM\otimes^L \MM^*\overset{\psi_{\mathcal{P}}}{\longrightarrow}\mathbbm{G}_m[1]$ factors through $A\otimes^L A^*$. The induced map $A\otimes^L A^*\rightarrow \mathbbm{G}_m[1]$ is $\psi_{\mathcal{P}_0}$.
\item The composition $\MM\otimes^L T^*\rightarrow \MM\otimes \MM^*\overset{\psi_{\mathcal{P}}}{\longrightarrow}\mathbbm{G}_m[1]$ factors through $Y[1]\otimes^L T^*$. The induced map $Y\otimes^L T^*\rightarrow \mathbbm{G}_m$ is $\gamma_{T^*}$. \end{enumerate}
\end{lem}
This result is well-known. However, the author could not find a proof in the literature that could be adapted directly to the condensed setting. 
\begin{proof}
By functoriality in the first two terms of the isomorphism \eqref{generalisedbiextensionpairing} (see \cref{rmk:functorialities:biextensions:complexes}), the composition $A\otimes^L \MM^*\rightarrow\MM\otimes^L \MM^* \overset{\psi_{\mathcal{P}}}{\longrightarrow}\mathbbm{G}_m[1]$ coincides with $\psi_{\mathcal{P}_1}:A\otimes^L \MM^*\to \mathbbm{G}_m[1]$. In order to prove a) it is enough to show that $\psi_{\mathcal{P}_1}$ is the image of $\psi_{\mathcal{P}_0}:A\otimes^L A^*\to \mathbbm{G}_m[1]$ via the morphism \[
\Hom(A\otimes^L A^*,\mathbbm{G}_m[1])\rightarrow \Hom(A\otimes^L \MM^*,\mathbbm{G}_m[1]).
\] Since by definition $\mathcal{P}_1$ is the pullback of $\mathcal{P}_0$ along $A\times E^*\to A\times A^*$, the result follows again by functoriality of \eqref{generalisedbiextensionpairing}.

As before, the composition $\MM\otimes^L T^*\rightarrow \MM\otimes^L \MM^*\overset{\psi_{\mathcal{P}}}{\longrightarrow}\mathbbm{G}_m[1]$ identifies with $\psi_{\tilde{\mathcal{P}}}$, where $\tilde{\mathcal{P}}$ is the biextension of $(\MM,T^*)$ by $\mathbbm{G}_m$ given by $(\mathcal{P}_0\times_{A\times A^*} A\times T^*,\sigma_{|Y\times T^*})$. We need to show that $\psi_{\tilde{\mathcal{P}}}$ is the image of $\gamma_{T^*}$ via the morphism \[
\Hom(Y\otimes^L T^*,\mathbbm{G}_m)\rightarrow \Hom(\MM\otimes^L T^*,\mathbbm{G}_m[1]).
\] The map $T^*\to A^*$ is trivial, hence $\mathcal{P}_0\times_{A\times A^*} A\times T^*$ is the trivial biextension. Using the identification of \cref{rmk:biextensionspecialcase}, c), we conclude that $\psi_{\tilde{\mathcal{P}}}$ is the image of $C_{\sigma}:Y\times T^*\to \mathbbm{G}_m$ via \[
\Hom(Y\otimes^L T^*,\mathbbm{G}_m)\rightarrow \Hom(\MM\otimes^L T^*,\mathbbm{G}_m[1]).
\] We conclude since $C_{\sigma}=\gamma_{T^*}$ by \cref{rmk:biadditivesection1}.
\end{proof}

\textbf{Construction} (General case). We have $\MM=[Y\overset{u}{\rightarrow}E]$. We define the dual motive $\MM^*\coloneqq[Y^*\overset{u^*}{\rightarrow}E^*]$, with \[Y^*=\underline{\Hom}_{\mathrm{fppf}}(T,\mathbbm{G}_m), \quad E^*=\underline{\Ext}_{\mathrm{fppf}}(\MM/T,\mathbbm{G}_m).\] The morphism $u^*$ is defined as follows. We consider $X\in\Sch^{\mathrm{lft}}_F$ and $\chi:T_X\to\mathbbm{G}_{m,X}$, an element of $Y^*(X)$. By pushing out $E_X$ along $\chi$ we obtain an extension of $A_X$ by $\mathbbm{G}_{m,X}$, say $\mathcal{E_{\chi}}$ \begin{equation}\label{diagram:chi}
\begin{tikzcd}
0\ar{r}&T_X\arrow[d,"\chi"]\ar{r}&E_X\arrow[d,"\overline{\chi}"]\ar{r}&A_X\arrow[d,equal]\ar{r}&0\\
0\ar{r}& \mathbbm{G}_{m,X}\ar{r}&\mathcal{E}_{\chi}\ar{r}&A_X\ar{r}&0.
\end{tikzcd}
\end{equation} Moreover, the composition $\overline{\chi}\circ u_X:Y_X\rightarrow \mathcal{E}_{\chi}$ defines a trivialisation of the extension $\mathcal{E}_{\chi}$ along the morphism $\pi_X\circ u_X:Y_X\to A_X$. The morphism $u^*(X)$ is thus defined as follows \[
u^*(X):Y^*(X)\rightarrow E^*(X), \quad \chi\mapsto (\mathcal{E}_{\chi},\overline{\chi}\circ u_X).
\] The previous procedure (where we supposed $T=0$) defines a biextension $\mathcal{P}'$ of $(\MM/T,E^*)$ by $\mathbbm{G}_m$, given by a biextension $\mathcal{P}'_1$ of $(A,E^*)$ by $\mathbbm{G}_m$ and a biadditive section $\sigma':Y\times E^*\to \mathcal{P}'_1$. The pullback of $\mathcal{P}'_1$  along the morphism $E\times E^*\rightarrow A\times E^*$ defines a biextension $\mathcal{P}_1$ of $(E,E^*)$ by $\mathbbm{G}_m$. Moreover, the biadditive section $\sigma':Y\times E^*\to \mathcal{P'}_1$ factors through $\mathcal{P}_1$ giving $\sigma:Y\times E^*\to \mathcal{P}_1$. In order to define a biextension of $(\MM,\MM^*)$ by $\mathbbm{G}_m$, we have to exhibit a biadditive section $\tau:E\times Y^*\to \mathcal{P}_1$. We firstly observe that for all $X\in\Sch^{\mathrm{lft}}_F$ and for all $(\mathcal{E},\eta)\in E^*(X)$, we have \[
\mathcal{P}_{1,(\mathcal{E},\eta)}=\mathcal{P}'_1\times_{A\times E^*} (E\times X)=\mathcal{P}'_{1,(\mathcal{E},\eta)}\times_{A_X}E_X=\mathcal{E}\times_{A_X}E_X\eqqcolon \Tilde{\mathcal{E}}.
\] Consequently, we have the following pullback diagram of exact sequences \[
\begin{tikzcd}
0\ar{r}&\mathbbm{G}_{m,X}\arrow[d,equal]\ar{r}&\Tilde{\mathcal{E}}\arrow[d]\ar{r}&E_X\arrow[d,"\pi_X"]\ar{r}&0\\
0\ar{r}& \mathbbm{G}_{m,X}\ar{r}&\mathcal{E}\ar{r}&A_X\ar{r}&0,
\end{tikzcd}
\] and $\mathcal{P}_{1,(\mathcal{E},\eta)}=\Tilde{\mathcal{E}}$ is an extension of $E_X$ by $\mathbbm{G}_{m,X}$.

If $(\mathcal{E},\eta)\in E^*(X)$ is in the essential image of $u^*(X):Y^*(X)\to E^*(X)$, then it is of the form $(\mathcal{E}_{\chi},\overline{\chi}\circ u_X)$ for some $\chi\in Y^*(X)$. Consequently we have a trivialisation of $\mathcal{E}$ given by $\overline{\chi}$ \[
\begin{tikzcd}
0\ar{r}&\mathbbm{G}_{m,X}\arrow[d,equal]\ar{r}&\Tilde{\mathcal{E}_{\chi}}\arrow[d]\ar{r}&E_X\arrow[dl,"\overline{\chi}"]\arrow[d,"\pi_X"]\ar{r}&0\\
0\ar{r}& \mathbbm{G}_{m,X}\ar{r}&\mathcal{E}_{\chi}\ar{r}&A_X\ar{r}&0.\end{tikzcd}
\]
This allows us to define, for all $X\in \Sch^{\mathrm{lft}}_F$ the biadditive section \begin{equation}\label{defn:biadditivesection2}
\tau(X):E(X)\times Y^*(X)\rightarrow \mathcal{P}_1(X), \quad (e,\chi)\mapsto (\overline{\chi}(X)(e),e)\in \mathcal{E}_{\chi}\times_{A(X)}E(X)=\Tilde{\mathcal{E}_{\chi}}(X)\subset \mathcal{P}_1(X).
\end{equation}
\begin{defn} 
The Poincaré biextension $\mathcal{P}$ of $(\MM,\MM^*)$ by $\mathbbm{G}_m$ consists in the following data \begin{itemize}
    \item the biextension $\mathcal{P}_1$ of $(E,E^*)$ by $\mathbbm{G}_m$, pullback of $\mathcal{P}'_1$, the Poincaré biextension of $(A,E^*)$ by $\mathbbm{G}_m$ presented in \cref{defn:poincarebiextension:t0}.
    \item the biadditive section $\sigma:Y\times E^* \rightarrow \mathcal{P}_1$ induced by $\sigma':Y\times E^*\to \mathcal{P}_1'$ defined in \eqref{defn:biadditivesection1}.
    \item the biadditive section $\tau:E\times Y^*\to \mathcal{P}_1$ defined in \eqref{defn:biadditivesection2}.
\end{itemize} 
\end{defn}
\begin{rmk}\label{rmk:biadditivesection2}
We determine the restriction of the biadditive section $\tau:E\times Y^*\to \mathcal{P}_1$ to $T\times Y^*$. We fix $X\in \Sch^{\mathrm{lft}}_F$. By definition, for all $t\in T(X)$ and all $\chi\in Y^*(X)=\Hom(T_X,\mathbbm{G}_{m,X})$ we have \[
\tau(X)(t,\chi)=(\overline{\chi}(X)(t),u^*_X(\chi))\in\mathcal{P}_1(X).
\] Since the restriction of $\overline{\chi}:E_X\to \mathcal{E}_{\chi}$ to $T_X$ is exactly $\chi:T_X\to \mathbbm{G}_{m,X}$, $\tau_{|T\times Y^*}(X)$ is given by \[
\tau_{|}(X):T(X)\times Y^*(X)\to \mathcal{P}_1(X), \qquad (t,\chi)\mapsto (\chi(X)(t),u^*_X(\chi)).
\] We have an induced biadditive map $C_{\tau}:T\times Y^*\rightarrow\mathbbm{G}_m$ given for all $X\in\Sch^{\mathrm{lft}}_F$ by \[
C_{\tau}(X):T(X)\times Y^*(X)\rightarrow \mathbbm{G}_m(X), \qquad (t,\chi)\mapsto \chi(X)(t).
\] $C_{\tau}$ is exactly $\gamma_T$, the map inducing Cartier duality between $T$ and $Y^*$ (see \cref{example:charactertorus}).
\end{rmk}

\textbf{Compatibility with weight filtration} (General case). The Poincaré biextension induces a morphism $\psi_{\mathcal{P}}:\MM\otimes^L \MM^*\rightarrow \mathbbm{G}_m[1]$. The weight filtration on $\MM$, resp.\ $\MM^*$, gives fiber sequences \[
T\rightarrow E\rightarrow A, \qquad E\rightarrow \MM\rightarrow Y[1],
\] and respectively \[
T^*\rightarrow E^*\rightarrow A^*, \qquad E^*\rightarrow \MM^*\rightarrow Y^*[1].
\] We have the following maps: \begin{itemize}
    \item  $\psi_{\mathcal{P}_0}:A\otimes^L A^*\to \mathbbm{G}_m[1]$ induced by the Poincaré biextension $\mathcal{P}_0$ of $(A,A^*)$ by $\mathbbm{G}_m$.
    \item $\gamma_T:T\otimes^L Y^*\to \mathbbm{G}_m$ induced by Cartier duality, which coincides with $C_{\tau}$ by \cref{rmk:biadditivesection2}.
    \item $\gamma_{T^*}:Y\otimes^L T^*\rightarrow \mathbbm{G}_m$ induced by Cartier duality, which coincides with $C_{\sigma}$ by \cref{rmk:biadditivesection1}.
\end{itemize}
Compatibility of $\psi_{\mathcal{P}}:\MM\otimes^L \MM^*\rightarrow \mathbbm{G}_m[1]$ with weight filtration is expressed by the following \begin{lem}\label{lem:poincarecompatibilitymotives}
The following facts hold:
\begin{enumerate}[a)]
    \item The composition $\MM\otimes^L E^*\rightarrow \MM\otimes^L \MM^*\overset{\psi_{\mathcal{P}}}{\longrightarrow}\mathbbm{G}_m[1]$ factors through $\MM/T\otimes^L E^*\rightarrow \mathbbm{G}_m[1]$. The induced map $\MM/T\otimes^L E^*\rightarrow \mathbbm{G}_m[1]$ is $\psi_{\mathcal{P}'}$.
    \item The composition $T\otimes^L \MM^*\rightarrow \mathbbm{G}_m[1]\rightarrow \MM\otimes^L \MM^*\overset{\psi_{\mathcal{P}}}{\longrightarrow}\mathbbm{G}_m[1]$ factors through $T\otimes^L Y^*[1]$. The induced map $T\otimes^L Y^*\rightarrow \mathbbm{G}_m$ is $\gamma_T$.
\end{enumerate}
\end{lem}
\begin{rmk}
By combining this with \cref{lem:poincarecompatibilitymotivest0}, we obtain that the map $\psi_{\mathcal{P}}$ induces the maps $\psi_{\mathcal{P}_0}$, $\gamma_T$ and $\gamma_{T^*}$ on the graded pieces of the weight filtration. Again, this result is well-known but the author could not find a satisfying proof directly adaptable to the condensed setting.
\end{rmk}
\begin{proof}
Let us start with a). As in the proof of \cref{lem:poincarecompatibilitymotivest0}, by functoriality of \cref{generalisedbiextensionpairing} the image of $\psi_{\mathcal{P}}$ via \[
\Hom(\MM\otimes^L \MM^*,\mathbbm{G}_m[1])\rightarrow \Hom(\MM\otimes E^*,\mathbbm{G}_m[1])
\] is $\psi_{\tilde{\mathcal{P}}}$, where $\tilde{\mathcal{P}}$ is the biextension of $(\MM,E^*)$ by $\mathbbm{G}_m$ given by $(\mathcal{P}_1,\sigma)$. To conclude, it is enough to show that $\psi_{\tilde{\mathcal{P}}}$ is also the image of $\mathcal{P}'$ via \[
\Hom(\MM/T\otimes^L E^*,\mathbbm{G}_m[1])\rightarrow \Hom(\MM\otimes^L E^*,\mathbbm{G}_m[1]).
\] By definition $\mathcal{P}_1$ is the pullback of $\mathcal{P}'_1$ along $E\times E^*\to A\times E^*$, hence the result follows again by functoriality of \cref{generalisedbiextensionpairing}.

The proof of b) is very similar to its analogue in \cref{lem:poincarecompatibilitymotivest0}. We observe that the composition $T\otimes^L \MM^*\rightarrow \MM\otimes^L \MM^*\overset{\psi_{\mathcal{P}}}{\rightarrow}\mathbbm{G}_m[1]$ is given by $\psi_{\overline{\mathcal{P}}}$. Here $\overline{\mathcal{P}}$ is the biextension of $(T,\MM^*)$ by $\mathbbm{G}_m$ defined by \[
(\mathcal{P}'_1\times_{A\times E^*} T\times E^*, \tau_{T\times Y^*}).
\] We need to show that $\psi_{\overline{\mathcal{P}}}$ is the image of $\gamma_T$ via the morphism \[
\Hom(T\otimes^L Y^*,\mathbbm{G}_m)\rightarrow \Hom(T\otimes^L \MM^*,\mathbbm{G}_m[1]).
\] The map $T\rightarrow A$ is trivial, hence $\mathcal{P}'_1\times_{A\times E^*} T\times E^*\rightarrow T\times E^*$ is isomorphic to the trivial biextension. Using the identification of \cref{rmk:biextensionspecialcase}, c), we conclude that $\psi_{\overline{\mathcal{P}}}$ is the image of $C_{\tau}:T\times Y^*\to \mathbbm{G}_m$ via \[\Hom(T\otimes^L Y^*,\mathbbm{G}_m)\rightarrow \Hom(T\otimes^L \MM^*,\mathbbm{G}_m[1]).\] We conclude since $C_{\tau}=\gamma_T$ by \cref{rmk:biadditivesection2}.
\end{proof}
\subsubsection{The condensed Poincaré biextension}\label{subsectionappendix:condensedpoincare}
By \cref{condensedbiextensions1motives} $\mathcal{P}(\overline{L})_{\mathrm{cond}}$ is a biextension of $(\MM(\overline{L})_{\mathrm{cond}},\MM^*(\overline{L})_{\mathrm{cond}})$ by $\overline{L}^{\times}$. We call it \emph{condensed} Poincaré biextension. It is given by: 
\begin{itemize}
\item the biextension $\mathcal{P}_1(\overline{L})_{\mathrm{cond}}$ of $(E(\overline{L})_{\mathrm{cond}},E^*(\overline{L})_{\mathrm{cond}})$ by $\overline{L}^{\times}$, condensed Weil-étale realisation of $\mathcal{P}_1$. 
\item the biadditive section $\sigma(\overline{L})_{\mathrm{cond}}:Y(\overline{L})_{\mathrm{cond}}\times E^*(\overline{L})_{\mathrm{cond}}\rightarrow \mathcal{P}_1(\overline{L})_{\mathrm{cond}}$, condensed Weil-étale realisation of $\sigma$.
\item the biadditive section $\tau(\overline{L})_{\mathrm{cond}}:E(\overline{L})_{\mathrm{cond}}\times Y^*(\overline{L})_{\mathrm{cond}}\rightarrow \mathcal{P}_1(\overline{L})_{\mathrm{cond}}$, condensed Weil-étale realisation of $\tau$.
\end{itemize} By functoriality of the condensed Weil-étale realisation and since it preserves fiber products, we have \begin{enumerate}
\item $\mathcal{P}_1(\overline{L})_{\mathrm{cond}}$ is the pullback of $\mathcal{P}'_1(\overline{L})_{\mathrm{cond}}$ along $E(\overline{L})_{\mathrm{cond}}\times E^*(\overline{L})_{\mathrm{cond}}\rightarrow A(\overline{L})_{\mathrm{cond}}\times E^*(\overline{L})_{\mathrm{cond}}$.
\item $\mathcal{P}'_1(\overline{L})_{\mathrm{cond}}$ is the pullback of $\mathcal{P}_0(\overline{L})_{\mathrm{cond}}$ along $A(\overline{L})_{\mathrm{cond}}\times E^*(\overline{L})_{\mathrm{cond}}\rightarrow A(\overline{L})_{\mathrm{cond}}\times A^*(\overline{L})_{\mathrm{cond}}$.
\item Let $\mathcal{P'}$ be the Poincaré biextension of $(\MM/T,E^*)$ by $\mathbbm{G}_m$ of \cref{defn:poincarebiextension:t0}. Then the pullback of $\mathcal{P}(\overline{L})_{\mathrm{cond}}$ to a biextension of $(\MM(\overline{L})_{\mathrm{cond}},E^*(\overline{L})_{\mathrm{cond}})$ by $\overline{L}^{\times}$ is $\mathcal{P}'(\overline{L})_{\mathrm{cond}}$.
\item The restriction of $\sigma(\overline{L})_{\mathrm{cond}}$ to $Y(\overline{L})_{\mathrm{cond}}\times T^*(\overline{L})_{\mathrm{cond}}$ is given by $\sigma_{|Y\times T^*}(\overline{L})_{\mathrm{cond}}$. The induced biadditive map $Y(\overline{L})_{\mathrm{cond}}\times T^*(\overline{L})_{\mathrm{cond}}\to \overline{L}^{\times}$ is $C_{\sigma}(\overline{L})_{\mathrm{cond}}=\gamma_{T^*}(\overline{L})_{\mathrm{cond}}$.
\item The restriction of $\tau(\overline{L})_{\mathrm{cond}}$ to $T(\overline{L})_{\mathrm{cond}}\times Y^*(\overline{L})_{\mathrm{cond}}$ is given by $\tau_{|T\times Y^*}(\overline{L})_{\mathrm{cond}}$. The induced biadditive map $T^*(\overline{L})_{\mathrm{cond}}\times Y(\overline{L})_{\mathrm{cond}}\to \overline{L}^{\times}$ is $C_{\tau}(\overline{L})_{\mathrm{cond}}=\gamma_{T}(\overline{L})_{\mathrm{cond}}$.
\end{enumerate}

\textbf{Compatibility with weight filtration}. The condensed Poincaré biextension induces a morphism $\psi_{\mathcal{P}(\overline{L})_{\mathrm{cond}}}:\MM(\overline{L})_{\mathrm{cond}}\otimes^L \MM^*(\overline{L})_{\mathrm{cond}}\rightarrow \overline{L}^{\times}[1]$. The condensed weight filtrations (\cref{condensedfiltration}) give fiber sequences \[
T(\overline{L})_{\mathrm{cond}}\rightarrow E(\overline{L})_{\mathrm{cond}}\rightarrow A(\overline{L})_{\mathrm{cond}}, \qquad E(\overline{L})_{\mathrm{cond}}\rightarrow \MM(\overline{L})_{\mathrm{cond}}\rightarrow Y(\overline{L})_{\mathrm{cond}}[1],
\] and \[
T^*(\overline{L})_{\mathrm{cond}}\rightarrow E^*(\overline{L})_{\mathrm{cond}}\rightarrow A^*(\overline{L})_{\mathrm{cond}}, \qquad E^*(\overline{L})_{\mathrm{cond}}\rightarrow \MM^*(\overline{L})_{\mathrm{cond}}\rightarrow Y^*(\overline{L})_{\mathrm{cond}}[1].
\] We have the following maps: \begin{itemize}
    \item  $\psi_{\mathcal{P}_0(\overline{L})_{\mathrm{cond}}}:A(\overline{L})_{\mathrm{cond}}\otimes^L A^*(\overline{L})_{\mathrm{cond}}\to \overline{L}^{\times}[1]$ induced by the condensed Poincaré biextension $\mathcal{P}_0(\overline{L})_{\mathrm{cond}}$ of $(A(\overline{L})_{\mathrm{cond}},A^*(\overline{L})_{\mathrm{cond}})$ by $\overline{L}^{\times}$.
    \item $\gamma_T(\overline{L})_{\mathrm{cond}}:T(\overline{L})_{\mathrm{cond}}\otimes^L Y^*(\overline{L})_{\mathrm{cond}}\to \overline{L}^{\times}$, the condensed Weil-étale realisation of $\gamma_{T^*}$, the morphism induced by Cartier duality.
        \item $\gamma_{T^*}(\overline{L})_{\mathrm{cond}}:Y(\overline{L})_{\mathrm{cond}}\otimes^L T^*(\overline{L})_{\mathrm{cond}}\rightarrow \overline{L}^{\times}$ the condensed Weil-étale realisation of $\gamma_T$, the morphism induced by Cartier duality.
        \end{itemize}
        The compatibility of $\psi_{\mathcal{P}(\overline{L})_{\mathrm{cond}}}$ with these maps is expressed by the following 
        \begin{lem}\label{lem:lastlemma}
        Let $\mathcal{P}(\overline{L})_{\mathrm{cond}}$ be the condensed Poincaré biextension of $(\MM(\overline{L})_{\mathrm{cond}},\MM^*(\overline{L})_{\mathrm{cond}})$ by $\overline{L}^{\times}$ and let $\mathcal{P}'(\overline{L})_{\mathrm{cond}}$ be the condensed Poincaré biextension of $(\MM/T(\overline{L})_{\mathrm{cond}},E^*(\overline{L})_{\mathrm{cond}})$ by $\overline{L}^{\times}$. The following facts hold:
 \begin{enumerate}[a)]
    \item The composition \[\MM(\overline{L})_{\mathrm{cond}}\otimes^L E^*(\overline{L})_{\mathrm{cond}}\rightarrow \MM(\overline{L})_{\mathrm{cond}}\otimes^L \MM^*(\overline{L})_{\mathrm{cond}}\overset{\psi_{\mathcal{P}(\overline{L})_{\mathrm{cond}}}}{\longrightarrow}\overline{L}^{\times}[1]\] factors through $\MM/T(\overline{L})_{\mathrm{cond}}\otimes^L E^*(\overline{L})_{\mathrm{cond}}\rightarrow\overline{L}^{\times}1]$. The induced map \[\MM/T(\overline{L})_{\mathrm{cond}}\otimes^L E^*(\overline{L})_{\mathrm{cond}}\rightarrow \overline{L}^{\times}[1]\] is $\psi_{\mathcal{P}'(\overline{L})_{\mathrm{cond}}}$.
    \item The composition \[T(\overline{L})_{\mathrm{cond}}\otimes^L \MM^*(\overline{L})_{\mathrm{cond}}\rightarrow \MM(\overline{L})_{\mathrm{cond}}\otimes^L \MM^*(\overline{L})_{\mathrm{cond}}\overset{\psi_{\mathcal{P}(\overline{L})_{\mathrm{cond}}}}{\longrightarrow}\overline{L}^{\times}[1]\] factors through $T(\overline{L})_{\mathrm{cond}}\otimes^L Y^*(\overline{L})_{\mathrm{cond}}[1]$. The induced map \[T(\overline{L})_{\mathrm{cond}}\otimes^L Y^*(\overline{L})_{\mathrm{cond}}\rightarrow \overline{L}^{\times}\] is $\gamma_T(\overline{L})_{\mathrm{cond}}$.
\end{enumerate}
Moreover we have:
\begin{enumerate}[i)]
\item The composition \[A(\overline{L})_{\mathrm{cond}}\otimes^L E^*(\overline{L})_{\mathrm{cond}}\rightarrow \MM/T(\overline{L})_{\mathrm{cond}}\otimes^L E^*(\overline{L})_{\mathrm{cond}}\overset{\psi_{\mathcal{P'}(\overline{L})_{\mathrm{cond}}}}{\longrightarrow}\overline{L}^{\times}[1]\] factors through $A(\overline{L})_{\mathrm{cond}}\otimes^L A^*(\overline{L})_{\mathrm{cond}}$. The induced map \[A(\overline{L})_{\mathrm{cond}}\otimes^L A^*(\overline{L})_{\mathrm{cond}}\rightarrow \overline{L}^{\times}[1]\] is $\psi_{\mathcal{P}_0(\overline{L})_{\mathrm{cond}}}$.
\item The composition \[\MM/T(\overline{L})_{\mathrm{cond}}\otimes^L T^*(\overline{L})_{\mathrm{cond}}\rightarrow \MM/T(\overline{L})_{\mathrm{cond}}\otimes E^*(\overline{L})_{\mathrm{cond}}\overset{\psi_{\mathcal{P}(\overline{L})_{\mathrm{cond}}}}{\longrightarrow}\overline{L}^{\times}[1]\] factors through $Y(\overline{L})_{\mathrm{cond}}[1]\otimes^L T^*(\overline{L})_{\mathrm{cond}}$. The induced map \[Y(\overline{L})_{\mathrm{cond}}\otimes^L T^*(\overline{L})_{\mathrm{cond}}\rightarrow \overline{L}^{\times}\] is $\gamma_{T^*}(\overline{L})_{\mathrm{cond}}$. \end{enumerate}
        \end{lem}
        \begin{proof}
        We just copy the proofs of \cref{lem:poincarecompatibilitymotives,lem:poincarecompatibilitymotivest0}, replacing the biextensions and the biadditive sections by their condensed Weil-étale realisation. 
        \end{proof}
        Consequently, the pairing $\psi_{\mathcal{P}(\overline{L})_{\mathrm{cond}}}$ induces the pairings $\psi_{\mathcal{P}_0(\overline{L})_{\mathrm{cond}}}$, $\gamma_{T*}(\overline{L})_{\mathrm{cond}}$ and $\gamma_T(\overline{L})_{\mathrm{cond}}$ on the graded pieces of the condensed weight filtration.



\section{Relation between two condensed Weil-\'etale duality formalisms (by T. Suzuki)}\label{appendixB}

In this appendix, we explain the relation between Theorem \ref{intro:thm:suz}
(a result of \cite[Section 10]{Suz})
and Theorem \ref{thm:duality1motives1}.
They both give the Weil-\'etale cohomology of $F$
with coefficients in a $1$-motive as an object of $\mathbf{D}^{\mathrm{b}}(\mathscr{C})$
and show its duality.
But there is an apparent difference:
the Weil-\'etale cohomology complex in Theorem \ref{intro:thm:suz} is defined
using the ind-rational pro-\'etale site
while the Weil-\'etale cohomology complex in Theorem \ref{thm:duality1motives1} is defined
using the classifying topos of the condensed Weil group.
We show below that these complexes are in fact canonically isomorphic
and the two duality pairings are compatible.
Since our explanations in this appendix do not use the results of the body of this paper,
this gives another proof of Theorem \ref{thm:duality1motives1} based on Theorem \ref{intro:thm:suz}.

In Section \ref{0026},
we first recall the formulation of Theorem \ref{intro:thm:suz} in Theorem \ref{0013}.
In Section \ref{0027},
we then rewrite it in the style of the classifying topos of the condensed Weil group
in Theorem \ref{0016},
thereby recovering Theorem \ref{thm:duality1motives1}.
Since this appendix is almost entirely about \cite{Suz},
we recall and follow the notation of \cite{Suz}.


\subsection{Recollection on the condensed duality in Theorem \ref{intro:thm:suz}}
\label{0026}

Let $\ast_{\mathrm{proet}}$ be the pro-\'etale site of a point,
namely the category of profinite sets with the Grothendieck (pre)topology
where a covering is a continuous surjection of profinite sets.
Let $k = \mathbbm{F}_{q}$ be a finite field.
For a finite set $S$ and an algebraic extension $k'$ of $k$,
set $S_{k'} = \bigsqcup_{x \in S} \operatorname{Spec} k'$.
For a profinite set $S = \varprojlim_{\lambda} S_{\lambda}$
(where each $S_{\lambda}$ is finite),
set $S_{k'} = \varprojlim_{\lambda} (S_{\lambda})_{k'}$.

Recall from \cite[Section 2.1]{Suz} that
a $k$-algebra $k'$ is said to be \emph{rational}
if it can be written as a finite product $k' = k'_{1} \times \dots \times k'_{n}$,
where each $k'_{i}$ is the perfection (direct limit along Frobenius)
of a finitely generated field extension of $k$.
A $k$-algebra $k'$ is said to be \emph{ind-rational}
if it is a filtered direct limit of rational $k$-algebras.
Let $k^{\mathrm{indrat}}$ be the category of ind-rational $k$-algebras
with $k$-algebra homomorphisms.
(The opposite of) this category can be equipped with the pro-\'etale topology.
The resulting site is called the \emph{ind-rational pro-\'etale site} of $k$
and denoted by $\operatorname{Spec} k^{\mathrm{indrat}}_{\mathrm{proet}}$.

Note that $S_{k'}$ above is the $\operatorname{Spec}$ of an ind-rational $k$-algebra.
Let
	$
			f
		\colon
			\operatorname{Spec} k^{\mathrm{indrat}}_{\mathrm{proet}}
		\to
			\ast_{\mathrm{proet}}
	$
be the premorphism of sites%
\footnote{
	A premorphism of sites $f \colon T' \to T$ between sites defined by given pretopologies
	is a functor $f^{-1}$ from the underlying category of $T$
	to the underlying category of $T'$ 
	that sends covering families to covering families such that
		$
				f^{-1}(Y \times_{X} Z)
			\overset{\sim}{\to}
				f^{-1} Y \times_{f^{-1} X} f^{-1} Z
		$
	whenever $Y \to X$ appears in a covering family.
	See \cite[Section 2.4]{Suz20a}
	or \cite[Definition 2.3]{Suz21}.
	A morphism of sites is further assumed to have
	exact pullback functor
	$f^{\ast} \colon \operatorname{Set}(T) \to \operatorname{Set}(T')$.
}
defined by the functor
$S_{\overline{k}} \mapsfrom S$ on the underlying categories.
Let $F$ be the $q$-th power map on any $k$-algebra,
which induces an action on any object of
the topos $\operatorname{Set}(k^{\mathrm{indrat}}_{\mathrm{proet}})$ and
the derived category $D(k^{\mathrm{indrat}}_{\mathrm{proet}})$.%
\footnote{
	For a site $S$,
	we denote the category of sheaves of sets on $S$ by $\operatorname{Set}(S)$,
	the category of sheaves of abelian groups on $S$ by $\operatorname{Ab}(S)$
	and the (unbounded) derived (triangulated) category of $\operatorname{Ab}(S)$ by $D(S)$.
	For a site such as $\operatorname{Spec} k^{\mathrm{indrat}}_{\mathrm{proet}}$,
	we omit ``$\operatorname{Spec}$'' from the notation
	and instead use the notation
	$\operatorname{Ab}(k^{\mathrm{indrat}}_{\mathrm{proet}})$ for example.
}
Recall from \cite[Section 10]{Suz} the triangulated functor
	\[
			R \Gamma(k_{W}, \;\cdot\;)
		\colon
			D(k^{\mathrm{indrat}}_{\mathrm{proet}})
		\to
			D(\ast_{\mathrm{proet}});
		\quad
			A
		\mapsto
			f_{\ast}[A \stackrel{F - 1}{\to} A][-1],
	\]
where $[\;\cdot\; \to \;\cdot\;]$ denotes the mapping cone construction.
This has a natural cup product morphism (or a lax symmetric monoidal structure)
	\[
			R \Gamma(k_{W}, A) \mathbin{\otimes}^{L} R \Gamma(k_{W}, B)
		\to
			R \Gamma(k_{W}, A \mathbin{\otimes}^{L} B)
	\]
in $D(\ast_{\mathrm{proet}})$
for any $A, B \in D(k^{\mathrm{indrat}}_{\mathrm{proet}})$.

Let $W_{k} \cong \mathbbm{Z}$ be the Weil group of $k$.
Let $B_{W_{k}}(\ast_{\mathrm{proet}})$ be
the classifying topos of $W_{k}$ in $\ast_{\mathrm{proet}}$.
Define a functor $\Gamma(\overline{k}, \;\cdot\;)$ by
	\[
			\Gamma(\overline{k}, \;\cdot\;)
		\colon
			\operatorname{Set}(k^{\mathrm{indrat}}_{\mathrm{proet}})
		\to
			\operatorname{Set}(B_{W_{k}}(\ast_{\mathrm{proet}}));
		\quad
			\Gamma(\overline{k}, C)(S)
		=
			C(S_{\overline{k}}),
	\]
where $C \in \operatorname{Set}(k^{\mathrm{indrat}}_{\mathrm{proet}})$,
$S \in \ast_{\mathrm{proet}}$
and $C(S_{\overline{k}})$ is equipped with the natural $W_{k}$-action.
It is an exact functor.
Let 
	$
			\Gamma(W_{k}, \;\cdot\;)
		\colon
			\operatorname{Set}(B_{W_{k}}(\ast_{\mathrm{proet}}))
		\to
			\operatorname{Set}(\ast_{\mathrm{proet}})
	$
be the $W_{k}$-invariant part functor.
Both $\Gamma(\overline{k}, \;\cdot\;)$ and $\Gamma(W_{k}, \;\cdot\;)$
have natural cup product morphisms.

\begin{prop} \label{0005}
	We have a natural isomorphism
		\[
				R \Gamma(k_{W}, \;\cdot\;)
			\cong
				R \Gamma(W_{k}, \Gamma(\overline{k}, \;\cdot\;))
		\]
	of triangulated functors
	$D(k^{\mathrm{indrat}}_{\mathrm{proet}}) \to D(\ast_{\mathrm{proet}})$
	compatible with the cup product morphisms.
\end{prop}

\begin{proof}
	For a complex $A$ in $\operatorname{Ab}(k^{\mathrm{indrat}}_{\mathrm{proet}})$,
	both objects $R \Gamma(k_{W}, A)$ and
	$R \Gamma(W_{k}, \Gamma(\overline{k}, A))$ are
	represented by the complex $[B \stackrel{F - 1}{\to} B][-1]$,
	where the $n$-th term of $B$ for each $n$ is the sheaf
	sending $S \in \ast_{\mathrm{proet}}$ to $A^{n}(S_{\overline{k}})$.
\end{proof}

Let $\operatorname{Spec} k^{\mathrm{indrat}}_{\mathrm{et}}$ be
the ind-rational \'etale site of $k$,
which is the category $k^{\mathrm{indrat}}$ endowed with the \'etale topology
(\cite[Section 2.1]{Suz}).
Let $\ast^{\mathrm{proet}}_{\mathrm{et}}$ be the category of profinite sets
endowed with the classical topology,
where a covering is an open covering of profinite sets in the usual sense.%
\footnote{
	The convention here is that the upper script (such as ``$\mathrm{indrat}$'') denotes
	the type of objects of the underlying category
	and the lower script (such as ``$\mathrm{proet}$'') denotes the topology.
	When they are ``the same'', the upper script is omitted.
	Hence $\ast^{\mathrm{proet}}_{\mathrm{proet}} = \ast_{\mathrm{proet}}$.
}
Let $B_{W_{k}}(\ast^{\mathrm{proet}}_{\mathrm{et}})$ be the classifying topos of
$W_{k}$ in $\ast^{\mathrm{proet}}_{\mathrm{et}}$.
Let
	\[
			P
		\colon
			\operatorname{Spec} k^{\mathrm{indrat}}_{\mathrm{proet}}
		\to
			\operatorname{Spec} k^{\mathrm{indrat}}_{\mathrm{et}}
		\quad \text{or} \quad
			P
		\colon
			\ast_{\mathrm{proet}}
		\to
			\ast^{\mathrm{proet}}_{\mathrm{et}}
	\]
be the morphism of sites defined by the identity functor,
so that $P^{\ast}$ is the pro-\'etale sheafification functor.
Then $\Gamma(\overline{k}, \;\cdot\;)$ extends to the \'etale topology:

\begin{prop} \label{0006}
	The diagram
		\[
			\begin{CD}
					D(k^{\mathrm{indrat}}_{\mathrm{et}})
				@> \Gamma(\overline{k}, \;\cdot\;) >>
					D(B_{W_{k}}(\ast^{\mathrm{proet}}_{\mathrm{et}}))
				\\
				@V P^{\ast} VV
				@VV P^{\ast} V
				\\
					D(k^{\mathrm{indrat}}_{\mathrm{proet}})
				@> \Gamma(\overline{k}, \;\cdot\;) >>
					D(B_{W_{k}}(\ast_{\mathrm{proet}}))
			\end{CD}
		\]
	commutes.
\end{prop}

\begin{proof}
	It is enough to prove this after forgetting the $W_{k}$-actions
	and identifying $\ast^{\mathrm{proet}}_{\mathrm{et}}$ and $\ast_{\mathrm{proet}}$
	with $\operatorname{Spec} \overline{k}^{\mathrm{proet}}_{\mathrm{et}}$
	(the \'etale site on the category of pro-\'etale $\overline{k}$-schemes) and
	$\operatorname{Spec} \overline{k}_{\mathrm{proet}}$, respectively.
	But then the diagram
		\[
			\begin{CD}
					D(k^{\mathrm{indrat}}_{\mathrm{et}})
				@>>>
					D(\overline{k}^{\mathrm{proet}}_{\mathrm{et}})
				\\
				@V P^{\ast} VV
				@VV P^{\ast} V
				\\
					D(k^{\mathrm{indrat}}_{\mathrm{proet}})
				@>>>
					D(\overline{k}_{\mathrm{proet}})
			\end{CD}
		\]
	commutes,
	where the horizontal functors are restrictions to the small sites over $\overline{k}$.
\end{proof}

Let $K$ be a non-archimedean local field
with ring of integers $\mathcal{O}_{K}$ and residue field $k = \mathbbm{F}_{q}$.
The characteristic of $K$ may be either zero or positive.
For $k' \in k^{\mathrm{indrat}}$, set
	\[
			\mathbf{K}(k')
		=
				(W(k') \mathbin{\Hat{\otimes}}_{W(k)} \mathcal{O}_{K})
			\mathbin{\otimes}_{\mathcal{O}_{K}}
				K,
	\]
where $W$ denotes the ring of $p$-typical Witt vectors of infinite length.
Let $\operatorname{Spec} K_{\mathrm{Et}}$ be (the opposite of) the category of $K$-algebras
endowed with the \'etale topology.
By \cite[Proposition 2.5.1]{Suz20a},
this functor $\mathbf{K}$ defines a premorphism of sites
	\[
			\pi
		\colon
			\operatorname{Spec} K_{\mathrm{Et}}
		\to
			\operatorname{Spec} k^{\mathrm{indrat}}_{\mathrm{et}}.
	\]
As in the paragraphs after \cite[Proposition 2.5.2]{Suz20a}
or \cite[Proposition 2.4]{Suz21},
we have a cup product morphism
	\[
			R \pi_{\ast} A \mathbin{\otimes}^{L} R \pi_{\ast} B
		\to
			R \pi_{\ast}(A \mathbin{\otimes}^{L} B)
	\]
in $D(k^{\mathrm{indrat}}_{\mathrm{et}})$ for any $A, B \in D(K_{\mathrm{Et}})$.
Set
	\[
			\mathbf{\Gamma}(K, \;\cdot\;)
		=
			P^{\ast} \pi_{\ast}
		\colon
			\operatorname{Set}(K_{\mathrm{Et}})
		\to
			\operatorname{Set}(k^{\mathrm{indrat}}_{\mathrm{proet}}).
	\]
As in the first paragraph of \cite[Section 4.1]{Suz}, we have a canonical morphism
	\[
			v_{K}
		\colon
			\mathbf{K}^{\times}
		\to
			\mathbbm{Z}
	\]
in $\operatorname{Ab}(k^{\mathrm{indrat}}_{\mathrm{proet}})$
(with $\mathbbm{Z}$ viewed as a constant group scheme over $k$),
which is the normalized valuation map
$\mathbf{K}^{\times}(k') \to \mathbbm{Z}$ if $k'$ is a perfect field over $k$.
We have $\mathbf{H}^{i}(K, \mathbbm{G}_{m}) = 0$ for $i \ge 1$
by \cite[Proposition (3.4.3)]{Suz},
where $\mathbf{H}^{i}$ denotes the $i$-th cohomology object of $R \mathbf{\Gamma}$.
This defines a morphism
	\[
			R \mathbf{\Gamma}(K, \mathbbm{G}_{m})
		\cong
			\mathbf{K}^{\times}
		\to
			\mathbbm{Z}
	\]
as in \cite[(4.1.1)]{Suz}.

Define
	\[
			R \Gamma(K_{W}, \;\cdot\;)
		=
			R \Gamma(k_{W}, R \mathbf{\Gamma}(K, \;\cdot\;))
		\colon
			D(K_{\mathrm{Et}})
		\to
			D(\ast_{\mathrm{proet}}).
	\]
We have morphisms
	\begin{equation} \label{0014}
			R \Gamma(K_{W}, \mathbbm{G}_{m})
		\to
			R \Gamma(k_{W}, \mathbbm{Z})
		\to
			\mathbbm{Z}[-1].
	\end{equation}
Let $\operatorname{\mathbf{Hom}}_{\ast_{\mathrm{proet}}}$ be the sheaf-Hom functor
for $\ast_{\mathrm{proet}}$.
Now the precise formulation of Theorem \ref{intro:thm:suz} is the following:

\begin{thm}[{\cite[Section 10]{Suz}}] \label{0013}
	Let $M$ and $N$ be $1$-motives over $K$ dual to each other.
	Let
		\[
			M \mathbin{\otimes}^{L} N \to \mathbbm{G}_{m}[1]
		\]
	be the Poincar\'e biextension morphism in $D(K_{\mathrm{Et}})$.
	Then the induced morphisms
		\[
				R \Gamma(K_{W}, M)
			\mathbin{\otimes}^{L}
				R \Gamma(K_{W}, N)
			\to
				R \Gamma(K_{W}, \mathbbm{G}_{m}[1])
			\to
				\mathbbm{Z}
		\]
	induce an isomorphism
		\[
				R \Gamma(K_{W}, M)
			\overset{\sim}{\to}
				R \operatorname{\mathbf{Hom}}_{\ast_{\mathrm{proet}}}(
					R \Gamma(K_{W}, N),
					\mathbbm{Z}
				).
		\]
\end{thm}

Note that \cite{Suz} originally uses the relative fppf site
$\operatorname{Spec} K_{\mathrm{fppf}} / k^{\mathrm{indrat}}_{\mathrm{et}}$
while the above uses the big \'etale site $\operatorname{Spec} K_{\mathrm{Et}}$.
As written in \cite[Section 2.5 and Appendix A]{Suz20a},
we may use the fppf site $\operatorname{Spec} K_{\mathrm{fppf}}$ in \cite{Suz}
in the above style.
Since the fppf cohomology with coefficients in a smooth group scheme
agrees with the \'etale cohomology,
we may further use $\operatorname{Spec} K_{\mathrm{Et}}$
when the coefficients are smooth group scheme
or complexes thereof.
In this manner, we get Theorem \ref{0013} above.

Similarly, the Poincar\'e biextension morphism does not care about the topology,
fppf or \'etale.
See also \cite[Chapter III, Proposition C.2]{ADT}.


\subsection{Deducing Theorem \ref{thm:duality1motives1} from Theorem \ref{intro:thm:suz}}
\label{0027}

Let $I_{K}$ be the inertia subgroup of $\operatorname{Gal}(K^{\mathrm{sep}} / K)$.
Let $W_{K}$ be the Weil group of $K$.
View them naturally as pro-objects in $\operatorname{Set}(\ast_{\mathrm{proet}})$.
Let $\Hat{I}_{K}$ and $\Hat{W}_{K}$ denote these pro-objects.
Let $B_{\Hat{W}_{K}}(\ast_{\mathrm{proet}})$ be the classifying topos of
$\Hat{W}_{K}$ in $\ast_{\mathrm{proet}}$.
Let $B_{\Hat{W}_{K}}(\ast^{\mathrm{proet}}_{\mathrm{et}})$ be similarly.
Let
	\[
			\Gamma(\Hat{I}_{K}, \;\cdot\;)
		\colon
			B_{\Hat{W}_{K}}(\ast_{\mathrm{proet}})
		\to
			B_{W_{k}}(\ast_{\mathrm{proet}}),
	\]
	\[
			\Gamma(\Hat{I}_{K}, \;\cdot\;)
		\colon
			B_{\Hat{W}_{K}}(\ast^{\mathrm{proet}}_{\mathrm{et}})
		\to
			B_{W_{k}}(\ast^{\mathrm{proet}}_{\mathrm{et}})
	\]
be the $\Hat{I}_{K}$-invariant part functors.
Set
	\[
			\Gamma(\Hat{W}_{K}, \;\cdot\;)
		=
			\Gamma(W_{k}, \Gamma(\Hat{I}_{K}, \;\cdot\;))
		\colon
			B_{\Hat{W}_{K}}(\ast_{\mathrm{proet}})
		\to
			B_{W_{k}}(\ast_{\mathrm{proet}}),
	\]
	\[
			\Gamma(\Hat{W}_{K}, \;\cdot\;)
		=
			\Gamma(W_{k}, \Gamma(\Hat{I}_{K}, \;\cdot\;))
		\colon
			B_{\Hat{W}_{K}}(\ast^{\mathrm{proet}}_{\mathrm{et}})
		\to
			B_{W_{k}}(\ast^{\mathrm{proet}}_{\mathrm{et}}),
	\]
which are $\Hat{W}_{K}$-invariant part functors.

Let $K^{\mathrm{ur}}$ be the maximal unramified extension of $K$
and $\Hat{K}^{\mathrm{ur}}$ its completion.
For $C \in \operatorname{Set}(K_{\mathrm{Et}})$ and $S \in \ast_{\mathrm{proet}}$,
define
	\[
			\Gamma(\Hat{K}^{\mathrm{sep}}_{W}, C)(S)
		=
			\varinjlim_{L / \Hat{K}^{\mathrm{ur}}}
				C \bigl(
					L \mathbin{\otimes}_{\Hat{K}^{\mathrm{ur}}} \mathbf{K}(S_{\overline{k}})
				\bigr),
	\]
where $L$ runs over finite separable extensions of $\Hat{K}^{\mathrm{ur}}$.%
\footnote{
	Here $\mathbf{K}(S_{\overline{k}})$ means
	$\mathbf{K}(\mathcal{O}(S_{\overline{k}}))$,
	where we write $S_{\overline{k}} = \operatorname{Spec} \mathcal{O}(S_{\overline{k}})$.
	We are sloppy about whether the functors are defined
	for rings or schemes.
}
It is naturally a discrete $W_{K}$-set.
Varying $S$, we have a presheaf
$\Gamma(\Hat{K}^{\mathrm{sep}}_{W}, C)$ on $\ast_{\mathrm{proet}}$
with an action of $\Hat{W}_{K}$.
It is a sheaf in the classical topology on $\ast_{\mathrm{proet}}$.
Hence we have an object
$\Gamma(\Hat{K}^{\mathrm{sep}}_{W}, C)$
of $B_{\Hat{W}_{K}}(\ast^{\mathrm{proet}}_{\mathrm{et}})$.
Thus we have a functor
	\[
			\Gamma(\Hat{K}^{\mathrm{sep}}_{W}, \;\cdot\;)
		\colon
			\operatorname{Set}(K_{\mathrm{Et}})
		\to
			B_{\Hat{W}_{K}}(\ast^{\mathrm{proet}}_{\mathrm{et}}).
	\]

\begin{prop} \label{0001}
	The functor $\Gamma(\Hat{K}^{\mathrm{sep}}_{W}, \;\cdot\;)$ is exact.
\end{prop}

\begin{proof}
	It is left exact.
	We need to show that if $D \to C$ is an epimorphism in $\operatorname{Set}(K_{\mathrm{Et}})$,
	then $\Gamma(\Hat{K}^{\mathrm{sep}}_{W}, D) \to \Gamma(\Hat{K}^{\mathrm{sep}}_{W}, C)$
	is an epimorphism in $B_{\Hat{W}_{K}}(\ast^{\mathrm{proet}}_{\mathrm{et}})$.
	Let $S \in \ast_{\mathrm{proet}}$ and $L / \Hat{K}^{\mathrm{ur}}$ a finite separable extension.
	Let
		$
				a
			\in
				C \bigl(
					L \mathbin{\otimes}_{\Hat{K}^{\mathrm{ur}}} \mathbf{K}(S_{\overline{k}})
				\bigr)
		$.
	Since $C \in \operatorname{Set}(K_{\mathrm{Et}})$,
	there exist a faithfully flat \'etale
	$L \mathbin{\otimes}_{\Hat{K}^{\mathrm{ur}}} \mathbf{K}(S_{\overline{k}})$-algebra $R$
	and an element $b \in D(R)$ such that $a$ and $b$ map to the same element of $C(R)$
	via the maps
		\[
				D(R)
			\to
				C(R)
			\gets
				C \bigl(
					L \mathbin{\otimes}_{\Hat{K}^{\mathrm{ur}}} \mathbf{K}(S_{\overline{k}})
				\bigr).
		\]
	Write $S = \varprojlim_{\lambda} S_{\lambda}$ with $S_{\lambda}$ finite.
	Set
		\[
				\mathbf{K}^{0}(S_{\overline{k}})
			=
				\varinjlim_{\lambda}
					\mathbf{K}((S_{\lambda})_{\overline{k}}).
		\]
	By \cite[Lemma 2.5.3]{Suz22},
	there exist a faithfully flat \'etale
	$L \mathbin{\otimes}_{\Hat{K}^{\mathrm{ur}}} \mathbf{K}^{0}(S_{\overline{k}})$-algebra $R'$
	and an $L \mathbin{\otimes}_{\Hat{K}^{\mathrm{ur}}} \mathbf{K}(S_{\overline{k}})$-algebra homomorphism
		\begin{equation} \label{0000}
				R
			\to
				R' \mathbin{\otimes}_{\mathbf{K}^{0}(S_{\overline{k}})} \mathbf{K}(S_{\overline{k}}).
		\end{equation}
	Choose an element $\lambda$ large enough
	so that there exist a faithfully flat \'etale
	$L \mathbin{\otimes}_{\Hat{K}^{\mathrm{ur}}} \mathbf{K}((S_{\lambda})_{\overline{k}})$-algebra $R''$
	and a co-cartesian diagram
		\[
			\begin{CD}
					L \mathbin{\otimes}_{\Hat{K}^{\mathrm{ur}}} \mathbf{K}((S_{\lambda})_{\overline{k}})
				@>>>
					R''
				\\ @VVV @VVV \\
					L \mathbin{\otimes}_{\Hat{K}^{\mathrm{ur}}} \mathbf{K}^{0}(S_{\overline{k}})
				@>>>
					R'.
			\end{CD}
		\]
	Since $S_{\lambda}$ is finite, we have
		\[
				\mathbf{K}((S_{\lambda})_{\overline{k}})
			\cong
				\prod_{x \in S_{\lambda}}
					\Hat{K}^{\mathrm{ur}},
		\]
		\[
				L \mathbin{\otimes}_{\Hat{K}^{\mathrm{ur}}} \mathbf{K}((S_{\lambda})_{\overline{k}})
			\cong
				\prod_{x \in S_{\lambda}}
					L.
		\]
	Hence we can write
		\[
				R''
			=
				\prod_{x \in S_{\lambda}}
					R''_{x},
		\]
	where each $R''_{x}$ is a faithfully flat \'etale $L$-algebra.
	Let $R'''$ be the tensor product of the $L$-algebras $R''_{x}$
	over all $x \in S_{\lambda}$,
	which is a faithfully flat \'etale $L$-algebra.
	Let $L'''$ be the residue field of $R'''$ at some maximal ideal,
	which is a finite separable extension of $L$.
	For each $x \in S_{\lambda}$,
	let $i_{x} \colon R''_{x} \to L'''$ be the natural map.
	The product of the maps $i_{x}$ for $x \in \lambda$ defines a map
		\[
				R''
			\to
				\prod_{x \in S}
				 	L'''
			\cong
				L''' \mathbin{\otimes}_{\Hat{K}^{\mathrm{ur}}} \mathbf{K}((S_{\lambda})_{\overline{k}})
		\]
	such that the composite
		\[
				L \mathbin{\otimes}_{\Hat{K}^{\mathrm{ur}}} \mathbf{K}((S_{\lambda})_{\overline{k}})
			\to
				R''
			\to
				L''' \mathbin{\otimes}_{\Hat{K}^{\mathrm{ur}}} \mathbf{K}((S_{\lambda})_{\overline{k}})
		\]
	is the inclusion
	$L \hookrightarrow L'''$ tensored with
	$\mathbf{K}((S_{\lambda})_{\overline{k}})$ over $\Hat{K}^{\mathrm{ur}}$.
	By base change $(\;\cdot\;) \mathbin{\otimes}_{\mathbf{K}((S_{\lambda})_{\overline{k}})} \mathbf{K}(S_{\overline{k}})$,
	we have ring homomorphisms
		\[
				L \mathbin{\otimes}_{\Hat{K}^{\mathrm{ur}}} \mathbf{K}(S_{\overline{k}})
			\to
				R' \mathbin{\otimes}_{\mathbf{K}^{0}(S_{\overline{k}})} \mathbf{K}(S_{\overline{k}})
			\to
				L''' \mathbin{\otimes}_{\Hat{K}^{\mathrm{ur}}} \mathbf{K}(S_{\overline{k}}).
		\]
	With \eqref{0000}, we have
	an $L \mathbin{\otimes}_{\Hat{K}^{\mathrm{ur}}} \mathbf{K}(S_{\overline{k}})$-algebra homomorphism
		\[
				R
			\to
				L''' \mathbin{\otimes}_{\Hat{K}^{\mathrm{ur}}} \mathbf{K}(S_{\overline{k}}).
		\]
	Then $a$ and $b$ map to the same element of
	$C(L''' \mathbin{\otimes}_{\Hat{K}^{\mathrm{ur}}} \mathbf{K}(S_{\overline{k}}))$
	via the maps
		\[
				D(R)
			\to
				D(L''' \mathbin{\otimes}_{\Hat{K}^{\mathrm{ur}}} \mathbf{K}(S_{\overline{k}}))
			\to
				C(L''' \mathbin{\otimes}_{\Hat{K}^{\mathrm{ur}}} \mathbf{K}(S_{\overline{k}}))
			\gets
				C(R)
			\gets
				C(L \mathbin{\otimes}_{\Hat{K}^{\mathrm{ur}}} \mathbf{K}(S_{\overline{k}})).
		\]
	Therefore the image of $b$ in $\Gamma(\Hat{K}^{\mathrm{sep}}_{W}, D)(S)$
	maps to the image of $a$ in $\Gamma(\Hat{K}^{\mathrm{sep}}_{W}, C)(S)$.
	Hence $\Gamma(\Hat{K}^{\mathrm{sep}}_{W}, D) \to \Gamma(\Hat{K}^{\mathrm{sep}}_{W}, C)$
	is an epimorphism in $B_{\Hat{W}_{K}}(\ast^{\mathrm{proet}}_{\mathrm{et}})$.
\end{proof}

We say that a complex $A$ of sheaves of abelian groups on a site $S$ is
\emph{K-limp} (\cite[Section 2.4]{Suz20a})
if for any object $X$ of the underlying category of $S$, the natural morphism
	\[
			\Gamma(X, A)
		\to
			R \Gamma(X, A)
	\]
in $D(\operatorname{Ab})$ is an isomorphism,
where the functor $\Gamma(X, \;\cdot\;)$ on the left is applied term-wise to $A$.
A K-injective complex is K-limp.

\begin{prop} \label{0002}
	Let $A$ be a K-limp complex in $\operatorname{Ab}(K_{\mathrm{Et}})$.
	Then the natural morphism
		\begin{equation} \label{0018}
				\Gamma(\Hat{I}_{K}, \Gamma(\Hat{K}^{\mathrm{sep}}_{W}, A))
			\to
				R \Gamma(\Hat{I}_{K}, \Gamma(\Hat{K}^{\mathrm{sep}}_{W}, A))
		\end{equation}
	in $D(B_{W_{k}}(\ast^{\mathrm{proet}}_{\mathrm{et}}))$ is an isomorphism,
	where the functor $\Gamma(\Hat{I}_{K}, \;\cdot\;)$ on the left
	is applied term-wise.
\end{prop}

\begin{proof}
	We may assume that $A$ is K-injective.
	Indeed, let $A \to I$ be a quasi-isomorphism to a K-injective complex.
	Then its mapping cone $B$ is exact as a complex of presheaves on $\operatorname{Spec} K_{\mathrm{Et}}$
	by the proof of \cite[Proposition 2.4.2]{Suz20a}.
	For such a complex $B$, both sides of \eqref{0018} are zero.
	Hence the statement for $I$ implies the statement for $A$.
	
	Now assume that $A = I$ is K-injective.
	The functor
	$B_{W_{k}}(\ast^{\mathrm{proet}}_{\mathrm{et}}) \to \operatorname{Set}(\ast^{\mathrm{proet}}_{\mathrm{et}})$
	forgetting the $W_{k}$-action
	and the functor
	$B_{\Hat{W}_{K}}(\ast^{\mathrm{proet}}_{\mathrm{et}}) \to B_{\Hat{I}_{K}}(\ast^{\mathrm{proet}}_{\mathrm{et}})$
	restricting the $\Hat{W}_{K}$-action to the $\Hat{I}_{K}$-action
	are exact and conservative.
	Hence it is enough to show that the morphism
		\[
				\Gamma(\Hat{I}_{K}, \Gamma(\Hat{K}^{\mathrm{sep}}_{W}, I))
			\to
				R \Gamma(\Hat{I}_{K}, \Gamma(\Hat{K}^{\mathrm{sep}}_{W}, I))
		\]
	in $D(\ast^{\mathrm{proet}}_{\mathrm{et}})$ is an isomorphism,
	where the functors $\Gamma(\Hat{K}^{\mathrm{sep}}_{W}, \;\cdot\;)$
	and $\Gamma(\Hat{I}_{K}, \;\cdot\;)$ are now viewed as
		\[
				\operatorname{Set}(K_{\mathrm{Et}})
			\stackrel{\Gamma(\Hat{K}^{\mathrm{sep}}_{W}, \;\cdot\;)}{\longrightarrow}
				B_{\Hat{I}_{K}}(\ast^{\mathrm{proet}}_{\mathrm{et}})
			\stackrel{\Gamma(\Hat{I}_{K}, \;\cdot\;)}{\longrightarrow}
				\operatorname{Set}(\ast^{\mathrm{proet}}_{\mathrm{et}}).
		\]
	It is enough to show that for an arbitrary $S \in \ast_{\mathrm{proet}}$,
	the morphism
		\begin{equation} \label{0017}
				\Gamma(\Hat{I}_{K}, \Gamma(\Hat{K}^{\mathrm{sep}}_{W}, I)(S))
			\to
				R \Gamma(\Hat{I}_{K}, \Gamma(\Hat{K}^{\mathrm{sep}}_{W}, I)(S))
		\end{equation}
	in $D(\operatorname{Ab})$ is an isomorphism,
	where the functors $\Gamma(\Hat{K}^{\mathrm{sep}}_{W}, \;\cdot\;)(S)$
	and $\Gamma(\Hat{I}_{K}, \;\cdot\;)$ are now viewed as
		\[
				\operatorname{Set}(\mathbf{K}(S_{\overline{k}})_{\mathrm{et}})
			\stackrel{\Gamma(\Hat{K}^{\mathrm{sep}}_{W}, \;\cdot\;)(S)}{\longrightarrow}
				B_{\Hat{I}_{K}}
			\stackrel{\Gamma(\Hat{I}_{K}, \;\cdot\;)}{\longrightarrow}
				\operatorname{Set}
		\]
	and $I$ is now a K-injective complex
	in $\operatorname{Ab}(\mathbf{K}(S_{\overline{k}})_{\mathrm{et}})$.
	In this setting, the functor $\Gamma(\Hat{K}^{\mathrm{sep}}_{W}, \;\cdot\;)(S)$
	admits a left adjoint,
	given by sending an object $C$ of $B_{\Hat{I}_{K}}$
	to the \'etale $\mathbf{K}(S_{\overline{k}})$-scheme whose base change to
	$(\Hat{K}^{\mathrm{ur}})^{\mathrm{sep}} \mathbin{\otimes}_{\Hat{K}^{\mathrm{ur}}} \mathbf{K}(S_{\overline{k}})$
	with $I_{K}$-equivariant structure
	is the constant $(\Hat{K}^{\mathrm{ur}})^{\mathrm{sep}} \mathbin{\otimes}_{\Hat{K}^{\mathrm{ur}}} \mathbf{K}(S_{\overline{k}})$-scheme $C$
	with $I_{K}$-action.
	This left adjoint is an exact functor
	(so $\Gamma(\Hat{K}^{\mathrm{sep}}_{W}, \;\cdot\;)(S)$ is a morphism of topoi
	in this restricted situation).
	Hence $\Gamma(\Hat{K}^{\mathrm{sep}}_{W}, \;\cdot\;)(S)$ sends
	a K-injective complex to a K-injective complex.
	This implies that \eqref{0017} is an isomorphism, as desired.
\end{proof}

\begin{prop} \label{0003}
	We have a natural isomorphism
		\begin{equation} \label{0020}
				\Gamma(\overline{k}, \pi_{\ast} \;\cdot\;)
			\cong
				\Gamma(\Hat{I}_{K}, \Gamma(\Hat{K}^{\mathrm{sep}}_{W}, \;\cdot\;))
		\end{equation}
	of functors
	$\operatorname{Set}(K_{\mathrm{Et}}) \to B_{W_{k}}(\ast^{\mathrm{proet}}_{\mathrm{et}})$.
\end{prop}

\begin{proof}
	For $C \in \operatorname{Set}(K_{\mathrm{Et}})$ and $S \in \ast_{\mathrm{proet}}$,
	both $\Gamma(\overline{k}, \pi_{\ast} C)(S)$
	and $\Gamma(\Hat{I}_{K}, \Gamma(\Hat{K}^{\mathrm{sep}}_{W}, C))(S)$
	are given by $C(\mathbf{K}(S_{\overline{k}}))$
	with the natural $W_{k}$-action.
\end{proof}

\begin{prop} \label{0009}
	For any $A, B \in D(K_{\mathrm{Et}})$,
	there exists a canonical cup product morphism
		\begin{equation} \label{0025}
					\Gamma(\Hat{K}^{\mathrm{sep}}_{W}, A)
				\mathbin{\otimes}^{L}
					\Gamma(\Hat{K}^{\mathrm{sep}}_{W}, B)
			\to
				\Gamma(\Hat{K}^{\mathrm{sep}}_{W}, A \mathbin{\otimes}^{L} B)
		\end{equation}
	in $D(B_{\Hat{W}_{K}}(\ast^{\mathrm{proet}}_{\mathrm{et}}))$
	functorial in $A, B$.
\end{prop}

\begin{proof}
	For any $A, B \in \operatorname{Ab}(E_{\mathrm{Et}})$,
	any finite separable extension $L / \Hat{K}^{\mathrm{ur}}$
	and any $S \in \ast_{\mathrm{proet}}$,
	we have a natural map
		\[
					A(L \mathbin{\otimes}_{\Hat{K}^{\mathrm{ur}}} \mathbf{K}(S_{\overline{k}}))
				\mathbin{\otimes}
					B(L \mathbin{\otimes}_{\Hat{K}^{\mathrm{ur}}} \mathbf{K}(S_{\overline{k}}))
			\to
				(A \mathbin{\otimes} B)(L \mathbin{\otimes}_{\Hat{K}^{\mathrm{ur}}} \mathbf{K}(S_{\overline{k}})).
		\]
	Hence we have a morphism
		\begin{equation} \label{0024}
					\Gamma(\Hat{K}^{\mathrm{sep}}_{W}, A)
				\mathbin{\otimes}
					\Gamma(\Hat{K}^{\mathrm{sep}}_{W}, B)
			\to
				\Gamma(\Hat{K}^{\mathrm{sep}}_{W}, A \mathbin{\otimes} B)
		\end{equation}
	in $\operatorname{Ab}(B_{\Hat{W}_{K}}(\ast^{\mathrm{proet}}_{\mathrm{et}}))$.
	When $A$ and $B$ are instead complexes in $\operatorname{Ab}(K_{\mathrm{Et}})$,
	this extends to a morphism of complexes in
	$\operatorname{Ab}(B_{\Hat{W}_{K}}(\ast^{\mathrm{proet}}_{\mathrm{et}}))$.
	If $A$ has torsion-free terms, then so does $\Gamma(\Hat{K}^{\mathrm{sep}}_{W}, A)$.
	Passing to the derived category, we get the result.
\end{proof}

Arguing as in the proof of Proposition \ref{0001},
we can actually show that the morphism \eqref{0024}
and hence the morphism \eqref{0025} are isomorphisms.

For the next proposition, recall that the cup product morphism
	\[
			R \pi_{\ast} A \mathbin{\otimes}^{L} R \pi_{\ast} B
		\to
			R \pi_{\ast} (A \mathbin{\otimes}^{L} B)
	\]
for $R \pi_{\ast}$ (for example) can be equivalently written as the morphism
	\[
			R \pi_{\ast} R \operatorname{\mathbf{Hom}}_{K_{\mathrm{Et}}}(A, C)
		\to
			R \operatorname{\mathbf{Hom}}_{k^{\mathrm{indrat}}_{\mathrm{et}}}(
				R \pi_{\ast} A,
				R \pi_{\ast} C
			)
	\]
of functoriality of $R \pi_{\ast}$ via the change of variables
$A \mathbin{\otimes}^{L} B \leadsto C$ and
$R \operatorname{\mathbf{Hom}}_{K_{\mathrm{Et}}}(A, C) \leadsto B$.

\begin{prop} \label{0004}
	We have a natural isomorphism
		\begin{equation} \label{0019}
				\Gamma(\overline{k}, R \pi_{\ast} \;\cdot\;)
			\cong
				R \Gamma(\Hat{I}_{K}, \Gamma(\Hat{K}^{\mathrm{sep}}_{W}, \;\cdot\;))
		\end{equation}
	of functors
	$D(K_{\mathrm{Et}}) \to D(B_{W_{k}}(\ast^{\mathrm{proet}}_{\mathrm{et}}))$
	compatible with the cup product morphisms.
\end{prop}

\begin{proof}
	The isomorphism follows from Propositions \ref{0001}, \ref{0002} and \ref{0003}.
	For the cup products, we need to show that for any $A, B \in D(K_{\mathrm{Et}})$,
	the composite
		\begin{align*}
			&
					\Gamma \bigl(
						\overline{k},
						R \pi_{\ast}
						R \operatorname{\mathbf{Hom}}_{K_{\mathrm{Et}}}(A, B)
					\bigr)
			\\
			&	\to
					\Gamma \bigl(
						\overline{k},
						R \operatorname{\mathbf{Hom}}_{k^{\mathrm{indrat}}_{\mathrm{et}}}(
							R \pi_{\ast} A,
							R \pi_{\ast} B
						)
					\bigr)
			\\
			&	\to
					R \operatorname{\mathbf{Hom}}_{B_{W_{k}}(\ast^{\mathrm{proet}}_{\mathrm{et}})} \bigl(
						\Gamma(\overline{k}, R \pi_{\ast} A),
						\Gamma(\overline{k}, R \pi_{\ast} B)
					\bigr)
		\end{align*}
	and the composite
		\begin{align*}
			&
					R \Gamma \Bigl(
						\Hat{I}_{K},
						\Gamma \bigl(
							\Hat{K}^{\mathrm{sep}}_{W},
							R \operatorname{\mathbf{Hom}}_{K_{\mathrm{Et}}}(A, B)
						\bigr)
					\Bigr)
			\\
			&	\to
					R \Gamma \Bigl(
						\Hat{I}_{K},
						R \operatorname{\mathbf{Hom}}_{B_{\Hat{W}_{K}}(\ast^{\mathrm{proet}}_{\mathrm{et}})} \bigl(
							\Gamma(\Hat{K}^{\mathrm{sep}}_{W}, A),
							\Gamma(\Hat{K}^{\mathrm{sep}}_{W}, B)
						\bigr)
					\Bigr)
			\\
			&	\to
					R \operatorname{\mathbf{Hom}}_{B_{W_{k}}(\ast^{\mathrm{proet}}_{\mathrm{et}})} \Bigl(
						R \Gamma \bigl(
							\Hat{I}_{K},
							\Gamma(\Hat{K}^{\mathrm{sep}}_{W}, A)
						\bigr),
						R \Gamma \bigl(
							\Hat{I}_{K},
							\Gamma(\Hat{K}^{\mathrm{sep}}_{W}, B)
						\bigr)
		\end{align*}
	are compatible via the isomorphism \eqref{0019}.
	We may assume that $A$ and $B$ are (represented by) K-injective complexes.
	Then $\operatorname{\mathbf{Hom}}_{K_{\mathrm{Et}}}(A, B)$ is K-limp by \cite[Proposition 2.4.1]{Suz20a}.
	Hence $R \pi_{\ast} R \operatorname{\mathbf{Hom}}_{K_{\mathrm{Et}}}(A, B)$ is represented by
	$\pi_{\ast} \operatorname{\mathbf{Hom}}_{K_{\mathrm{Et}}}(A, B)$
	by \cite[Proposition 2.4.2]{Suz20a}.
	Also, the objects
		\[
				R \Gamma \bigl(
					\Hat{I}_{K},
					\Gamma(\Hat{K}^{\mathrm{sep}}_{W}, A)
				\bigr),
			\quad
				R \Gamma \bigl(
					\Hat{I}_{K},
					\Gamma(\Hat{K}^{\mathrm{sep}}_{W}, B)
				\bigr),
			\quad
				R \Gamma \Bigl(
					\Hat{I}_{K},
					\Gamma \bigl(
						\Hat{K}^{\mathrm{sep}}_{W},
						R \operatorname{\mathbf{Hom}}_{K_{\mathrm{Et}}}(A, B)
					\bigr)
				\Bigr)
		\]
	are represented by
		\[
				\Gamma \bigl(
					\Hat{I}_{K},
					\Gamma(\Hat{K}^{\mathrm{sep}}_{W}, A)
				\bigr),
			\quad
				\Gamma \bigl(
					\Hat{I}_{K},
					\Gamma(\Hat{K}^{\mathrm{sep}}_{W}, B)
				\bigr),
			\quad
				\Gamma \Bigl(
					\Hat{I}_{K},
					\Gamma \bigl(
						\Hat{K}^{\mathrm{sep}}_{W},
						\operatorname{\mathbf{Hom}}_{K_{\mathrm{Et}}}(A, B)
					\bigr)
				\Bigr),
		\]
	respectively, by Proposition \ref{0002}.
	Therefore it is enough to show that
	for any complexes $A, B$ in $\operatorname{Ab}(K_{\mathrm{Et}})$,
	the composite
		\begin{align*}
			&
					\Gamma \bigl(
						\overline{k},
						\pi_{\ast}
						\operatorname{\mathbf{Hom}}_{K_{\mathrm{Et}}}(A, B)
					\bigr)
			\\
			&	\to
					\Gamma \bigl(
						\overline{k},
						\operatorname{\mathbf{Hom}}_{k^{\mathrm{indrat}}_{\mathrm{et}}}(
							\pi_{\ast} A,
							\pi_{\ast} B
						)
					\bigr)
			\\
			&	\to
					\operatorname{\mathbf{Hom}}_{B_{W_{k}}(\ast^{\mathrm{proet}}_{\mathrm{et}})} \bigl(
						\Gamma(\overline{k}, \pi_{\ast} A),
						\Gamma(\overline{k}, \pi_{\ast} B)
					\bigr)
		\end{align*}
	and the composite
		\begin{align*}
			&
					\Gamma \Bigl(
						\Hat{I}_{K},
						\Gamma \bigl(
							\Hat{K}^{\mathrm{sep}}_{W},
							\operatorname{\mathbf{Hom}}_{K_{\mathrm{Et}}}(A, B)
						\bigr)
					\Bigr)
			\\
			&	\to
					\Gamma \Bigl(
						\Hat{I}_{K},
						\operatorname{\mathbf{Hom}}_{B_{\Hat{W}_{K}}(\ast^{\mathrm{proet}}_{\mathrm{et}})} \bigl(
							\Gamma(\Hat{K}^{\mathrm{sep}}_{W}, A),
							\Gamma(\Hat{K}^{\mathrm{sep}}_{W}, B)
						\bigr)
					\Bigr)
			\\
			&	\to
					\operatorname{\mathbf{Hom}}_{B_{W_{k}}(\ast^{\mathrm{proet}}_{\mathrm{et}})} \Bigl(
						\Gamma \bigl(
							\Hat{I}_{K},
							\Gamma(\Hat{K}^{\mathrm{sep}}_{W}, A)
						\bigr),
						\Gamma \bigl(
							\Hat{I}_{K},
							\Gamma(\Hat{K}^{\mathrm{sep}}_{W}, B)
						\bigr)
		\end{align*}
	are compatible via the isomorphism \eqref{0020}.
	We may assume that $A, B \in \operatorname{Ab}(K_{\mathrm{Et}})$
	(namely, have terms only in degree zero).
	Then the statement to prove is equivalent to the statement that
	the composite
		\begin{align*}
			&
						\Gamma(\overline{k}, \pi_{\ast} A)
					\mathbin{\otimes}
						\Gamma(\overline{k}, \pi_{\ast} B)
			\\
			&	\to
					\Gamma(
						\overline{k},
						\pi_{\ast} A \mathbin{\otimes} \pi_{\ast} B
					)
			\\
			&	\to
					\Gamma(
						\overline{k},
						\pi_{\ast}(A \mathbin{\otimes} B)
					)
		\end{align*}
	and the composite
		\begin{align*}
			&
						\Gamma \bigl(
							\Hat{I}_{K},
							\Gamma(\Hat{K}^{\mathrm{sep}}_{W}, A)
						\bigr)
					\mathbin{\otimes}
						\Gamma \bigl(
							\Hat{I}_{K},
							\Gamma(\Hat{K}^{\mathrm{sep}}_{W}, B)
						\bigr)
			\\
			&	\to
					\Gamma \bigl(
						\Hat{I}_{K},
							\Gamma(\Hat{K}^{\mathrm{sep}}_{W}, A)
						\mathbin{\otimes}
							\Gamma(\Hat{K}^{\mathrm{sep}}_{W}, B)
					\bigr)
			\\
			&	\to
					\Gamma \bigl(
						\Hat{I}_{K},
						\Gamma(\Hat{K}^{\mathrm{sep}}_{W}, A \mathbin{\otimes} B)
					\bigr)
		\end{align*}
	are compatible via the isomorphism \eqref{0020}.
	It is enough to prove the statement
	where all the tensor products $\mathbin{\otimes}$ here are replaced by
	presheaf tensor products
	and $A$ and $B$ are presheaves on $\operatorname{Spec} K_{\mathrm{Et}}$.
	In this case, a direct section-wise comparison
	(over each profinite $S$)
	gives the result.
\end{proof}

\begin{prop} \label{0021}
	We have a commutative diagram
		\[
			\begin{CD}
					D(B_{\Hat{W}_{K}}(\ast^{\mathrm{proet}}_{\mathrm{et}}))
				@> R \Gamma(\Hat{I}_{K}, \;\cdot\;) >>
					D(B_{W_{k}}(\ast^{\mathrm{proet}}_{\mathrm{et}}))
				\\ @V P^{\ast} VV @VV P^{\ast} V \\
					D(B_{\Hat{W}_{K}}(\ast_{\mathrm{proet}}))
				@> R \Gamma(\Hat{I}_{K}, \;\cdot\;) >>
					D(B_{W_{k}}(\ast_{\mathrm{proet}})).
			\end{CD}
		\]
\end{prop}

\begin{proof}
	As in the proof of Proposition \ref{0002},
	it is enough to show that the diagram
		\[
			\begin{CD}
					D(B_{\Hat{I}_{K}}(\ast^{\mathrm{proet}}_{\mathrm{et}}))
				@> R \Gamma(\Hat{I}_{K}, \;\cdot\;) >>
					D(\ast^{\mathrm{proet}}_{\mathrm{et}})
				\\ @V P^{\ast} VV @VV P^{\ast} V \\
					D(B_{\Hat{I}_{K}}(\ast_{\mathrm{proet}}))
				@> R \Gamma(\Hat{I}_{K}, \;\cdot\;) >>
					D(\ast_{\mathrm{proet}})
			\end{CD}
		\]
	commutes.
	Let $A \in D(B_{\Hat{I}_{K}}(\ast^{\mathrm{proet}}_{\mathrm{et}}))$.
	For any extremally disconnected profinite set $S$, we have
		\begin{align*}
			&
					\Gamma(S, R \Gamma(\Hat{I}_{K}, P^{\ast} A))
				\cong
					R \Gamma(\Hat{I}_{K}, \Gamma(S, P^{\ast} A))
				\cong
					R \Gamma(\Hat{I}_{K}, \Gamma(S, A))
			\\
			&
				\cong
					\Gamma(S, R \Gamma(\Hat{I}_{K}, A))
				\cong
					\Gamma(S, P^{\ast} R \Gamma(\Hat{I}_{K}, A)).
		\end{align*}
	Hence $R \Gamma(\Hat{I}_{K}, P^{\ast} A) \cong P^{\ast} R \Gamma(\Hat{I}_{K}, A)$.
\end{proof}

\begin{prop} \label{0008}
	We have a natural isomorphism
		\[
				R \Gamma(K_{W}, \;\cdot\;)
			\cong
				R \Gamma(\Hat{W}_{K}, P^{\ast} \Gamma(\Hat{K}^{\mathrm{sep}}_{W}, \;\cdot\;))
		\]
	of functors
	$D(K_{\mathrm{Et}}) \to D(\ast_{\mathrm{proet}})$
	compatible with the cup product morphisms.
\end{prop}

\begin{proof}
	Apply $R \Gamma(W_{k}, P^{\ast} \;\cdot\;)$
	to the isomorphism in Proposition \ref{0004}.
	By Propositions \ref{0005}, \ref{0006} and \ref{0021}, we get the result.
\end{proof}

\begin{prop} \label{0007}
	For any $K$-scheme $X$ locally of finite type,
	the object $\Gamma(\Hat{K}^{\mathrm{sep}}_{W}, X)$
	of the category $\operatorname{Set}(B_{\Hat{W}_{K}}(\ast^{\mathrm{proet}}_{\mathrm{et}}))$ is canonically isomorphic
	to the direct limit of $\underline{X(L)^{\mathrm{top}}}$
	over finite separable extensions $L / \Hat{K}^{\mathrm{ur}}$,
	where $X(L)^{\mathrm{top}}$ is the classical topological space structure on $X(L)$
	and $\underline{X(L)^{\mathrm{top}}}$ is the corresponding sheaf on $\ast_{\mathrm{proet}}$.
	In particular, we have
	$\Gamma(\Hat{K}^{\mathrm{sep}}_{W}, X) \in \operatorname{Set}(B_{\Hat{W}_{K}}(\ast_{\mathrm{proet}}))$.
\end{prop}

\begin{proof}
	For $n \ge 1$,
	set $\mathcal{O}_{K, n} = \mathcal{O}_{K} / \mathfrak{p}_{K}^{n}$,
	where $\mathfrak{p}_{K}$ is the maximal ideal of $\mathcal{O}_{K}$.
	Define $\Hat{\mathcal{O}}^{\mathrm{ur}}_{K, n}$ similarly from $\Hat{K}^{\mathrm{ur}}$.
	For $k' \in k^{\mathrm{indrat}}$, define $\mathbf{O}_{K, n}(k')$ to be
	$W(k') \mathbin{\otimes}_{W(k)} \mathcal{O}_{K, n}$.
	Then $\underline{\Hat{\mathcal{O}}^{\mathrm{ur}}_{K, n}}$ is discrete.
	For any finite set $S$, we have
		\[
				\underline{\Hat{\mathcal{O}}^{\mathrm{ur}}_{K, n}}(S)
			\cong
				\mathbf{O}_{K, n}(S_{\overline{k}}),
		\]
	since both are $\prod_{x \in S} \Hat{\mathcal{O}}^{\mathrm{ur}}_{K, n}$.
	Taking direct limits in $S$,
	this isomorphism holds for any $S \in \ast_{\mathrm{proet}}$.
	Taking the inverse limit in $n$ and $\mathbin{\otimes}_{\mathcal{O}_{K}} K$,
	we have
		\[
				\underline{\Hat{K}^{\mathrm{ur}}}(S)
			\cong
				\mathbf{K}(S_{\overline{k}}).
		\]
	
	Let $L / \Hat{K}^{\mathrm{ur}}$ be finite separable.
	Taking $L \mathbin{\otimes}_{\Hat{K}^{\mathrm{ur}}}$, we have
		\begin{equation} \label{0015}
				\underline{L}(S)
			\cong
				L \mathbin{\otimes}_{\Hat{K}^{\mathrm{ur}}} \mathbf{K}(S_{\overline{k}}).
		\end{equation}
	Therefore
		\[
				X(\underline{L}(S))
			\cong
				X(L \mathbin{\otimes}_{\Hat{K}^{\mathrm{ur}}} \mathbf{K}(S_{\overline{k}})).
		\]
	
	Hence it is enough to show that
		\[
				\underline{X(L)^{\mathrm{top}}}(S)
			\cong
				X(\underline{L}(S)).
		\]
	This is true if $X = \mathbbm{A}^{1}_{K}$.
	It follows that it is also true for any affine $X$.
	By \cite[Lemma 2.5.2]{Suz22},
	any Zariski covering of $\operatorname{Spec} \mathbf{K}(S_{\overline{k}})$ can be refined
	by a disjoint Zariski covering coming from a disjoint open covering of $S$.
	Hence the same is true for $\operatorname{Spec} \underline{L}(S)$
	by \eqref{0015}.
	Now a patching argument shows that the desired isomorphism holds
	for any $X$.
\end{proof}

\begin{prop} \mbox{} \label{0012}
	\begin{enumerate}
	\item \label{0022}
		For any complex $A$ of commutative group schemes locally of finite type over $K$,
		we have
			\[
					R \Gamma(K_{W}, A)
				\cong
					R \Gamma(\Hat{W}_{K}, \Gamma(\Hat{K}^{\mathrm{sep}}_{W}, A))
			\]
		in $D(\ast_{\mathrm{proet}})$.
	\item \label{0023}
		For any complexes $A, B, C$ of commutative group schemes locally of finite type over $K$
		and any morphism $A \mathbin{\otimes}^{L} B \to C$ in $D(K_{\mathrm{Et}})$,
		the cup product morphism
			\[
					R \Gamma(K_{W}, A) \mathbin{\otimes}^{L} R \Gamma(K_{W}, B)
				\to
					R \Gamma(K_{W}, C)
			\]
		and the composite of the cup product morphisms
			\begin{align*}
				&
							R \Gamma(\Hat{W}_{K}, \Gamma(\Hat{K}^{\mathrm{sep}}_{W}, A))
						\mathbin{\otimes}^{L}
							R \Gamma(\Hat{W}_{K}, \Gamma(\Hat{K}^{\mathrm{sep}}_{W}, B))
				\\
				&	\to
						R \Gamma(\Hat{W}_{K},
								\Gamma(\Hat{K}^{\mathrm{sep}}_{W}, A)
							\mathbin{\otimes}^{L}
								\Gamma(\Hat{K}^{\mathrm{sep}}_{W}, B)
						)
				\\
				&	\to
						R \Gamma(\Hat{W}_{K}, \Gamma(\Hat{K}^{\mathrm{sep}}_{W}, C))
			\end{align*}
		are compatible via the isomorphism in Statement \ref{0022}.
	\end{enumerate}
\end{prop}

\begin{proof}
	This follows from Propositions \ref{0008} and \ref{0007}.
\end{proof}

\begin{thm} \label{0016} \mbox{}
	\begin{enumerate}
	\item \label{0010}
		There exists a canonical morphism
			\[
					R \Gamma(\Hat{W}_{K}, \Gamma(\Hat{K}^{\mathrm{sep}}_{W}, \mathbbm{G}_{m}))
				\to
					\mathbbm{Z}[-1]
			\]
		in $D(\ast_{\mathrm{proet}})$.
	\item \label{0011}
		Let $M$ and $N$ be $1$-motives over $K$ dual to each other.
		Consider the Poincar\'e biextension morphism
		$M \mathbin{\otimes}^{L} N \to \mathbbm{G}_{m}[1]$ in $D(K_{\mathrm{Et}})$.
		Consider the induced morphism
			\[
						\Gamma(\Hat{K}^{\mathrm{sep}}_{W}, M)
					\mathbin{\otimes}^{L}
						\Gamma(\Hat{K}^{\mathrm{sep}}_{W}, N)
				\to
					\Gamma(\Hat{K}^{\mathrm{sep}}_{W}, \mathbbm{G}_{m}[1])
			\]
		in $D(B_{\Hat{W}_{K}}(\ast_{\mathrm{proet}}))$
		by Proposition \ref{0009}.
		Consider the induced morphisms
			\[
						R \Gamma(\Hat{W}_{K}, \Gamma(\Hat{K}^{\mathrm{sep}}_{W}, M))
					\mathbin{\otimes}^{L}
						R \Gamma(\Hat{W}_{K}, \Gamma(\Hat{K}^{\mathrm{sep}}_{W}, N))
				\to
					R \Gamma(\Hat{W}_{K}, \Gamma(\Hat{K}^{\mathrm{sep}}_{W}, \mathbbm{G}_{m}[1]))
				\to
					\mathbbm{Z}
			\]
		in $D(\ast_{\mathrm{proet}})$ by Statement \ref{0010}.
		The induced morphism
			\[
					R \Gamma(\Hat{W}_{K}, \Gamma(\Hat{K}^{\mathrm{sep}}_{W}, M))
				\to
					R \operatorname{\mathbf{Hom}}_{\ast_{\mathrm{proet}}} \bigl(
						R \Gamma(\Hat{W}_{K}, \Gamma(\Hat{K}^{\mathrm{sep}}_{W}, N)),
						\mathbbm{Z}
					\bigr)
			\]
		is an isomorphism.
	\end{enumerate}
\end{thm}

\begin{proof}
	This is nothing but a translation of \eqref{0014} and Theorem \ref{0013}
	using Proposition \ref{0012}.
\end{proof}

Proposition \ref{0007} says that
$\Gamma(\Hat{K}^{\mathrm{sep}}_{W}, X)$ is the same as
the condensed Weil-\'etale realization of $X$
in the sense of Definition \ref{defn:cwer}.
With the comparison result, Proposition \ref{0012},
the above theorem thus recovers Theorem \ref{thm:duality1motives1}.


\noindent
\textbf{Marco Artusa}

\noindent Institut de Recherche Math\'ematique Avanc\'ee, Universit\'e de Strasbourg, 7 rue Ren\'e Descartes, 67000 Strasbourg, France, \href{mailto:artusa@unistra.fr}{artusa@unistra.fr}.

\vspace{0.5em}\noindent
\textbf{Takashi Suzuki}

\noindent Department of Mathematics, Chuo University,	1-13-27 Kasuga, Bunkyo-ku, Tokyo 112-8551, Japan, \href{mailto:tsuzuki@gug.math.chuo-u.ac.jp}{tsuzuki@gug.math.chuo-u.ac.jp}.

\begin{thebibliography}{99}
\bibitem{Andrey}
G.~Andreychev, Pseudocoherent and Perfect Complexes and Vector Bundles on Analytic Adic Spaces. 2021, arXiv:2105.12591.

\bibitem{SGA4}
  M.~Artin, A.~Grothendieck, and J.~L. Verdier, \emph{Th\'eorie des topos et cohomologie \'etale des sch\'emas. Tome 1-2-3}.
 Lecture Notes in Math. 269, 270, 305,  Springer, Berlin, 1972/1973
 
\bibitem{Artusa}
M.~Artusa, Duality for condensed cohomology of the Weil group of a {{\(p\)}}-adic field.
 \emph{Doc. Math.} \textbf{29} (2024), no.~6, 1381--1434.
 
 \bibitem{BH}
 C.~Barwick, and P.~Haine,
Pyknotic objects, I. Basic notions.
 2019.
 
 \bibitem{proetale}
 B.~Bhatt, and P.~Scholze, The pro-\'{e}tale topology for schemes.
 \emph{Ast\'{e}risque} \textbf{369} (2015), 99--201
 
 \bibitem{blrneron}
 S.~Bosch, W.~L\"utkebohmert and M.~Raynaud, \emph{N\'eron models}. Ergebnisse der Mathematik und ihrer Grenzgebiete (3), 21, Springer, Berlin, 1990.
 
 \bibitem{LCM}
D.~Clausen, and P.~Scholze, Lectures on Condensed Mathematics. 2019,
 \url{http://www.math.uni-bonn.de/people/scholze/Condensed.pdf}
 
\bibitem{Conrad}
 B.~Conrad, Weil  and  Grothendieck  approaches  to  adelic  point.
 \emph{Enseign. Math.} \textbf{58} (2012), no.~1,2, 61--97.

\bibitem{Deligne}
P.~Deligne, Th\'eorie de Hodge. III. \emph{Inst. Hautes \'Etudes Sci. Publ. Math.} \textbf{44} (1974), 5--77.

\bibitem{GeisserMor2}
T.~Geisser, and B.~Morin, Pontryagin duality for varieties over {{\(p\)}}-adic fields.
\emph{J. Inst. Math. Jussieu} \textbf{23} (2024), no.~1, 425--462.

\bibitem{Schemata}
M.~J. Greenberg, Schemata over local rings, \emph{Ann. of Math.} \textbf{73} (1961), no.~2, 624--648. 

\bibitem{Schemata2}
M.~J. Greenberg, Schemata over local rings. II, \emph{Ann. of Math.} \textbf{78} (1963), no.~2, 256--266.

\bibitem{SGA1}
A.~Grothendieck, \emph{Rev\^etements \'etales et groupe fondamental. Fasc. II: Expos\'es 6, 8 \`a 11}, Institut des Hautes \'Etudes Scientifiques, Paris, 1963.

\bibitem{SGA7I}
  A.~Grothendieck, M.~Raynaud and D.~S. Rim, \emph{S\'eminaire de G\'eom\'etrie Alg\'ebrique Du Bois-Marie 1967--1969. Groupes de monodromie en g\'eom\'etrie alg\'ebrique (SGA 7 I). Expos\'es I, II, VI, VII, VIII, IX}.
 Lect. Notes Math. 288, Springer, Cham, 1972.

\bibitem{HarariSzamuely}
D.~Harari and T.~Szamuely, Arithmetic duality theorems for 1-motives. \emph{J. Reine Angew. Math.} \textbf{578} (2005), 93--128.

\bibitem{Hoff}
 N.~Hoffmann, and M.~Spitzweck, Homological algebra with locally compact {abelian} groups.
 \emph{Adv. Math.} \textbf{212} (2007), no.~2, 504--524.

\bibitem{Karpuk}
 D.~Karpuk, Cohomology of the {Weil} group of a {{\(p\)}}-adic field.
 \emph{J. Number Theory} \textbf{133} (2013), no.~4, 1270--1288.
 
\bibitem{Karpuk2}
 D.~Karpuk, Weil-\'etale cohomology over {{\(p\)}}-adic fields. 2012, arXiv:1111.6710.
 
  \bibitem{Licht}
  S.~Lichtenbaum, The {Weil}-{\'e}tale topology for number rings.
 \emph{Ann. Math. (2)} \textbf{170} (2009), no.~2, 657--683.
 
 \bibitem{HA}
  J.~Lurie, Higher algebra. 2017,
 \url{https://www.math.ias.edu/~lurie/ papers/HA.pdf}.


\bibitem{Mattuck}
A.~Mattuck, Abelian varieties over {{\(p\)}}-adic ground fields.
\emph{Ann. of Math.} \textbf{62} (1955), no.~2, 92--119.

\bibitem{MilneAG}
  J.~S. Milne, \emph{Algebraic groups. The theory of group schemes of finite type over a field}.
 Cambridge University Press,  Cambridge, 2017.
 
 \bibitem{ADT}
  J.~S. Milne, \emph{Arithmetic duality theorems}.
 Perspect. Math. 1,  Academic Press, Boston, MA, 1986.
 
 \bibitem{MilneEC}
 J.~S. Milne, \emph{\'Etale cohomology}.
Princeton Mathematical Series 33,  Princeton Univ. Press, Princeton, NJ, 1980.
 
\bibitem{Ray}
M.~Raynaud, 1-motifs et monodromie g\'eom\'etrique.
 \emph{Ast\'erisque} \textbf{223} (1994), 295--319.
 
  \bibitem{Schn}
  J.~P. Schneiders, Quasi-abelian categories and sheaves.
 \emph{M{\'e}m. Soc. Math. Fr., Nouv. S{\'e}r.} \textbf{76} (1998), 1--140.
 
\bibitem{Suz20a}
T.~Suzuki.
\newblock Duality for cohomology of curves with coefficients in abelian
  varieties.
\newblock {\em Nagoya Math. J.}, 240:42--149, 2020.

\bibitem{Suz}
 T.~Suzuki, Grothendieck's pairing on N\'eron component groups: Galois descent from the semistable case. \emph{Kyoto J. Math.} \textbf{60} (2020), no. 2, pp. 593–-716.
 
  \bibitem{Suz0}
  T.~Suzuki, Grothendieck's pairing on N\'eron component groups: Galois descent from the semistable case. 2014, arXiv:1410.3046v1.

\bibitem{Suz21}
T.~Suzuki.
\newblock An improvement of the duality formalism of the rational \'{e}tale
  site.
\newblock In {\em Algebraic number theory and related topics 2018}, RIMS
  K\^{o}ky\^{u}roku Bessatsu, B86, pages 287--330. Res. Inst. Math. Sci.
  (RIMS), Kyoto, 2021.
 
\bibitem{Suz22}
T.~Suzuki.
\newblock Duality for local fields and sheaves on the category of fields.
\newblock {\em Kyoto J. Math.}, 62(4):789--864, 2022.

\bibitem{WC}
J.~T. Tate, $WC$-groups over  {{\(p\)}}-adic fields, \emph{Secr\'etariat math\'ematique, Paris}, 1958.

 
\end{thebibliography}
\end{document}